\documentclass[11pt]{amsart}

\usepackage{amsthm,amssymb,amsthm,graphicx,enumitem}

\usepackage[mathlines]{lineno}

\usepackage[colorlinks=false]{hyperref}

\usepackage{cleveref}

\theoremstyle{plain}
\newtheorem{theorem}{Theorem}[section]
\newtheorem{lemma}[theorem]{Lemma}
\newtheorem{proposition}[theorem]{Proposition}
\newtheorem{corollary}[theorem]{Corollary}
\newtheorem{sublemma}{}[theorem]

\theoremstyle{definition}
\newtheorem{definition}[theorem]{Definition}

\newcommand{\mcal}[1]{\ensuremath{\mathcal{#1}}}
\newcommand{\mbb}[1]{\ensuremath{\mathbb{#1}}}

\newcommand{\M}[1]{\ensuremath{M^{(#1)}}}

\newcommand{\dash}{\nobreakdash-\hspace{0pt}}
\newcommand{\ba}{\backslash}
\newcommand{\cl}{\operatorname{cl}}
\newcommand{\si}{\operatorname{si}}
\newcommand{\core}{\operatorname{Core}}

\title{Fan-extensions in fragile matroids}

\author[Chun]{Carolyn Chun}
\address{School of Information Systems, Computing and Mathematics,
Brunel University, United Kingdom}
\email{carolyn.chun@brunel.ac.uk}

\author[Chun]{Deborah Chun}
\address{College of Engineering \& Sciences,
Institute of Technology, West Virginia University,
United States}
\email{deborah.chun@mail.wvu.edu}

\author[Mayhew]{Dillon Mayhew}
\address{School of Mathematics, Statistics and Operations Research,
Victoria University of Wellington,
New Zealand}
\email{dillon.mayhew@msor.vuw.ac.nz}

\author[Van Zwam]{Stefan H. M. van Zwam}
\address{Department of Mathematics,
Louisiana State University,
United States}
\email{svanzwam@math.lsu.edu}

\date{\today}

\begin{document}

\begin{abstract}
If \mcal{S} is a set of matroids, then the matroid $M$ is
\mcal{S}\dash fragile if, for every element $e\in E(M)$,
either $M\ba e$ or $M/e$ has no minor isomorphic to a member of
\mcal{S}.
Excluded-minor characterizations often depend,
implicitly or explicitly, on understanding classes of fragile matroids.
In certain cases, when \mcal{M} is a minor-closed class of
\mcal{S}\dash fragile matroids, and $N\in \mcal{M}$, the only
members of \mcal{M} that contain $N$ as a minor
are obtained from $N$ by increasing the length of fans.
We prove that if this is the case, then we can certify
it with a finite case-analysis.
The analysis involves examining matroids that are
at most two elements larger than $N$.
\end{abstract}

\maketitle

\section{Introduction}
\label{sect1}

Let \mcal{S} be a set of matroids.
When we say that a matroid has an
\emph{\mcal{S}\dash minor}, we mean that
it has a minor isomorphic to a member of \mcal{S}.
The matroid $M$ is \emph{\mcal{S}\dash fragile}
if, for every element $e\in E(M)$, either
$M\ba e$ or $M/e$ has no \mcal{S}\dash minor.
Note that every minor of an \mcal{S}\dash fragile
matroid is also \mcal{S}\dash fragile.
Fragility has been studied at various times under different names:
Oxley examined non-binary $\{U_{2,4}\}$\dash fragile
matroids \cite{Oxl90};
Truemper proved a constructive characterization of
a class of binary $\{F_{7},F_{7}^{*}\}$\dash fragile
matroids \cite{Tru92a}; and Kingan and Lemos
have made a study of binary
$\{F_{7},F_{7}^{*},M^{*}(K_{3,3}),M^{*}(K_{5})\}$\dash fragile
matroids \cite{KL02,KL12}.

Our study of fragile matroids is motivated by
the goal of finding new excluded-minor
characterizations.
The matroid $S$ is a \emph{strong stabilizer} for the
partial field \mbb{P}, if, roughly speaking,
every \mbb{P}\dash representation of $S$ extends
uniquely to a \mbb{P}\dash representation of
any \mbb{P}\dash representable matroid that contains
$S$ as a minor.
More information on strong stabilizers can be found
in \cite{GOVW98} or \cite{PZ10b}.
Understanding
\mcal{S}\dash fragile matroids,
where \mcal{S} is a set of strong stabilizers,
has been important in
excluded-minor characterizations.
For example, $U_{2,4}$ is a strong stabilizer for
both $\mathrm{GF}(4)$ and the near-regular partial field.
The excluded-minor characterizations
of $\mathrm{GF}(4)$\dash representable \cite{GGK00} and
near-regular matroids \cite{HMZ11} both implicitly
use the fact that a non-binary $3$\dash connected
$\mathrm{GF}(4)$\dash representable
matroid is $\{U_{2,4}\}$\dash fragile if and only if it is a whirl.
Geelen, Gerards, and Whittle
conjecture that, for any prime power $q$, and any
$\mathrm{GF}(q)$\dash representable matroid $N$,
there is an integer $k$ such that every
$\mathrm{GF}(q)$\dash representable
$\{N\}$\dash fragile matroid has branch width
at most $k$ \cite[Conjecture~5.9]{GGW07}.
In another application of the link between fragility and
excluded-minor results, Mayhew, Van Zwam, and Whittle \cite{MWZ12}
have shown that this conjecture implies that there are
only finitely many excluded minors for
$\mathrm{GF}(5)$\dash representability.

In our study of \mcal{S}\dash fragile matroids,
\mcal{S} will typically be a set of excluded minors for
representability over some partial field.
This allows us to assume certain properties of \mcal{S}.
In particular, we can assume that the members of \mcal{S} are
$3$\dash connected and contain at least four elements.
We say that the matroid $N$ is
\emph{$3$\dash connected up to series and parallel sets}
if it is connected, and
$\min\{r(X),r(Y)\}=1$ or
$\min\{r^{*}(X),r^{*}(Y)\}=1$ for every
$2$\dash separation $(X,Y)$ of $N$.
The next result follows immediately from
\cite[Proposition~3.1]{KL02}.

\begin{proposition}
\label{block}
Let \mcal{S} be a collection of $3$\dash connected matroids,
each of which has at least four elements.
Assume that $N$ is \mcal{S}\dash fragile and that $N$ has an
\mcal{S}\dash minor.
Then $N$ is $3$\dash connected up to series and parallel sets.
\end{proposition}

It is a fairly easy exercise to show that wheels and
whirls are representable over every partial field, except that
whirls are not representable over $\mathrm{GF}(2)$ and the
regular partial field.
Therefore we will henceforth assume that \mcal{S} contains no
wheels or whirls.
This implies that any matroid with an \mcal{S}\dash minor is
neither a wheel nor a whirl.

A \emph{fan} is a sequence of elements, where consecutive
sets of three elements alternately form triads and triangles.
In some cases, the only way to construct \mbb{P}\dash representable
\mcal{S}\dash fragile matroids by building from a matroid $N$
is to increase the length of fans in $N$.
For example, \Cref{magic}(i) shows the rank\dash $6$
binary matroid $N_{12}$.
This matroid is obtained by gluing three copies of
$M(K_{4})$ to $F_{7}$ along three lines that contain a common
point $p$, and then deleting the points of intersection,
apart from three that lie in a common line avoiding $p$.
The matroid in \Cref{magic}(ii) is obtained from
$N_{12}$ by lengthening the fan $(u_{1},u_{2},u_{3},u_{4})$
to $(u_{1},a,b,u_{2},u_{3},u_{4})$, and by lengthening
$(v_{1},v_{2},v_{3},v_{4})$ to
$(v_{1},v_{2},v_{3},v_{4},c)$.
Any matroid obtained by lengthening fans in this way
is a \emph{fan-extension} of $N_{12}$.
(We delay the formal definition of fan-extensions until \Cref{sect3}.)
Any $3$\dash connected binary matroid that
is $\{F_{7},F_{7}^{*}\}$\dash fragile
and contains $N_{12}$ as a minor is obtained
from $N_{12}$ by lengthening the three
disjoint $4$\dash element fans.
The resulting family of matroids is essentially the same
as the family $\mcal{F}_{1}(m,n,r)$, as described in \cite{KL02}.
This result, and other applications of our main theorem,
will be described in \Cref{sect6}.
These applications bring the excluded-minor characterisations
for matroids representable over the partial fields $\mathbb{H}_{5}$
and $\mathbb{U}_{2}$ within reach.
\begin{figure}[htb]
\centering
\includegraphics{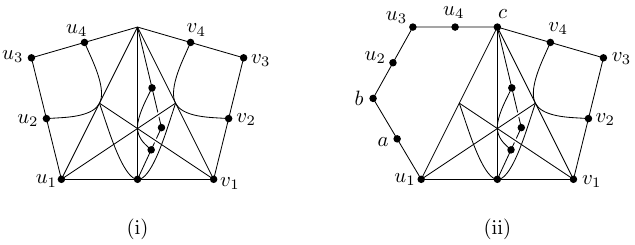}
\caption{$N_{12}$, and one of its fan-extensions.}
\label{magic}
\end{figure}

Suppose we are given a minor-closed class, \mcal{M}, and a matroid
$N\in \mcal{M}$.
We would like to know whether $M\in \mcal{M}$ being $3$\dash connected
with $N$ as a minor implies that $M$ is a fan-extension of $N$.
Our main theorem allows us to use a finite case-analysis to check
whether this implication is true.
From now on we make no mention of fragility.
\Cref{stage} instead uses the
conditions implied by \Cref{block}, and the assumption that
\mcal{S} does not contain any wheels or whirls.

\begin{theorem}
\label{stage}
Let \mcal{M} be a set of matroids closed under
isomorphism and minors.
Let $N\in \mcal{M}$ be a $3$\dash connected matroid
such that $|E(N)|\geq 4$ and $N$ is neither a wheel
nor a whirl.
Assume that any member of \mcal{M}
with $N$ as a minor is $3$\dash connected
up to series and parallel sets.
If there is a $3$\dash connected matroid in \mcal{M} with $N$
as a minor that is not a fan-extension of $N$, then
there exists such a matroid, $M$, satisfying $|E(M)|-|E(N)|\leq 2$.
\end{theorem}

This is restated as \Cref{arrow} later in the paper.
The assumptions on $\mcal{M}$ and $N$ are justified if
$\mcal{M}$ is a minor-closed class of \mcal{S}\dash fragile
matroids, and $N$ is a member of \mcal{M} with an
\mcal{S}\dash minor, where we make the assumption that the
members of \mcal{S} are $3$\dash connected with at least
four elements, and \mcal{S} contains no wheel or whirl.

To see that the bound $|E(M)|-|E(N)|\leq 2$ is best possible,
we let $G$ be one of the graphs drawn schematically in
\Cref{beach}.
If $G$ is the graph on the left, let $N=M(G)\ba x\ba y$.
If $G$ is the graph on the right, let $N=M(G)/x\ba y$.
In either case, let $M=M(G)$.
Now set \mcal{M} to be the class containing all minors
of $M$ and their isomorphs.
We note that $M$ is not a fan-extension of $N$ relative
to the fans $(u_{1},u_{2},u_{3})$ and $(v_{1},v_{2},v_{3})$.
However, any member of \mcal{M} that has $N$ as a minor
and that is at most one element larger than $N$ is a fan-extension.
\begin{figure}[htb]
\centering
\includegraphics{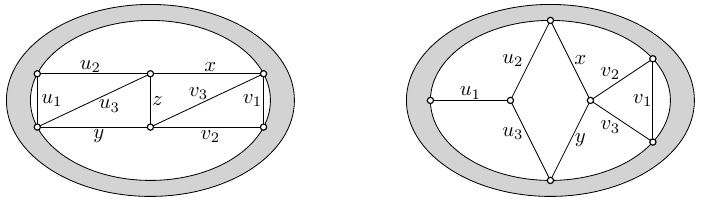}
\caption{Schematic drawings of two graphs.}
\label{beach}
\end{figure}

In \Cref{sect3} we state some definitions, including that of a fan-extension.
\Cref{sect5} is dedicated to an alternative formulation of fan-extensions,
based upon the idea of gluing wheels to a `core' matroid.
We use this alternative formulation in \Cref{sect6}, where we sketch some
applications of our theorem to binary $\{F_{7},F_{7}^{*}\}$\dash fragile
matroids and $\{U_{2,5},U_{3,5}\}$\dash fragile matroids that are
representable over the partial field $\mathbb{H}_{5}$.
The proof of the main theorem is contained in \Cref{sect2,sect4}.

\section{Fan-extensions}
\label{sect3}

Any unexplained terminology or notation that we use can be found in
Oxley \cite{Oxl11}.

\begin{definition}
\label{ebony}
Let $M$ be a matroid.
A \emph{fan} of $M$ is an
ordered sequence, $(e_{1},\ldots, e_{n})$,
of $n\geq 3$ distinct elements such that
\[
\{e_{1},e_{2},e_{3}\},
\{e_{2},e_{3},e_{4}\},\ldots,
\{e_{n-2},e_{n-1},e_{n}\}
\]
is an alternating sequence of triangles and triads.
\end{definition}

If $\{e_{1},e_{2},e_{3}\}$ is a triangle, then
the elements with odd indices in $(e_{1},\ldots, e_{n})$ are \emph{spoke}
elements, and elements with even indices are \emph{rim} elements.
These labels are reversed if $\{e_{1},e_{2},e_{3}\}$ is a triad.
We sometimes blur the distinction between ordered and unordered
sets where doing so creates no problems.
So for example, we may talk about fans being disjoint.
Note that $(e_{1},\ldots, e_{n})$ is a fan of $M$ if and only
if it is a fan of $M^{*}$.
If $1 < i < n$, then $e_{i}$ is an \emph{internal} element
of the fan, otherwise it is a \emph{terminal} element.
In a $3$\dash connected matroid with at least five elements, no
triangle can be a triad, so the partitioning into spoke and rim
elements is unambiguous.
In particular, in a $3$\dash connected matroid with
at least four elements that is not a whirl,
an element in a fan is either a spoke element or a rim element,
and not both.
We frequently replace the ordering $(e_{1},\ldots, e_{n})$
with $(e_{n},\ldots, e_{1})$.
We call this process \emph{reversing}.
A \emph{contiguous subsequence} of
$(e_{1},\ldots, e_{n})$ is a subsequence of the form
$(e_{s},e_{s+1},\ldots, e_{t-1},e_{t})$, where $1\leq s \leq t \leq n$.
If  $F=(e_{1},\ldots, e_{n})$ and $F'=(e_{1}',\ldots, e_{m}')$
are two sequences, then
we say that $F$ is \emph{consistent} with $F'$ if
$(e_{1},\ldots, e_{n})$ is a subsequence (not necessarily contiguous)
of either $(e_{1}',\ldots, e_{m}')$ or $(e_{m}',\ldots, e_{1}')$.
If $F=(e_{1},\ldots,e_{n})$ is an ordered sequence of
elements, and $X\subseteq \{e_{1},\ldots, e_{n}\}$, then
$F-X$ is the subsequence produced from $F$ by omitting all elements in $X$.
As usual, we abbreviate the singleton set $\{x\}$ to $x$.
We say that the fan $F=(e_{1},\ldots, e_{n})$ is \emph{maximal} if there is
no fan $(e_{1}',\ldots, e_{m}')$ such that
$\{e_{1}',\ldots, e_{m}'\}$ properly contains
$\{e_{1},\ldots, e_{n}\}$.

Now we formally define fan-extensions.
To avoid disrupting the exposition, we will relegate some technical lemmas
until later sections.
Let $M$ be a $3$\dash connected matroid with a fan
$(e_{1},\ldots, e_{n})$, where $n\geq 4$.
If $e_{1}$ is a spoke element, and $M\ba e_{1}$ is $3$\dash connected,
then clearly $(e_{2},\ldots, e_{n})$ is a fan of
$M\ba e_{1}$, and $M$ is said to be obtained from $M\ba e_{1}$
by a \emph{fan-lengthening} move on this fan.
Similarly, if $e_{1}$ is a rim element, and $M/e_{1}$ is
$3$\dash connected, then $M$ is obtained from $M/e_{1}$ by
a fan-lengthening move on $(e_{2},\ldots, e_{n})$.
If $n\geq 5$, and $e_{i}$ is a rim element, where $1\leq i\leq n-1$, and
$M/e_{i}\ba e_{i+1}$ is $3$\dash connected, then
\Cref{cider} shows $(e_{1},\ldots, e_{i-1},e_{i+2},\ldots,e_{n})$
is a fan of $M/e_{i}\ba e_{i+1}$, and $M$ is said to be obtained from
$M/e_{i}\ba e_{i+1}$ by a fan-lengthening move on this fan.
Note that $M$ is obtained by a fan-lengthening
move on $M'$ if and only if $M^{*}$ is obtained
by a fan-lengthening move on $(M')^{*}$ (applied to the
reversed fan, in the case that $|E(M)|=|E(M')|+2$).
Moreover, $M$ is necessarily $3$\dash connected.

Let $N$ be a $3$\dash connected matroid with at least
four elements.
Let $\mcal{F}_{N}$ be a collection of pairwise disjoint
fans in $N$.
Note that we do not require fans in $\mcal{F}_{N}$ to be maximal.
A fan in $N$ could potentially contain many fans in
$\mcal{F}_{N}$ as subsequences.
If $M'$ has $N$ as a minor, and $\mcal{F}'$ is a family of fans
of $M'$, then we say that $\mcal{F}'$ is a \emph{covering family}
of $M'$ (relative to $N$ and $\mcal{F}_{N}$) if the following
conditions are satisfied:
\begin{enumerate}[label=\textup{(\roman*)}]
\item the fans in $\mcal{F}'$ are pairwise disjoint,
\item $|\mcal{F}'|=|\mcal{F}_{N}|$,
\item for every $F_{N}\in \mcal{F}_{N}$, there is a
fan $F'\in \mcal{F}'$ such that $F_{N}$ is consistent with $F'$,
\item every element in $E(M')-E(N)$ is contained
in one of the fans in $\mcal{F}'$.
\end{enumerate}
Observe that the fan $F'$ in condition~(iii) may contain elements
of $E(N)$ that are not in $F_{N}$.
If we reverse any fan in a covering family, we obtain another
covering family.
Informally, a fan-extension of $N$ is obtained by finding a covering
family, applying a fan-lengthening move to one of the fans in that
family, and then reiterating this process.
More formally, we have:

\begin{definition}
\label{pinch}
We define $N$ to be a fan-extension of $N$.
We recursively define the set of fan-extensions of $N$
(relative to $\mcal{F}_{N}$) to be the smallest set satisfying
the following condition:
\begin{itemize}
\item if $M'$ is a fan-extension of $N$, and $\mcal{F}'$ is a
covering family of $M'$ containing the fan $F'$, then any matroid
obtained from $M'$ by a fan-lengthening move on $F'$ is a
fan-extension of $N$ (relative to $\mcal{F}_{N}$).
\end{itemize}
\end{definition}

If $M$ is obtained from $M'$ by lengthening $F'$ into the
fan $F$, then $(\mcal{F}'-\{F'\})\cup \{F\}$ is a covering family
of $M$.
Therefore the next result follows easily from the definition.

\begin{proposition}
\label{bunch}
Let $N$ be a $3$\dash connected matroid with at least
four elements and let $\mcal{F}_{N}$ be a collection of pairwise
disjoint fans in $N$.
If $M$ is a fan-extension of $N$, then $M$ is $3$\dash connected,
has $N$ as a minor, and has a covering family.
\end{proposition}

Note that $M$ is a fan-extension of $N$ relative to $\mcal{F}_{N}$
if and only if $M^{*}$ is a fan-extension of $N^{*}$ relative to
$\mcal{F}_{N}$.
The converse of \Cref{bunch} need not hold.
In \Cref{fever}, we essentially construct a $3$\dash connected
matroid that is not a fan-extension, although it does have $N$ as a minor
and a covering family.
However, we will have occasion to use the
partial converse in \Cref{court}.
If \mcal{F} is a covering family of $M$, then we say that \mcal{F} admits
a \emph{fan-shortening move} if $M$ is obtained from $M'$ by using a
fan-lengthening move on $F'\in \mcal{F}'$ to produce the fan $F$, where
$F\in \mcal{F}$ and $M'$ has $N$ as a minor.

\begin{proposition}
\label{burst}
Assume $N$ is neither a wheel nor a whirl, and
that no fan in $N$ contains two distinct fans in $\mcal{F}_{N}$
(considered as unordered sets).
Let $M$ be a $3$\dash connected matroid with $N$ as a minor,
and assume that every minor of $M$ that has $N$ as a minor is $3$\dash connected
up to series and parallel sets.
Let \mcal{F} be a covering family of $M$.
If $M\ne N$, then \mcal{F} admits a fan-shortening move.
\end{proposition}

\begin{proof}
Let $(e_{1},\ldots, e_{n})$ be an arbitrary fan in \mcal{F}, and assume that
the internal element $e_{i}$ belongs to $E(M)-E(N)$.
By duality, we can assume that $e_{i}$ is a rim element.
If $N$ is a minor of $M\ba e_{i}$, then $n\leq 4$, for otherwise
$M\ba e_{i}$ contains a triangle that contains a series pair, and this
contradicts the hypotheses of the \namecref{burst}
(see \Cref{lasso}).
But if $n\leq 4$, the three elements in $\{e_{1},e_{2},e_{3},e_{4}\}-e_{i}$
all belong to a fan in $\mcal{F}_{N}$, and as this set contains a series pair in
$M\ba e_{i}$, it follows that $N$ is not cosimple, a contradiction.
Therefore $N$ is a minor of $M/e_{i}$.
As $\{e_{i-1},e_{i+1}\}$ is a parallel pair in $M/e_{i}$, we can reverse
$(e_{1},\ldots, e_{n})$ as necessary, and assume that $N$ is a minor of
$M/e_{i}\ba e_{i+1}$.
Thus $n\geq 5$, as $(e_{1},\ldots, e_{n})$ must contain three elements from
$E(N)$.
\Cref{shrub} implies that
$M/e_{i}\ba e_{i+1}$ is $3$\dash connected, and we can
set $M'$ to be $M/e_{i}\ba e_{i+1}$.
Now \Cref{cider} implies $(e_{1},\ldots, e_{i-1},e_{i+2},\ldots, e_{n})$ is a fan
of $M'$.
As $M$ is obtained from $M'$ by performing a fan-lengthening move on
this fan, we are done.
Therefore we will now assume that the internal elements of fans in \mcal{F}
all belong to $E(N)$.

As $M\ne N$, we can reverse as necessary, and
let $(e_{1},\ldots, e_{n})$ be a fan in \mcal{F} where
$e_{1}$ belongs to $E(M)-E(N)$.
By duality, we can assume that $e_{1}$ is a spoke element.
If $M\ba e_{1}$ is $3$\dash connected, then we set $M'$ to
be $M\ba e_{1}$ and we are done.
Therefore we assume that $M\ba e_{1}$ is not $3$\dash connected, and thus
contains a series pair.
Now $e_{1}$ is contained in a triad, $T^{*}$, of $M$.
Orthogonality with the triangle $\{e_{1},e_{2},e_{3}\}$ means that
$T^{*}$ contains $e_{2}$ or $e_{3}$.
Let $x$ be the element in $T^{*}-\{e_{1},e_{2},e_{3}\}$.
Because $(e_{1},\ldots, e_{n})$ contains at least three elements of $E(N)$,
$n\geq 4$, so $e_{2}$ and $e_{3}$ are internal elements,
and hence belong to $E(N)$.
As either $\{x,e_{2}\}$ or $\{x,e_{3}\}$ is a series pair in $M\ba e_{1}$,
we see $N$ is a minor of $M\ba e_{1}/x$.
Because \mcal{F} is a covering family, $x$ is contained in a
fan in \mcal{F}.
Because $x$ is not in $E(N)$, it is not an internal element,
so $x$ is a terminal element of a fan in \mcal{F}.
Orthogonality with $T^{*}$ shows that it is a rim element.
It cannot be the case that $x$ is in $(e_{1},\ldots, e_{n})$, for then
$x=e_{n}$, so this fan would contain a triad that does not consist of three
consecutive elements.
The dual of \Cref{steal} shows that this is a contradiction.
Assume that $x=f_{m}$, where $(f_{1},\ldots, f_{m})$ is a fan in
\mcal{F}.
If $M/f_{m}$ is $3$\dash connected we are done, so we assume
$f_{m}$ is in a triangle, $T$.
This triangle must contain either $f_{m-2}$ or $f_{m-1}$, and an element from
$T^{*}$.
Orthogonality with the triad $\{e_{2},e_{3},e_{4}\}$ shows that $T$
cannot contain $e_{2}$ or $e_{3}$, so it contains $e_{1}$.

If $T=\{f_{m-1},f_{m},e_{1}\}$ and $T^{*}=\{f_{m},e_{1},e_{2}\}$, then
$(f_{1},\ldots, f_{m},e_{1},\ldots, e_{n})$ is a fan of $M$ that contains two
fans in $\mcal{F}_{N}$.
Let $(x_{1},\ldots, x_{p})$ be the subsequence obtained from
$(f_{1},\ldots, f_{m},e_{1},\ldots, e_{n})$ by omitting the elements not in
$E(N)$.
(Any such elements have to be in $\{f_{1},f_{m},e_{1},e_{n}\}$.)
It follows from  \cite[Corollary~8.2.5]{Oxl11} that any three consecutive elements
in $(x_{1},\ldots, x_{p})$ are $3$\dash separating in $N$,
and therefore form either a triangle or a triad in $N$.
It is not too difficult to see, using orthogonality, that either
$(x_{1},\ldots, x_{p})$ is a fan in $N$, or
$\min\{r_{N}(\{x_{1},\ldots, x_{p}\}),r_{N}^{*}(\{x_{1},\ldots, x_{p}\})\}\leq 2$.
In the former case, $N$ has a fan that contains two distinct fans
from $\mcal{F}_{N}$, so we have a contradiction to the hypotheses.
Therefore $\{x_{1},\ldots, x_{p}\}$ is a line in either $N$ or $N^{*}$.
However, $\{e_{2},e_{3},e_{4}\}$ is a triad of $M$, and of $N$ also
(since $\{e_{2},\ldots, e_{n-1}\}\subseteq E(N)$ and $\{e_{2},\ldots, e_{n}\}$
contains at least three elements of $E(N)$).
Similarly, $\{f_{m-3},f_{m-2},f_{m-1}\}$ is a triangle of $N$, so
$\{f_{m-3},f_{m-2},f_{m-1},e_{2},e_{3},e_{4}\}$ cannot be contained in a line
of $N$ or $N^{*}$.

This disposes of the case that $T=\{f_{m-1},f_{m},e_{1}\}$ and
$T^{*}=\{f_{m},e_{1},e_{2}\}$.
Next we assume that $T\ne \{f_{m-1},f_{m},e_{1}\}$.
A symmetrical argument will then deal with the case that
$T^{*}\ne \{f_{m},e_{1},e_{2}\}$.
Because $T$ is not $\{f_{m-1},f_{m},e_{1}\}$, it
is $\{f_{m-2},f_{m},e_{1}\}$ instead.
This means that $m=4$, for otherwise we violate orthogonality between
$T$ and $\{f_{m-4},f_{m-3},f_{m-2}\}$.
If $T^{*}=\{f_{m},e_{1},e_{2}\}$, then
$(f_{1},f_{3},f_{2},f_{4},e_{1},e_{2},\ldots, e_{n})$ is a fan of $M$ that
contains two fans of $\mcal{F}_{N}$.
If $T^{*}$ is not $\{f_{m},e_{1},e_{2}\}$, then it is
$\{f_{m},e_{1},e_{3}\}$, which implies that $n=4$.
In this case $(f_{1},f_{3},f_{2},f_{4},e_{1},e_{3},e_{2},e_{4})$ is a fan of
$M$ that contains two fans in $\mcal{F}_{N}$.
In either case, we can obtain a contradiction to the hypotheses of the
\namecref{burst}, exactly as in the previous paragraph.
This completes the proof.
\end{proof}

\begin{corollary}
\label{court}
Let $M$ and $N$ be as described in \textup{\Cref{burst}}.
If $M$ contains a covering family, then it is a fan-extension
of $N$ relative to $\mcal{F}_{N}$.
\end{corollary}

\begin{proof}
The proof is by induction on $|E(M)|-|E(N)|$.
If $M=N$, then $M$ is a fan-extension of $N$, as desired.
Therefore we assume $M\ne N$.
Let \mcal{F} be a covering family of $M$.
By \Cref{burst}, $M$ is obtained from some matroid $M'$ by a performing a
fan-lengthening move on $F'$ to obtain $F\in \mcal{F}$.
As $(\mcal{F}-\{F\})\cup\{F'\}$ is a covering family of $M'$, it follows 
by induction that $M'$ is a fan-extension of $N$.
Now the result is immediate.
\end{proof}

\section{A wheel-gluing lemma}
\label{sect5}

In this section we develop an alternative description of fan-extensions
that will be of use in \Cref{sect6}, where we describe some applications
of our main theorem.
Intuitively, a family of fans in $N$ can be seen as the result of gluing wheels
along a family of triangles in a matroid that is smaller than $N$.
A fan-extension of $N$ can be obtained in the same way, except that we glue
on larger wheels.
The object of this section is to make these ideas formal.
Our focus here will be on the case that $N$ is representable, although
it would be possible to extend these results to arbitrary matroids.

\begin{proposition}
\label{peace}
Let $M_{1}$ and $M_{2}$ be matroids on the same ground set.
Assume $(e_{1},e_{2},e_{3}, e_{4})$ is a fan with $e_{1}$ as a spoke
element in both $M_{1}$ and $M_{2}$.
If $M_{1}\ba e_{1}=M_{2}\ba e_{1}$, then $M_{1}=M_{2}$.
\end{proposition}

\begin{proof}
Assume that $M_{1}\ne M_{2}$.
Without loss of generality, we can assume that $X$ is a circuit in $M_{1}$,
but an independent set in $M_{2}$.
Then $e_{1}$ must be contained in $X$.
As $M_{1}\ba e_{1}/e_{2}=M_{2}\ba e_{1}/e_{2}$, and $M_{i}/ e_{2}$ is obtained
from $M_{i}\ba e_{1}/e_{2}$ by adding the element $e_{1}$ parallel to $e_{3}$
(for $i=1,2$), it follows that $M_{1}/e_{2}=M_{2}/e_{2}$.
Therefore $e_{2}$ is not in $X$.
Now $M_{1}/e_{2}\ba e_{3}=M_{2}/e_{2}\ba e_{3}$, and
$M_{i}\ba e_{3}$ is obtained from $M_{i}/e_{2}\ba e_{3}$
by adding $e_{2}$ in series to $e_{4}$.
Therefore $M_{1}\ba e_{3}=M_{2}\ba e_{3}$, so $e_{3}\in X$.
By strong circuit-exchange between $X$ and $\{e_{1},e_{2},e_{3}\}$,
there is a circuit, $C$, in $M_{1}$ contained in $(X-e_{1})\cup e_{2}$ that contains
$e_{2}$.
As $C$ does not contain $e_{1}$, it is also a circuit of $M_{2}$.
Now $C$ and $\{e_{1},e_{2},e_{3}\}$ are distinct circuits of $M_{2}$, both of which
are contained in $X\cup e_{2}$ and both of which contain $e_{2}$.
As $X$ is independent in $M_{2}$, this is a contradiction.
\end{proof}

Let $M_{1}$ and $M_{2}$ be matroids with $E(M_{1})\cap E(M_{2})=T$.
The \emph{generalised parallel connection} of $M_{1}$ and $M_{2}$
is defined if $M_{1}|T=M_{2}|T$, and $T$ is a modular flat of $M_{2}$.
In this case, we use the notation $M_{1}\boxtimes_{T} M_{2}$ to denote
the generalized parallel connection.
Note that a triangle in a simple binary matroid is a modular flat.
The flats of $M_{1}\boxtimes_{T} M_{2}$ are precisely the subsets
$F\subseteq E(M_{1})\cup E(M_{2})$ such that
$F\cap E(M_{i})$ is a flat of $M_{i}$, for $i=1,2$
(see \cite[Proposition~11.4.13]{Oxl11}).
Assume $M_{1}$, $M_{2}$, and $M_{3}$ are matroids with
$E(M_{2})\cap E(M_{3})\subseteq E(M_{1})$.
Let $T_{i}=E(M_{1})\cap E(M_{i})$ for $i=2,3$.
If $M_{1}\boxtimes_{T_{2}}M_{2}$
and $M_{1}\boxtimes_{T_{3}}M_{3}$ are both defined, then
it follows easily from the definition that
$(M_{1}\boxtimes_{T_{2}}M_{2})\boxtimes_{T_{3}}M_{3}$ and
$(M_{1}\boxtimes_{T_{3}}M_{3})\boxtimes_{T_{2}}M_{2}$ are defined and equal.

Let $N_{0}$ be a matroid and let $\mcal{T}=\{T_{i}\}_{i\in I}$ be a multiset of
triangles of $N_{0}$, indexed by the set $I=\{1,\ldots, t\}$.
Note that triangles in \mcal{T} need not be disjoint, nor indeed distinct.
These are the triangles to which we will glue wheels.
For each $i\in I$, let $T_{i}=\{a_{i},b_{i},c_{i}\}$.
The end points of the fan that we generate by the gluing operation will
be $a_{i}$ and $c_{i}$.
Let $r$ be a function from $I$ to $\{2,3,4,5,\ldots\}$.
This function determines the rank of the wheels that we glue to the triangles
in \mcal{T}.
Finally, let $X$ be a subset of $\cup_{k\in I}T_{k}$ such that,
for all $i\in I$, $b_{i}\notin X$ implies $b_{i}=a_{j}$ or $b_{i}=c_{j}$
for some $j\in I$.
After gluing the wheels to the triangles in \mcal{T}, we delete the
set $X$.
We call the tuple $(N_{0},\mcal{T},r,X)$ a \emph{blueprint}.

Assume that $(N_{0},\mcal{T},r,X)$ is a blueprint.
For each $i\in I$, we let the matroid $W_{i}$ be a copy of a wheel
with rank $r(i)$.
The ground set of $W_{i}$  will be
$\{x_{1}^{i},\ldots, x_{r(i)}^{i}, y_{1}^{i},\ldots, y_{r(i)}^{i}\}$,
where we make the identifications
$x_{1}^{i}=a_{i}$, $y_{r(i)}^{i}=b_{i}$, and $x_{r(i)}^{i}=c_{i}$.
The ground set of $W_{i}$ is labeled in such a way that
$(x_{1}^{i},y_{1}^{i},x_{2}^{i},y_{2}^{i},\ldots, x_{r(i)}^{i},y_{r(i)}^{i})$
is a fan with $x_{1}^{i}$ as a spoke element.
Moreover $\{x_{1}^{i},y_{r(i)}^{i},x_{r(i)}^{i}\}$ is a triangle.
Thus $E(W_{i})\cap E(W_{j})\subseteq \cup_{k\in I}T_{k}$
when $i\ne j$.
For each $i\in I=\{1,\ldots, t\}$, we recursively define $N_{i}$ to be
$N_{i-1}\boxtimes_{T_{i}} W_{i}$.
By an earlier observation, the ordering of the indices in
$I$ makes no difference to the definition of $N_{t}$.
We define $\boxtimes(N_{0},\mcal{T},r,X)$ to be $N_{t}\ba X$,
and we say that this matroid is obtained by \emph{gluing wheels} to
$N_{0}$.
From the definition of generalized parallel connection, it is straightforward to show
that
$(x_{1}^{i},y_{1}^{i},x_{2}^{i},\ldots, y_{r(i)-1}^{i},x_{r(i)}^{i})-X$ is a fan
in $\boxtimes(N_{0},\mcal{T},r,X)$, for every $i\in I$.
We call such a fan a \emph{canonical fan} of $\boxtimes(N_{0},\mcal{T},r,X)$.

Henceforth we take $N$ to be a $3$\dash connected representable matroid
that is neither a wheel nor a whirl.
Let $E$ be the ground of $N$.
We assume that $|E|\geq 4$, so that $N$ is simple.
If $X$ and $Y$ are disjoint subsets of
$E(N)$, then $\sqcap(X,Y)=r(X)+r(Y)-r(X\cup Y)$.
Given an embedding of $N$ in a projective geometry, the parameter $\sqcap(X,Y)$
tells us the rank of the maximal subspace spanned by both $X$ and $Y$.
The next two results are standard, and not difficult to prove.
We omit their proofs.

\begin{proposition}
\label{smash}
Let $F=(e_{1},\ldots, e_{n})$ be a fan of $N$.
If $e_{1}$ is a rim element, then $\sqcap(\{e_{1},e_{2}\},E-F)=1$.
\end{proposition}

\begin{proposition}
\label{cough}
Let $F=(e_{1},\ldots, e_{n})$ be a fan of $N$, and let $R$ be the
set of rim elements in $F$.
Then $\sqcap(R,E-F)=1$.
\end{proposition}

Recall that we have required $N$ to be representable, so henceforth we
identify elements of $E$ with points in a projective geometry, $P$.
Let $\mcal{F}=\{F_{i}\}_{i\in I}$ be a family of pairwise disjoint
fans in $N$ indexed by the set $I=\{1,\ldots, t\}$.
As $N$ is neither a wheel nor a whirl, \Cref{hocus} implies that
the complement of any fan $F_{i}$ contains at least two elements.
We are going to apply \Cref{smash,cough} to each fan in \mcal{F} and its
reversal.
This identifies three points in $P$ that are distinguished by the fan.
We add three elements in parallel to these points.
To be more precise, we let $F_{i}=(e_{1},\ldots,e_{n})$ be a fan in \mcal{F}.
If $e_{1}$ is a spoke element, then we add $a_{i}$ in parallel to $e_{1}$.
If $e_{1}$ is a rim element, then we add $a_{i}$ in parallel to the single point in $P$
that is in the closure of $E-F_{i}$ and $\{e_{1},e_{2}\}$.
Note that if $e_{1}$ is a spoke element, and $e_{1}$ is not in the closure of $E-F_{i}$,
then it is in the closure and the coclosure of $\{e_{2},\ldots, e_{n}\}$.
This implies $\lambda(F_{i})\leq 1$, which contradicts the $3$\dash connectivity of
$N$.
Therefore $a_{i}$ is in the closure of $E-F_{i}$ in any case.
Similarly, we add $c_{i}$ in parallel to $e_{n}$ if $e_{n}$ is a spoke element,
and otherwise we add $c_{i}$ in parallel to the point of $P$ that lies
in the closure of both $E-F_{i}$ and $\{e_{n-1},e_{n}\}$.
Let $R$ be the set of rim elements in $F_{i}$.
Add $b_{i}$ in parallel to the point of $P$ that is in the closure of both $R$
and $E-F_{i}$.
\Cref{floor} implies that $\lambda(F_{i})=2$, so the maximal subspace
of $P$ that lies in the closure of $F_{i}$ and $E-F_{i}$ has rank two.
The points $a_{i}$, $b_{i}$, and $c_{i}$ are all parallel to points in this
subspace.

Define $N^{+}$ to be the matroid consisting of the points in
$E(N)$ and $\cup_{i\in I}\{a_{i},b_{i},c_{i}\}$.
(We have defined $N^{+}$ relative to a representation of $N$.
In fact, any two representations of $N$ will lead to the same matroid $N^{+}$,
but we will not make use of this fact.)
We define $\core(N)$ to be the matroid obtained from $N^{+}$
by deleting $\cup_{i\in I}F_{i}$ and any other point of $E(N)$ that
is parallel to a point $a_{i}$ or $c_{i}$ for some $i\in I$.
\Cref{print} shows schematic representations of the matroids $N$,
$N^{+}$, and $\core(N)$.
\Cref{drunk} shows that $\mcal{T}=\{\{a_{i},b_{i},c_{i}\}\}_{i\in I}$ is a family of
pairwise disjoint triangles in $\core(N)$.
We say that $\core(N)$ and \mcal{T} are \emph{determined} by \mcal{F}.

\begin{figure}[htb]
\centering
\includegraphics{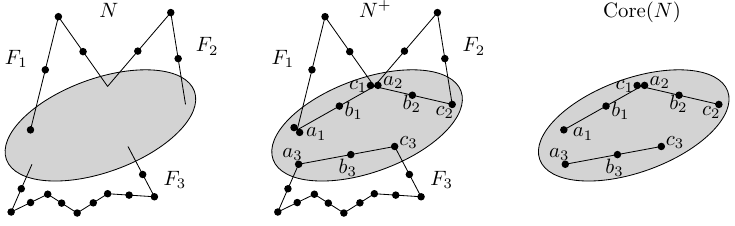}
\caption{Geometric illustrations of $N$, $N^{+}$, and $\core(N)$.}
\label{print}
\end{figure}

If $F_{i}=(e_{1},\dots, e_{n})$ is a fan in \mcal{F}, then
we define $F_{i}^{+}$ to be the sequence $(x_{1},\ldots,x_{m})$,
formed from $(e_{1},\ldots,e_{n})$ by prepending $a_{i}$ if $e_{1}$ is a rim
element, and appending $c_{i}$ if $e_{n}$ is a rim element.
Note that if $e_{1}$ is a rim element, then orthogonality and our
choice of $a_{i}$ mean that $a_{i}$, $e_{1}$, and $e_{2}$ are
pairwise distinct.
Thus $\{x_{1},x_{2},x_{3}\}$ is a triangle of $N^{+}$ in any case.
Similarly, $\{x_{m-2},x_{m-1},x_{m}\}$ is a triangle of $N^{+}$.
Any triangle of $N$ contained in $F_{i}$ is also a triangle in $N^{+}$.
If $T^{*}$ is a triad of $N$ contained in $F_{i}$, then every point in
$\cup_{i\in I}\{a_{i},b_{i},c_{i}\}$ is in the closure of the complement of $T^{*}$.
Therefore $T^{*}$ is a triad in $N^{+}$.
We have just shown that $F_{i}^{+}$ is a fan in $N^{+}$,
and that both $x_{1}$ and $x_{m}$ are spoke elements.
This implies that $m$ is odd.

\begin{proposition}
\label{drunk}
For each $i\in I$, $\{a_{i},b_{i},c_{i}\}$ is a triangle in $N^{+}$.
\end{proposition}

\begin{proof}
We have already noted that $r_{N^{+}}(\{a_{i},b_{i},c_{i}\})\leq 2$.
Let $F_{i}=(e_{1},\ldots, e_{n})$, and let
$F_{i}^{+}$ be $(x_{1},\ldots, x_{m})$.
Assume that $a_{i}$ and $c_{i}$ are parallel.
This means that $x_{1}$ and $x_{m}$ are parallel.
As $\{x_{1},x_{2},x_{3}\}$ and $\{x_{m-2},x_{m-1},x_{m}\}$ are triangles,
there is a circuit contained in $\{x_{2},x_{3},x_{m-2},x_{m-1}\}$
that contains $x_{m-1}$.
From this it follows that if $x_{m-1}=e_{k}$ (where $k$ is either
$n-1$ or $n$), then $e_{k}$ is in the closure of
$\{e_{1},\ldots, e_{k-1}\}$.
However, we also know that $\{e_{k-2},e_{k-1},e_{k}\}$ is a triad, so
$e_{k}$ is in the coclosure of  $\{e_{1},\ldots, e_{k-1}\}$ in $N$.
From this we can deduce that $\lambda_{N}(\{e_{1},\ldots,e_{k}\})\leq 1$,
and this leads to a contradiction to the fact that $N$ is
$3$\dash connected.
Therefore $a_{i}$ and $c_{i}$ are not parallel.
Next we will show that $a_{i}$ is not parallel to $b_{i}$.
A symmetrical argument can be used to prove that $c_{i}$ is not parallel to $b_{i}$,
and that therefore $\{a_{i},b_{i},c_{i}\}$ is a triangle, as desired.
Assume that $a_{i}$ and $b_{i}$ are parallel.
The set of rim elements in $(e_{1},\ldots,e_{n})$ is equal to
$\{x_{2},x_{4},\ldots,x_{m-1}\}$.
Thus $a_{i}$ is in the closure of this set and of $\{x_{2},x_{3}\}$.
Hence $\{x_{2},x_{3},x_{4},x_{6},\ldots, x_{m-1}\}$ contains a
circuit, and this circuit must contain at least three elements.
If $k$ is the largest integer such that $e_{k}$ is in this circuit,
then $e_{k}$ is in the closure and the coclosure of
$\{e_{1},\ldots, e_{k-1}\}$ in $N$.
From this we can derive a contradiction to the $3$\dash connectivity of $N$.
\end{proof}

If $F_{1}=(e_{1},\ldots, e_{m})$ and $F_{2}=(e_{1}',\ldots,e_{n}')$
are fans, we will say that $F_{1}$ is \emph{enclosed} in
$F_{2}$ if either $(e_{1},\ldots, e_{m})$ or
$(e_{m},\ldots, e_{1})$ is equal to
$(e_{i}',\ldots, e_{i+m-1}')$,
for some $i\in \{1,\ldots, n-m+1\}$.
Let \mcal{F} and $\mcal{F}'$ be collections of fans.
We will say that \mcal{F} is \emph{enclosed} in
$\mcal{F}'$ if there is an bijective function from
\mcal{F} to $\mcal{F}'$ such that every fan in \mcal{F}
is enclosed in its image.

The next result shows that a fan-extension of $N$ can be constructed
by gluing wheels to $\core(N)$, where $\core(N)$ is defined relative to
some pairwise disjoint family of fans in $N$.
Note that \Cref{court} and the hypotheses of the \namecref{chest} imply
that $M$ is a fan-extension of $N$.

\begin{lemma}
\label{chest}
Let $N$ be a $3$\dash connected representable matroid,
where $|E(N)|\geq 4$ and $N$ is neither a wheel nor a whirl.
Let $\mcal{F}_{N}$ be a pairwise disjoint family of fans in $N$.
Assume there is no fan, $F$, in $N$ such that two distinct fans in $\mcal{F}_{N}$
(considered as unordered sets) are contained in $F$.
Assume that $M$ is a $3$\dash connected matroid with $N$ as a minor, and
every minor of $M$ with $N$ as a minor is $3$\dash connected up to series and
parallel sets.
Let \mcal{F} be a covering family in $M$ relative to $\mcal{F}_{N}$.
There exists a pairwise disjoint family of fans, $\mcal{F}_{0}$, in $N$, such that
$\mcal{F}_{0}$ encloses $\mcal{F}_{N}$, and moreover, up
to relabeling, $M=\boxtimes(\core(N),\mcal{T},r,X)$, for some
blueprint $(\core(N),\mcal{T},r,X)$
where $\core(N)$ and \mcal{T} are determined by $\mcal{F}_{0}$.
Furthermore, \mcal{F} is enclosed in the family of canonical fans of
$\boxtimes(\core(N),\mcal{T},r,X)$.
\end{lemma}

\begin{proof}
Let $E=E(N)$.
The proof will be by induction on $|E(M)|-|E|$.
First we assume $|E(M)|=|E|$.
Then $M=N$.
Now \mcal{F} is a covering family of $N$, relative to $\mcal{F}_{N}$.
This does not mean that $\mcal{F}=\mcal{F}_{N}$.
The fans in $\mcal{F}_{N}$ may be enclosed in larger fans, and these
larger fans may be members of \mcal{F}.
Nonetheless, \mcal{F} is a pairwise disjoint family of fans in $N$.
Because each fan in $\mcal{F}_{N}$ is consistent with
a fan in \mcal{F}, it follows without difficulty from \Cref{unity} that
each fan in $\mcal{F}_{N}$ must be enclosed
by a fan in \mcal{F}.
Thus we can define $\mcal{F}_{0}$ to be \mcal{F}.
Let $\mcal{F}=\{F_{i}\}_{i\in I}$, where $I=\{1,\ldots, t\}$, and let
$\core(N)$ and $\mcal{T}=\{\{a_{i},b_{i},c_{i}\}\}_{i\in I}$ be the matroid and
family of triangles determined by \mcal{F}.
To prove the base case of the induction, we now need only check that
$N=\boxtimes(\core(N),\mcal{T},r,X)$ (up to relabeling).

Recall that $N^{+}$ is the matroid consisting of the points in $E$
and $\cup_{i\in I}\{a_{i},b_{i},c_{i}\}$.
Let $F_{t}^{+}$ be $(x_{1},\ldots, x_{m})$.
This is a fan in $N^{+}$, and $x_{1}$ and $x_{m}$ are spoke elements.
The rim elements in $F_{t}$ are $\{x_{2},x_{4},\ldots, x_{m-1}\}$.
Thus $b_{t}$ is in the closure of this set.
Any circuit contained in $\{x_{2},x_{4},\ldots, x_{m-1},b_{t}\}$
must contain $\{x_{2},x_{4},\ldots, x_{m-1}\}$, for otherwise we can find a
violation of orthogonality with one of the triads in $F_{t}$.
Thus $\{x_{2},x_{4},\ldots, x_{m-1},b_{t}\}$ is a circuit in $N^{+}$, so
$\{x_{m-1},b_{t}\}$ is a circuit in
$N^{+}/\{x_{2},x_{4},\ldots, x_{m-3}\}\ba \{x_{3},x_{5},\ldots,x_{m-2}\}$.
We are going to apply a result by Oxley and Wu \cite[Theorem~1.8]{OW00}.
This theorem applies only to $3$\dash connected matroids, but
the simplification of $N^{+}$ is $3$\dash connected, and it
is easy to see that the result still holds.
The theorem by Oxley and Wu tells us that, up to relabeling, $N^{+}$ is equal to
$(N^{+}\ba F_{t})\boxtimes_{\{a_{t},b_{t},c_{t}\}} W_{t}$,
where $W_{t}$ has rank $(m+1)/2$.
Moreover $N^{+}\ba F_{t}$ is $3$\dash connected up
to parallel pairs.
Now $F_{t-1}^{+}$ is a fan in $N^{+}$.
Note that $F_{t}$ is disjoint from $F_{t-1}^{+}$ by construction of $\core(N)$.
As $\si(N^{+}\ba F_{t})$ is $3$\dash connected, it is easy to see that
$F_{t-1}^{+}$ is a fan in $N^{+}\ba F_{t}$ also.
By again using \cite[Theorem~1.8]{OW00}, we see that
$N^{+}\ba F_{t}$ can be obtained by using
generalized parallel connection to glue a wheel of the correct rank to
$N^{+}\ba F_{t}\ba F_{t-1}$ along the line $\{a_{t-1},b_{t-1},c_{t-1}\}$.
We proceed inductively in this way, and conclude that
$N^{+}$ is obtained from
\[
N^{+}\ba F_{t}\ba F_{t-1}\ba \cdots\ba F_{1}
\]
by attaching wheels via generalized parallel connections.
Let $S$ be the set of elements in $E-(\cup_{i\in I}F_{i})$ that are parallel to
some element $a_{i}$ or $c_{i}$.
Then $\core(N)=N^{+}\ba (\cup_{i\in I}F_{i})\ba S$.
This means that $N^{+}\ba S$ can be obtained from
$\core(N)$ by attaching wheels to the lines
$\{\{a_{i},b_{i},c_{i}\}\}_{i\in I}$.
This in turn means that $N$ can be obtained from $\core(N)$
in the following way:
attach wheels to the lines $\{\{a_{i},b_{i},c_{i}\}\}_{i\in I}$, and then,
for each element $s\in S$, distinguish an element
in $\{\{a_{i},b_{i},c_{i}\}\}_{i\in I}$ that is parallel to it, and relabel
that element as $s$, and then finally delete all other elements in
$\{\{a_{i},b_{i},c_{i}\}\}_{i\in I}$.
In other words, up to relabeling, we can obtain $N$ from $\core(N)$
by gluing on wheels.
This is exactly what we aimed to prove, and it establishes the
base case of our induction.

Now we can assume that $M\ne N$.
\Cref{burst} implies that there is a fan $F\in \mcal{F}$ such that
$M$ and $F$ are obtained from a $3$\dash connected matroid $M'$
by performing a fan-lengthening move on the fan $F'$.
It is clear that $\mcal{F}'=(\mcal{F}-\{F\})\cup\{F'\}$ is a covering family in $M'$, and
we can apply the inductive hypothesis.
There is a family, $\mcal{F}_{0}$, of pairwise disjoint fans in $N$ such that
$\mcal{F}_{0}$ encloses $\mcal{F}_{N}$.
We relabel $M'$ in such a way that it is obtained by gluing wheels to
$\core(N)$ (relative to $\mcal{F}_{0}$), and $\mcal{F}'$ is enclosed
in the family of canonical fans.
Let $X$ be the set of elements that we delete after attaching wheels
to $\core(N)$.
This means that $F'$ is enclosed in the
canonical fan $(x_{1}^{i},y_{1}^{i},x_{2}^{i},\ldots, y_{r(i)-1}^{i},x_{r(i)}^{i})-X$,
for some $i$.
Let $F$ be $(e_{1},\ldots, e_{m})$.

Assume that $M'$ is obtained from $M$ by deleting a terminal spoke element of $F$.
By reversing, we may assume that $M'=M\ba e_{1}$.
Now $F'=(e_{2},\ldots, e_{m})$ is enclosed in
$(x_{1}^{i},y_{1}^{i},x_{2}^{i},\ldots, y_{r(i)-1}^{i},x_{r(i)}^{i})-X$.
As $\{e_{2},e_{3},e_{4}\}$ is a triad in $M'$, by reversing
$(x_{1}^{i},y_{1}^{i},x_{2}^{i},\ldots, y_{r(i)-1}^{i},x_{r(i)}^{i})$
as necessary, we can assume
$e_{2}=y_{k}^{i}$ and $e_{3}=x_{k+1}^{i}$.
Assume for a contradiction that either (i) $k>1$, or (ii) $k=1$ and $x_{1}^{i}=e_{1}$,
and hence $x_{1}^{i}\in E(N)$, or (iii) $k=1$ and $x_{1}^{i}=a_{i}$,
but $a_{i}\notin X$.
If any of these situations hold, then $\{x_{k}^{i},y_{k}^{i},x_{k+1}^{i}\}$ and
$\{e_{1},e_{2},e_{3}\}$ are both triangles in $M$, and they intersect in the elements
$e_{2}$ and $e_{3}$.
This means that $\{x_{k}^{i},e_{1},e_{2}\}$ is a triangle in $M$ that intersects
the triad $\{e_{2},e_{3},e_{4}\}$ in a single element.
This contradiction means that $k=1$, $x_{1}^{i}=a_{i}$, and $a_{i}\in X$.
Now let $M''$ be the matroid obtained by gluing the same wheels to
$\core(N)$, except that instead of deleting $X$, we delete $X-x_{1}^{i}$.
Obviously $M''\ba x_{1}^{i}=M'=M\ba e_{1}$.
We relabel $M$ so that it inherits the relabeling of $M'$ and $e_{1}$ is relabeled
as $x_{1}^{i}$.
This means that $(x_{1}^{i},y_{1}^{i},x_{2}^{i},y_{2}^{i})$ is a fan in both $M$ and $M''$.
Now \Cref{peace} implies that $M''=M$.
It is clear that after the relabeling $F$ is enclosed in the canonical fan
$(x_{1}^{i},y_{1}^{i},x_{2}^{i},\ldots, y_{r(i)-1}^{i},x_{r(i)}^{i})-(X-x_{1}^{i})$, so
the result holds.

Next we will assume that $M'=M/e_{1}$, where $e_{1}$ is a rim element in $F$.
We can assume that $e_{2}=x_{k}^{i}$ and $e_{3}=y_{k}^{i}$.
If $k>1$, then $\{y_{k-1}^{i},x_{k}^{i},y_{k}^{i}\}$ is a triad in $M'$, and hence in
$M$.
As $\{e_{1},e_{2},e_{3}\}$ is also a triad of $M$, it follows that $\{y_{k-1}^{i},e_{1},e_{2}\}$
is a triad in $M$.
As $\{e_{2},e_{3},e_{4}\}$ is a triangle, this is a contradiction, so $k=1$.
Therefore $a_{i}$ is parallel to $x_{1}^{i}=e_{2}$.
This means that we can assume that $e_{2}$ is relabeled as $a_{i}$
when $M'$ is relabeled in such a way that
it becomes equal to $\boxtimes(\core(N),\mcal{T},r,X)$.
Thus $a_{i}\notin X$.
Let $W_{i}''$ be a copy of a wheel with rank $r(i)+1$.
Let the ground set of $W_{i}''$ be $E(W_{i})\cup \{x,y\}$, labeled in such a way that
$(x,y,x_{1}^{i},y_{1}^{i},x_{2}^{i},\ldots, y_{r(i)-1}^{i},x_{r(i)}^{i},y_{r(i)}^{i})$ is a fan
and $\{x,y_{r(i)}^{i},x_{r(i)}^{i}\}$ is a triangle.
Therefore $x$ is identified with $a_{i}$.
Now let $M''$ be the matroid obtained from $\core(N)$ by gluing on the same
wheels as before, except that we use $W_{i}''$ instead of $W_{i}$,
and we then delete $X\cup a_{i}$.
Now $W_{i}''/y\ba x=W_{i}$.
From this it follows easily that $M''/y=M'$.
We relabel $M$ in such a way that it inherits the labeling we applied to $M'$,
and $e_{1}$ is labeled as $y$.
Now $(y,x_{1}^{i},y_{1}^{i},x_{2}^{i})$ is a fan in both $M$ and $M''$,
and $M/y=M'=M''/y$.
Therefore the dual of \Cref{peace} implies that $M''=M$, as desired.

For the final case, we assume $|E(M')|=|E(M)|-2$.
Since $(\mcal{F}-\{F\})\cup\{F'\}$ is a covering family in $M'$, it follows that $m\geq 5$.
By reversing $F$ as necessary, we can assume that $M'=M/e_{j}\ba e_{j+1}$,
where $e_{j}$ is a rim element of $F$.
We will assume that $j>1$.
An almost identical argument will hold in the case that $j=1$.
Now $(e_{1},\ldots, e_{j-1},e_{j+2},\ldots, e_{m})$ is a fan of $M'$ that is
enclosed in $(x_{1}^{i},y_{1}^{i},x_{2}^{i},\ldots, y_{r(i)-1}^{i},x_{r(i)}^{i})$.
By reversing $(x_{1}^{i},y_{1}^{i},x_{2}^{i},\ldots, y_{r(i)-1}^{i},x_{r(i)}^{i})$
as necessary, we can assume that $e_{j-1}=x_{k}^{i}$ and
$e_{j+2}=y_{k}^{i}$ (if $j+2\leq m$).
Let $W_{i}''$ be a copy of a wheel with rank $r(i)+1$, and
let the ground set of $W_{i}''$ be $E(W_{i})\cup \{x,y\}$.
We label $W_{i}''$ in such a way that
$(x_{1}^{i},y_{1}^{i},\ldots,x_{k}^{i},y,x,y_{k}^{i},\ldots,x_{r(i)}^{i})$ is a fan.
Now let $M''$ be the matroid obtained from $\core(N)$ by gluing on the same
wheels used to obtain $M'$, except that we use $W_{i}''$ instead of $W_{i}$.
After attaching these wheels to $\core(N)$ we delete exactly
the same set $X$.
Now $W_{i}''/y\ba x=W_{i}$, and it follows that $M''/y\ba x=M'$.
We relabel $M$ so that it inherits the relabeling of $M'$, while
$e_{i}$ is labeled as $y$ and $e_{i+1}$ is labeled as $x$.
After this relabeling, $M/y\ba x=M'=M''/y\ba x$.
We claim that this implies $M=M''$.
Note that $M/y$ and $M''/y$ are both obtained by adding $x$ parallel
to $x_{k}^{i}$ in $M/y\ba x=M''/y\ba x$.
Thus $M/y=M''/y$.
If either $(x_{k-1}^{i},y_{k-1}^{i},x_{k}^{i},y)$ or $(y,x,y_{k}^{i},x_{k+1}^{i})$ is a fan in
both $M$ and $M''$, then the dual of \Cref{peace} implies that $M=M''$ as desired.
If neither of these sequences is a fan in $M$ and $M''$, then
since $m\geq 5$, it is not difficult to see that $m=5$, and $e_{j}=e_{3}$.
Thus $(x_{1}^{i},x_{2}^{i},y,x,x_{3}^{i})$ is a fan in both $M$ and $M''$.
In this case $M\ba x$ and $M''\ba x$ are both obtained from
$M/y\ba x=M''/y\ba x$ by adding $y$ in series to $x_{3}^{i}$.
Thus $M\ba x=M''\ba x$, and we can use  \Cref{peace} to show that
$M=M''$, as claimed.
Obviously $F$ is contained in the canonical fan associated with gluing
$W_{i}''$ to $\core(N)$, so the proof is complete.
\end{proof}

\section{Applications}
\label{sect6}

Now we consider some applications of \Cref{stage}.
We omit the proofs, since \Cref{stage} and \Cref{chest} reduce them
to computational case checking.
Details can be found in \cite{CCCMWZ}.
Our first application concerns the binary matroids that are
$\{F_{7},F_{7}^{*}\}$\dash fragile.
We describe these via \emph{grafts}
(see \cite[p.~386]{Oxl11}).
Recall that $N_{12}$ is illustrated in \Cref{magic}(i).
\Cref{twist} gives graft representations of the matroids $N_{11}^{+}$ and $N_{12}$.
In both case, the distinguished vertices in the graft are those
vertices with squares around them.
In $N_{12}$, let $w_{4}$ be the matroid element that corresponds to the
set of distinguished vertices.

\begin{figure}[htb]
\centering
\includegraphics{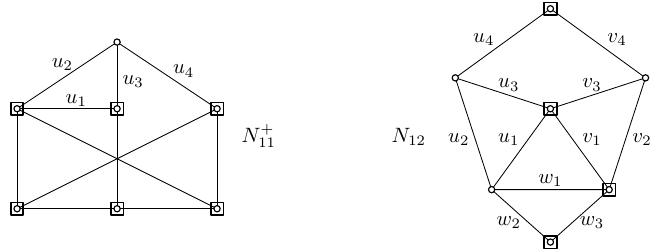}
\caption{Graft representations of $N_{11}^{+}$ and $N_{12}$.}
\label{twist}
\end{figure}

Now we have the tools to characterize binary
$\{F_{7},F_{7}^{*}\}$\dash fragile matroids that have $N_{11}^{+}$
or $N_{12}$ as a minor.
The following result is known to Truemper \cite{Tru92a} and
Kingan and Lemos \cite{KL02}.
It is a straightforward application of
\Cref{stage} and \Cref{chest}.

\begin{theorem}
\label{smoke}
Let $M$ be a $3$\dash connected binary $\{F_{7},F_{7}^{*}\}$\dash fragile matroid.
If $M$ has $N_{11}^{+}$ as a minor, then $M$ is a fan-extension of $N_{11}^{+}$
relative to $\{(u_{1},u_{2},u_{3},u_{4})\}$, and $M$ is a restriction of
a member of the family illustrated in the lefthand diagram in \textup{\Cref{night}}.
If $M$ has $N_{12}$ as a minor, then $M$ is a fan-extension of $N_{12}$
relative to $\{(u_{1},u_{2},u_{3},u_{4}),(v_{1},v_{2},v_{3},v_{4}),(w_{1},w_{2},w_{3},w_{4})\}$,
and $M$ is a restriction of a member of the family illustrated in the
righthand diagram in \textup{\Cref{night}}.
\end{theorem}

\begin{figure}[htb]
\centering
\includegraphics{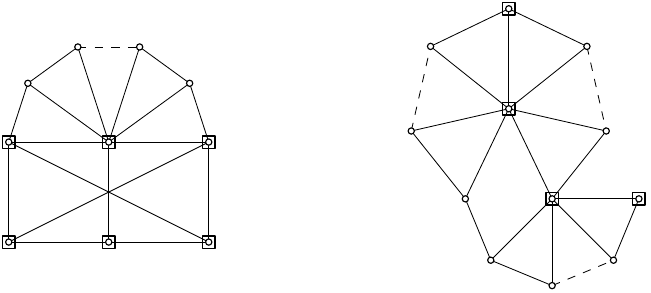}
\caption{The $N_{11}^{+}$ and $N_{12}$ families.}
\label{night}
\end{figure}

Next we consider fragile classes of matroid representable
over the partial fields $\mathbb{U}_{2}$ and $\mathbb{H}_{5}$.
Every $\mathbb{U}_{2}$\dash representable matroid is also
$\mathbb{H}_{5}$\dash representable, so it suffices to
consider only the latter partial field.
We currently have a solid understanding of the
$\mathbb{H}_{5}$\dash representable $\{U_{2,5},U_{3,5}\}$\dash fragile
matroids, and this brings excluded-minor characterizations for these two
partial fields within grasp.
Proofs and definitions can all be found in \cite{CCCMWZ}.

\begin{proposition}
\label{guide}
Let $M$ be a $3$\dash connected $\mathbb{H}_{5}$\dash representable
$\{U_{2,5},U_{3,5}\}$\dash fragile matroid.
Assume that $M$ has $N$ as minor, where $N$ is in
$\{M_{9,0},M_{9,2},M_{9,15},M_{9,18}\}$.
Then $M$ is a fan-extension of $N$.
\end{proposition}

\begin{theorem}
\label{water}
Let $M$ be a $3$\dash connected $\mathbb{H}_{5}$\dash representable matroid.
Up to relabeling, one of the following statements holds:
\begin{enumerate}[label=\textup{(\roman*)}]
\item $M$ has one of $X_{8}$, $Y_{8}$, or $Y_{8}^{*}$ as a minor,
\item $M$ is in $\{U_{2,6}, U_{3,6}, U_{4,6}, P_6, M_{9,9}, M_{9,9}^*\}$,
\item $M$ or $M^{*}$ can be obtained from $U_{2,5}$
(with groundset $\{a,b,c,d,e\}$) by gluing wheels to $\{(a,c,b), (a,d,b), (a,e,b)\}$,
\item $M$ or $M^{*}$ can be obtained from $U_{2,5}$
(with groundset $\{a,b,c,d,e\}$) by gluing wheels to $\{(a,b,c), (c,d,e)\}$,
\item $M$ or $M^{*}$ can be obtained from $M_{7,1}$ by gluing on a wheel.
\end{enumerate}
\end{theorem}

\section{Results on fans}
\label{sect2}

Now we start moving towards a proof of our main theorem.
This section collects some results that we will need.
The next \namecref{floor} can be proved with an easy inductive argument.

\begin{proposition}
\label{floor}
Let $(e_{1},\ldots, e_{n})$ be a fan
of the matroid $M$.
Then $\{e_{1},\ldots, e_{n}\}$ is 
$3$\dash separating in $M$.
\end{proposition}

\begin{proposition}
\label{cider}
Let $M$ be a $3$\dash connected matroid, and
let $F_{1}=(e_{1},\ldots, e_{n})$ be a fan of
$M$ such that $n\geq 5$.
Let $e_{i}$ be a rim element, for some $i\in\{1,\ldots, n-1\}$.
Then $F_{2}=(e_{1},\ldots, e_{i-1},e_{i+2},\ldots, e_{n})$ is 
a fan of $M/e_{i}\ba e_{i+1}$, and
for all $j$, $e_{j}$ is a spoke
element in $F_{1}$ if and only if it is
a spoke element in $F_{2}$
\end{proposition}

\begin{proof}
Any triad of $M$ contained in $\{e_{1},\ldots, e_{i-2}\}$
remains a triad in $M/e_{i}\ba e_{i+1}$, by orthogonality
with the triangle $\{e_{i-1},e_{i},e_{i+1}\}$.
Similarly, any triad of $M$ in $\{e_{i+4},\ldots, e_{n}\}$
is a triad in $M/e_{i}\ba e_{i+1}$ by orthogonality with
$\{e_{i+1},e_{i+2},e_{i+3}\}$.
In the same way, the triad $\{e_{i},e_{i+1},e_{i+2}\}$
implies that any triangle in
$\{e_{1},\ldots, e_{i-1}\}$ or $\{e_{i+3},\ldots, e_{n}\}$
is also a triangle in $M/e_{i}\ba e_{i+1}$.

Now the only sets we need check are
$\{e_{i-2},e_{i-1},e_{i+2}\}$ and $\{e_{i-1},e_{i+2},e_{i+3}\}$.
Strong cocircuit-exchange between
$\{e_{i-2},e_{i-1},e_{i}\}$ and $\{e_{i},e_{i+1},e_{i+2}\}$
gives a cocircuit contained in
$\{e_{i-2},e_{i-1},e_{i+1},e_{i+2}\}$ that contains $e_{i+1}$.
This cocircuit contains $e_{i-1}$, by orthogonality
with $\{e_{i-1},e_{i},e_{i+1}\}$.
If $e_{i-2}$ is not in the cocircuit, then
$\{e_{i-1},e_{i+1},e_{i+2}\}$ and $\{e_{i},e_{i+1},e_{i+2}\}$
are triads, so $\{e_{i-1},e_{i},e_{i+1}\}$ is a triad
and a triangle.
This contradicts the fact that $M$ is $3$\dash connected with
at least $5$ elements.
Hence $e_{i-2}$ is in the cocircuit, and a similar argument
shows that $\{e_{i-2},e_{i-1},e_{i+1},e_{i+2}\}$ is a cocircuit
of $M$.
Thus $\{e_{i-2},e_{i-1},e_{i+2}\}$ is a triad in
$M/e_{i}\ba e_{i+1}$, as desired.
We can show that $\{e_{i-1},e_{i+2},e_{i+3}\}$ is
a triangle of $M/e_{i}\ba e_{i+1}$ by a dual argument.
\end{proof}

\begin{proposition}
\label{ninny}
Let $(e_{1},\ldots, e_{n})$ be a fan in the $3$\dash connected
matroid $M$, and assume that $n\geq 4$.
Let $e\in E(M)-\{e_{1},\ldots, e_{n}\}$ be such that there is a triangle,
$T$, satisfying $\{e\}\subseteq T\subseteq \{e,e_{1},\ldots, e_{n}\}$.
Then either:
\begin{enumerate}[label=\textup{(\roman*)}]
\item $T=\{e,e_{1},e_{2}\}$, $e_{1}$ is a rim element,
and $(e,e_{1},\ldots, e_{n})$ is a fan,
\item $T=\{e,e_{n-1},e_{n}\}$, $e_{n}$ is a rim element,
and $(e_{1},\ldots, e_{n},e)$ is a fan,
\item $T=\{e,e_{1},e_{n}\}$, and $e_{1}$ and $e_{n}$ are spoke
elements,
\item $T=\{e,e_{2},e_{4}\}$, $e_{2}$ is a rim element, and $n\leq 5$, or
\item $T=\{e,e_{n-1},e_{n-3}\}$, $e_{n-1}$ is a rim element, and $n\leq 5$.
\end{enumerate}
\end{proposition}

\begin{proof}
Let $T=\{e,e_{x},e_{y}\}$, where $1\leq x<y\leq n$.
If $x=1$ and $y=n$, then statement~(iii) must hold,
since orthogonality with $T$ requires that $\{e_{1},e_{2},e_{3}\}$ and
$\{e_{n-2},e_{n-1},e_{n}\}$ are both triangles.
We will assume that $x>1$, and show that statement~(ii) or
statement~(iv) holds.
If $x=1$, then $y<n$, so we can replace $(e_{1},\ldots, e_{n})$ with
$(e_{n},\ldots, e_{1})$, and swap labels on $x$ and $y$, and then
apply exactly the same arguments to deduce that (i) or (v) holds.

Now we assume that $x>1$.
Assume also that $y=x+1$.
If $y<n$, then either $\{e_{x-1},e_{x},e_{x+1}\}$ is a triangle
and $\{e_{x},e_{x+1},e_{x+2}\}$ is a triad, or
$\{e_{x},e_{x+1},e_{x+2}\}$ is a triangle and
$\{e_{x-1},e_{x},e_{x+1}\}$ is a triad.
As $M$ is $3$\dash connected, in the first case
$\{e,e_{x-1},e_{x}\}$ is a triangle, and in the second case
$\{e,e_{x+1},e_{x+2}\}$ is a triangle.
In either case we have a contradiction to orthogonality.
Therefore $y=n$, and $e_{y}$ is a rim element, by
orthogonality between $T=\{e,e_{n-1},e_{n}\}$ and
$\{e_{n-3},e_{n-2},e_{n-1}\}$.
Now it is clear that statement~(ii) holds.
Hence we will assume that $y>x+1$.
Thus $e_{x}$ is a rim element, by orthogonality between
$T$ and $\{e_{x-1},e_{x},e_{x+1}\}$.
If $2<x$, then we have a contradiction to orthogonality
between $T$ and $\{e_{x-2},e_{x-1},e_{x}\}$.
Therefore $x=2$.
If $y>x+2$, then $T$ intersects the triad $\{e_{x},e_{x+1},e_{x+2}\}$
in a single element.
Therefore $y=x+2=4$.
Finally, $n\leq 5$, for otherwise $T$ intersects the triad
$\{e_{y},e_{y+1},e_{y+2}\}$ in a single element.
We have shown that statement~(iv) holds, so the proof is complete.
\end{proof}

\begin{proposition}
\label{hocus}
Assume that $(e_{1},\ldots, e_{n})$ is a fan of the
$3$\dash connected matroid $M$, and that
$4\leq |E(M)| \leq n+1$.
Then $M$ is a wheel or a whirl.
\end{proposition}

\begin{proof}
If $|E(M)|=4$, then $M$ is isomorphic to the whirl
$U_{2,4}$, and we are done.
Therefore we assume $|E(M)|>4$.
Since wheels and whirls are self-dual, we can switch to
$M^{*}$ as required, and assume that $e_{1}$ is a spoke
element of $(e_{1},\ldots, e_{n})$.
If $|E(M)|=n$, then result follows immediately from
\cite[Lemma~4.8]{OW00}.
Therefore we assume that $|E(M)| = n+1$.
Let $e$ be the single element in
$E(M)-\{e_{1},\ldots, e_{n}\}$.

Suppose that $e_{n}$ is a rim element, so that
$n$ is even.
Then
$\{e_{1},e_{3},\ldots, e_{n-3},e_{n-1}\}$ spans
$E(M)-\{e,e_{n}\}$.
Since $M$ is $3$\dash connected, it also spans $E(M)$.
Let $C$ be a circuit contained in
$\{e_{1},e_{3},\ldots, e_{n-3},e_{n-1}\}\cup\{e\}$
that contains $e$.
If $i$ is odd and $3\leq i \leq n-1$, then
$e_{i}\notin C$, or else $C$ intersects the triad
$\{e_{i-1},e_{i},e_{i+1}\}$ in a single
element.
Therefore $C\subseteq \{e,e_{1}\}$, which is
impossible as $M$ is $3$\dash connected.
This implies that $e_{n}$ is a spoke element and $n$ is odd.
Suppose that
$\{e_{3},e_{5},\ldots, e_{n-2},e_{n}\}$
spans $E(M)$.
Let $C$ be a circuit in
$\{e_{3},e_{5},\ldots, e_{n-2},e_{n}\}\cup\{e_{1}\}$
that contains $e_{1}$.
Then $C\subseteq \{e_{1},e_{n}\}$, for otherwise
$C$ intersects a triad in a single element.
This is a contradiction, so
$\{e_{3},e_{5},\ldots, e_{n-2},e_{n}\}$
does not span $E(M)$.
Since it spans $E(M)-\{e,e_{1},e_{2}\}$, it
follows that $\{e,e_{1},e_{2}\}$ is a triad.
Thus $(e,e_{1},\ldots, e_{n})$ is a fan, and
we can again apply \cite[Lemma~4.8]{OW00} to deduce
that $M$ is a wheel or a whirl.
\end{proof}

\begin{proposition}
\label{thumb}
Let $M$ be a $3$\dash connected matroid with
$|E(M)| \geq 4$.
Let $(e_{1},\ldots, e_{n})$ be a fan of $M$.
If there is an element $e\in E(M)-\{e_{1},\ldots, e_{n}\}$, and
elements $x,y,z\in \{e_{1},\ldots, e_{n}\}$ such that
$\{e,x,y\}$ is a triangle and $\{e,y,z\}$ is a triad, then
$M$ is a wheel or a whirl.
\end{proposition}

\begin{proof}
Let $F=\{e_{1},\ldots, e_{n}\}$.
Then $\lambda(F)\leq 2$, and as $e$ is in the closure and
coclosure of $F$, $\lambda(F\cup e)\leq 1$.
Therefore the complement of $F\cup e$ contains at most one
element.
If $n=3$, then either $F$ and $\{e,x,y\}$ are both triangles, or
$F$ and $\{e,y,z\}$ are both triads.
In the first case $M|(F\cup e)\cong U_{2,4}$, and in
the second $M^{*}|(F\cup e)\cong U_{2,4}$.
In either case, $F\cup e$ contains a triangle that is a triad.
Since $M$ is $3$\dash connected, this means $|E(M)|=4$,
and $M$ is isomorphic to $U_{2,4}$, so we are done.
Therefore we assume $n\geq 4$.

We apply \Cref{ninny}.
If statement~(i) or~(ii) holds, then we can apply
\Cref{hocus} to $(e_{1},\ldots, e_{n},e)$ or
$(e,e_{1},\ldots, e_{n})$ and conclude that $M$ is a wheel or a
whirl.
Next assume that statement~(iii) holds.
Then $n\geq 5$.
By reversing $(e_{1},\ldots, e_{n})$ as necessary, we
assume that $y=e_{n}$.
Then orthogonality between the triad $\{e,y,z\}$ and
the triangle $\{e_{n-2},e_{n-1},e_{n}\}$
requires that $z$ is in $\{e_{n-2},e_{n-1}\}$.
If $z=e_{n-1}$, then $(e_{1},\ldots, e_{n},e)$ is a fan,
and we can again apply \Cref{hocus}.
Therefore we assume $z=e_{n-2}$.
This means that the triangle $\{e_{n-4},e_{n-3},e_{n-2}\}$
and the triad $\{e,y,z\}$ meet in a single element, a
contradiction.
Therefore statement~(iv) or~(v) in \Cref{ninny} holds.
By reversing, we can assume that $\{x,y\}=\{e_{2},e_{4}\}$,
and $e_{2}$ is a rim element.

If $n=4$, then we apply \Cref{hocus} to the fan
$(e_{1},e_{3},e_{2},e_{4},e)$.
Therefore we assume $n=5$.
If $z=e_{1}$, then $y=e_{2}$, by orthogonality between
the triangles $\{e_{1},e_{2},e_{3}\}$ and $\{e_{3},e_{4},e_{5}\}$,
and the triad $\{e,y,z\}$.
Therefore $(e,e_{1},e_{2},e_{3},e_{4},e_{5})$ is a fan.
The same argument disposes of the case that $z=e_{5}$.
Therefore we assume $z=e_{3}$.
But in this case either $\{e_{1},e_{2},e_{3}\}$ or
$\{e_{3},e_{4},e_{5}\}$ is a triangle that intersects the
triad $\{e,y,z\}$ in a single element.
This contradiction completes the proof.
\end{proof}

\begin{proposition}
\label{steal}
Let $M$ be a $3$\dash connected matroid such that $|E(M)|\geq 4$,
and $M$ is not a wheel or a whirl.
Let $(e_{1},\ldots, e_{n})$ be a fan of $M$.
If $T\subseteq \{e_{1},\ldots, e_{n}\}$ is a triangle,
then $T=\{e_{i},e_{i+1},e_{i+2}\}$ for some
$i\in\{1,\ldots, n-2\}$ such that $e_{i}$ is a spoke
element.
\end{proposition}

\begin{proof}
Let $i$, $j$, and $k$ be such that
$1 \leq i<j< k\leq n$, and $T=\{e_{i},e_{j},e_{k}\}$.
If $e_{i}$ is a rim element of $(e_{1},\ldots, e_{n})$,
then $\{e_{i},e_{i+1},e_{i+2}\}$ is a triad.
Since $T$ is a triangle, it follows that
$e_{i}$ is in the closure and the coclosure of
$\{e_{i+1},\ldots, e_{n}\}$.
As $\lambda(\{e_{i+1},\ldots, e_{n}\})\leq 2$, it follows that
$\lambda(\{e_{i},\ldots, e_{n}\})\leq 1$.
As $M$ is $3$\dash connected, we deduce that
the complement of $\{e_{i},\ldots, e_{n}\}$ contains at
most one element.
Thus $M$ is a wheel or a whirl by \Cref{hocus},
and this is a contradiction.
Therefore $e_{i}$ is a spoke element.
The same argument shows that $e_{k}$ is a spoke element.

Assume that $j>i+1$.
If $e_{j}$ is a rim element of
$(e_{1},\ldots, e_{n})$, then $j\geq i+3$, so $T$
intersects the triad $\{e_{j-2},e_{j-1},e_{j}\}$ in a single
element, violating orthogonality.
Therefore $e_{j}$ is a spoke element.
If $k>j+1$, then $T$ intersects the
triad $\{e_{j-1},e_{j},e_{j+1}\}$ in a single
element.
Therefore $k=j+1$, implying that $e_{k}$ is a rim element.
Since this contradicts our earlier conclusion, we see that
$j=i+1$.
Now $k=i+2$, by orthogonality between $T$ and
$\{e_{i+1},e_{i+2},e_{i+3}\}$, and the fact that $e_{k}$
is a spoke element.
\end{proof}

\begin{proposition}
\label{vicar}
Let $M$ be a $3$\dash connected
matroid such that $|E(M)| \geq 4$ and $M$ is neither a
wheel nor a whirl.
Let $(e_{1},\ldots, e_{n})$ be a fan of $M$.
If $X\subseteq \{e_{1},\ldots, e_{n}\}$ is a
$3$\dash separating set of $M$ such that $|X|\geq 3$,
then there are integers $1\leq x < y \leq n$ such that
$X=\{e_{x},e_{x+1},\ldots, e_{y}\}$.
\end{proposition}

\begin{proof}
Assume that the result fails for $X$, and that $X$ is as
large as possible with respect to this assumption.
If $|E(M)|=4$, then $M$ is $U_{2,4}$, contradicting the fact that
$M$ is not a whirl.
If $|E(M)|=5$, then $M$ is $U_{2,5}$ or $U_{3,5}$.
In the first case $M$ has no triads, and in the second,
$M$ has no triangles.
In either case, $n=3$, and $X=\{e_{1},e_{2},e_{3}\}$.
Therefore we may as well assume that $|E(M)|\geq 6$.
Let $i$ be the least index such that
$|X\cap \{e_{1},\ldots, e_{i}\}|=3$.
Then $X\cup\{e_{1},\ldots, e_{i}\}$ is contained
in $\{e_{1},\ldots, e_{n}\}$, so the
complement of $X\cup\{e_{1},\ldots, e_{i}\}$
contains at least two elements, or else
$M$ is a wheel or a whirl by \Cref{hocus}.
As $M$ is $3$\dash connected, it follows that
$X\cup\{e_{1},\ldots, e_{i}\}$ is not
$2$\dash separating.
As $\lambda(\{e_{1},\ldots, e_{i}\})\leq 2$,
$\lambda(X)\leq 2$,
and $\lambda(X\cup \{e_{1},\ldots, e_{i}\})\geq 2$,
it follows that $X'=X\cap\{e_{1},\ldots, e_{i}\}$ is
a $3$\dash separating set with cardinality three,
by the submodularity of $\lambda$
(\cite[Lemma~8.2.9]{Oxl11}).
Since $|E(M)|\geq 6$, it follows that
$(X',E(M)-X')$ is a $3$\dash separation of $M$.
Thus $X'$ is a triad or a triangle
(\cite[Corollary~8.2.2]{Oxl11}).
By duality, we will assume that $X'$ is a triangle of $M$.
\Cref{steal} implies that
$X'=\{e_{i-2},e_{i-1},e_{i}\}$.

There must be an index $p> i$ such that $e_{p} \notin X$,
or else $X$ is not a counterexample.
Let $p$ be the least such index.
Note that $e_{p-2}$ and $e_{p-1}$ are in $X$.
Therefore $e_{p}$ is in either $\cl_{M}(X)$ or
$\cl_{M}^{*}(X)$, so $X\cup e_{p}$ is
$3$\dash separating.
Because $X\cup e_{p}$ is strictly larger than $X$, it follows that
$X\cup e_{p}$ is not a counterexample, so
$X\cup e_{p}=\{e_{x},e_{x+1},\ldots, e_{y}\}$ for some
indices $x$ and $y$.
Clearly $p < y$, or else
$X=\{e_{x},e_{x+1},\ldots, e_{y-1}\}$.
Thus $e_{p+1}$ is in $X$.
Now $\{e_{p-2},e_{p-1},e_{p}\}$ and
$\{e_{p-1},e_{p},e_{p+1}\}$ show that
$e_{p}$ is in the closure and the coclosure of $X$.
Therefore $\lambda(X\cup e_{p})\leq 1$, so
the complement of $X\cup e_{p}$ contains at most
one element.
Hence $M$ is a wheel or a whirl, and
this contradiction completes the proof.
\end{proof}

\begin{lemma}
\label{unity}
Let $M$ be a $3$\dash connected matroid such that
$|E(M)|\geq 4$ and $M$ is not a wheel or a whirl.
Let $(f_{1},\ldots, f_{m})$ and
$(g_{1},\ldots, g_{n})$ be fans of $M$.
Assume that
$|\{f_{1},\ldots, f_{m}\}\cap\{g_{1},\ldots, g_{n}\}|
\geq 3$.
Then there are integers $1\leq x < y\leq m$
and $1\leq x' < y'\leq n$ such that
\[
\{f_{1},\ldots, f_{m}\}\cap\{g_{1},\ldots, g_{n}\}
=\{f_{x},f_{x+1},\ldots, f_{y}\}
=\{g_{x'},g_{x'+1},\ldots, g_{y'}\}.
\]
\end{lemma}

\begin{proof}
The proof is by induction on $m+n$.
The hypotheses imply that $m+n\geq 6$.
If $m+n=6$, then obviously $m=n=3$ and
\[
\{f_{1},f_{2},f_{3}\}\cap\{g_{1},g_{2}, g_{3}\}
=\{f_{1},f_{2}, f_{3}\}
=\{g_{1},g_{2}, g_{3}\},
\]
so there is nothing left to prove.
Let us assume that $m+n>6$, and that the result does not
hold for $(f_{1},\ldots, f_{m})$
and $(g_{1},\ldots, g_{n})$.
We will make the inductive assumption that the result is true for
any pair of fans with combined length less than $m+n$.

Let $F=\{f_{1},\ldots, f_{m}\}$, and let
$G=\{g_{1},\ldots, g_{n}\}$.

\begin{sublemma}
\label{cloth}
$m,n>3$.
\end{sublemma}

\begin{proof}
By symmetry, it suffices to prove that $m > 3$.
If $m=3$, then $F\cap G=\{f_{1},f_{2},f_{3}\}$.
Moreover, $F$ is a triangle or a triad contained
in $G$.
\Cref{steal} or its dual tells us that $F=F\cap G$
consists of three consecutive elements from
$(g_{1},\ldots, g_{n})$.
Therefore the result holds, contrary to our choice of
$(f_{1},\ldots, f_{m})$ and
$(g_{1},\ldots, g_{n})$.
\end{proof}

\begin{sublemma}
\label{deity}
$f_{1},f_{m}\in G$ and
$g_{1},g_{n}\in F$.
\end{sublemma}

\begin{proof}
Suppose that $f_{1}\notin G$.
Then $(f_{2},\ldots, f_{m})$ and
$(g_{1},\ldots, g_{n})$ are fans by \ref{cloth}.
Since
$|\{f_{2},\ldots, f_{m}\}\cap\{g_{1},\ldots, g_{n}\}|\geq 3$,
the inductive assumption implies that
there are integers $1\leq x < y\leq m$
and $1\leq x' < y'\leq n$ such that
\begin{linenomath}
\begin{multline*}
\{f_{1},\ldots, f_{m}\}\cap\{g_{1},\ldots, g_{n}\}
=\{f_{2},\ldots, f_{m}\}\cap\{g_{1},\ldots, g_{n}\}\\
=\{f_{x},f_{x+1},\ldots, f_{y}\}
=\{g_{x'},g_{x'+1},\ldots, g_{y'}\},
\end{multline*}
\end{linenomath}
and the result holds, contrary to assumption.
Exactly the same argument shows that
$f_{m}\in G$ and $g_{1},g_{n}\in F$.
\end{proof}

Let the ground set of $M$ be $E$.

\begin{sublemma}
\label{cobra}
$|E-(F\cup G)| \leq 1$.
\end{sublemma}

\begin{proof}
Assume that $|E-(F\cup G)| \geq 2$.
As $|F\cup G| \geq 2$, and $M$ is $3$\dash connected,
it follows that $\lambda(F\cup G) \geq 2$.
Since $\lambda(F),\lambda(G)\leq 2$, the submodularity of
$\lambda$ shows that $\lambda(F\cap G) \leq 2$.
Then we can set $X=F\cap G$ and apply \Cref{vicar}
to see that
\[F\cap G=\{f_{x},f_{x+1},\ldots, f_{y}\}
=\{g_{x'},g_{x'+1},\ldots, g_{y'}\}
\]
for integers $1\leq x < y \leq m$ and $1\leq x' < y' \leq n$.
This contradiction completes the proof of the claim.
\end{proof}

\begin{sublemma}
\label{smell}
$|F-G|,|G-F|\geq 2$.
\end{sublemma}

\begin{proof}
We have complete symmetry between $F$ and $G$, so it suffices
to prove that $|F-G|\geq 2$.
Assume that $|F-G|\leq 1$.
If $|F-G|=0$, then $|E-G|\leq 1$ by \ref{cobra}, so
\Cref{hocus} implies that $M$ is a wheel or a whirl.
This contradiction means that $|F-G|=1$.
Let $f_{i}$ be the unique element in $F-G$.
By \ref{deity}, we deduce that $1 < i < m$.
Either $\{f_{i-1},f_{i},f_{i+1}\}$ and
$\{f_{i},f_{i+1},f_{i+2}\}$ are a triangle and a triad
that are contained in $G\cup f_{i}$, or we can make the same
statement about  $\{f_{i-1},f_{i},f_{i+1}\}$ and
$\{f_{i-2},f_{i-1},f_{i}\}$.
This means that we can apply \Cref{thumb}, and
deduce that $M$ is a wheel or a whirl.
This contradiction completes the proof of \ref{smell}.
\end{proof}

\begin{sublemma}
\label{faith}
The elements of $F-G$ form a consecutive subsequence
of $(f_{1},\ldots, f_{m})$.
The elements of $G-F$ form a consecutive subsequence
of $(g_{1},\ldots, g_{n})$.
\end{sublemma}

\begin{proof}
It suffices to prove that the elements in $F-G$ form
a consecutive subsequence of $(f_{1},\ldots, f_{m})$.
Certainly $F\cup(E-G)$ contains at least two elements.
Its complement, $G-F$, contains at least two elements
by \ref{smell}.
Therefore $\lambda(F\cup(E-G))\geq 2$.
As $\lambda(F),\lambda(E-G) \leq 2$, it follows that
$\lambda(F\cap (E-G))=\lambda(F-G)\leq 2$.

If $|F-G| \geq 3$, then we can let $X=F-G$, and apply
\Cref{vicar} to the fan $(f_{1},\ldots, f_{m})$.
Therefore we assume $|F-G|=2$.
Let $i,j\in\{1,\ldots, m\}$ be chosen so that
$i < j$, and $\{f_{i},f_{j}\}=F-G$.
By \ref{deity}, $1 < i < j < m$.
If $j=i+1$, there is nothing left to prove, so $j>i+1$.
Assume that $j>i+2$.
Then $\{f_{i-1},f_{i},f_{i+1}\}$ and $\{f_{i},f_{i+1},f_{i+2}\}$
are a triangle or triad contained in $G\cup f_{i}$.
\Cref{thumb} implies that $M$ is a wheel or a whirl.
Hence $j=i+2$.

By duality, we can assume that $f_{i}$ is a rim
element of $(f_{1},\ldots, f_{m})$.
Since $\{f_{i-1},f_{i},f_{i+1}\}$ is a triangle that
contains $f_{i}$ and is otherwise contained
in $G$, we can apply \Cref{ninny}.
From \Cref{smell} and the fact that $|F\cap G|\geq 3$,
we see that $m,n\geq 5$.
Thus $\{f_{i-1},f_{i+1}\}$ is $\{g_{1},g_{2}\}$,
$\{g_{n-1},g_{n}\}$, $\{g_{1},g_{n}\}$, or
$\{g_{2},g_{4}\}$.
But $\{f_{i+1},f_{i+2},f_{i+3}\}$ is a triangle, and
we can apply the same arguments to show that
$\{f_{i+1},f_{i+3}\}$ is also one of the same four sets.
We deduce that $\{\{f_{i-1},f_{i+1}\},\{f_{i+1},f_{i+3}\}\}$
is one of
$\{\{g_{1},g_{2}\},\{g_{1},g_{n}\}\}$,
$\{\{g_{1},g_{n}\},\{g_{n-1},g_{n}\}\}$,
$\{\{g_{1},g_{2}\},\{g_{2},g_{4}\}\}$,
or $\{\{g_{2},g_{4}\},\{g_{4},g_{5}\}\}$.
\Cref{ninny} implies that, in the first case,
$g_{1}$ is both a spoke element and a rim element.
In the second case it implies
$g_{n}$ is a spoke element and a rim element.
In the third case, $g_{2}$ is both a rim element and a
spoke element, and in the fourth, $g_{4}$ is a rim element
and a spoke element.
Thus we have a contradiction in any case.
\end{proof}

By \ref{deity}, \ref{smell}, and \ref{faith}, there are indices
$1\leq i < i+3\leq j\leq m$ such that
$F\cap G=\{f_{1},\ldots, f_{i}\}\cup\{f_{j}\ldots, f_{m}\}$.
Note that $\{f_{i},f_{i+1},f_{i+2}\}$ is a triangle or a triad
that contains exactly one element of $G$.
If $f_{i}=g_{k}$, then $k=1$ or $k=n$, for otherwise, since
$n \geq 4$, it follows that $g_{k}$ is contained in a
triangle and a triad that are both contained in $G$.
This leads to a contradiction to orthogonality.
By applying the same arguments to $\{f_{j-2},f_{j-1},f_{j}\}$,
we see that $f_{j}$ is equal to either $g_{1}$ or $g_{n}$.
By reversing $(g_{1},\ldots, g_{n})$
as necessary, we can assume that
$f_{i}=g_{n}$ and that $f_{j}=g_{1}$.

Assume that $|F\cap G| \geq 4$.
Either $1<i$ or $j<m$.
Let us first assume $1<i$.
The fans $(f_{1},\ldots, f_{m})$ and
$(g_{1},\ldots, g_{n-1})$ intersect in
$\{f_{1},\ldots, f_{i-1}\}\cup\{f_{j},\ldots, f_{m}\}$,
and this set contains at least three elements.
The inductive assumption implies that
$(f_{1},\ldots, f_{m})$ and
$(g_{1},\ldots, g_{n-1})$ should meet in consecutive
subsequences of $(f_{1},\ldots, f_{m})$ and
$(g_{1},\ldots, g_{n-1})$, but
$G-g_{n}$ contains $f_{i-1}$ and $f_{j}$, and
no element between them.
Similarly, if $j<m$, then
$(f_{1},\ldots, f_{m})$ and
$(g_{2},\ldots, g_{n})$ intersect in
$\{f_{1},\ldots, f_{i}\}\cup\{f_{j+1},\ldots, f_{m}\}$,
and this set is not consecutive in
$(f_{1},\ldots, f_{m})$.
From this contradiction to the inductive hypothesis, we
deduce that $|F\cap G|=3$.
Either $i=1$ and $j=m-1$, or $i=2$ and $j=m$.
We can reverse $(f_{1},\ldots, f_{m})$
and $(g_{1},\ldots, g_{n})$, and assume that the former case
holds.
Thus $f_{1}=g_{n}$ and $f_{m-1}=g_{1}$.
Since $\{f_{m-2},f_{m-1},f_{m}\}$ is a triangle
or a triad contained in $G\cup f_{m-2}$ that contains
$g_{1}$ but not $g_{n}$, \Cref{ninny} implies that
$f_{m}=g_{2}$.
Hence $F\cap G=\{f_{1},f_{m-1},f_{m}\}=\{g_{n},g_{1},g_{2}\}$.

By replacing $M$ with $M^{*}$, we can assume that $f_{m}=g_{2}$
is a spoke element of $(f_{1},\ldots, f_{m})$, so that
$\{f_{m-2},f_{m-1},f_{m}\}$ is a triangle.
Note that $\{g_{1},g_{2},g_{3}\}$ is not a triangle,
or else it intersects the triad $\{f_{m-3},f_{m-2},f_{m-1}\}$
in the single element $f_{m-1}=g_{1}$.
Therefore $g_{1}$ is a rim element of $(g_{1},\ldots, g_{n})$.
It follows that
\[
(f_{1},\ldots,f_{m-2}, f_{m-1},f_{m},g_{3},\ldots, g_{n-1})
=(f_{1},\ldots, f_{m-2},g_{1},g_{2},g_{3},\ldots, g_{n-1})
\]
is a fan of $M$.
By \ref{cobra}, there is at most one element of $E$ not contained in
this fan, so $M$ is a wheel or a whirl.
This contradiction completes the proof of \Cref{unity}.
\end{proof}

\section{A finite case-check theorem}
\label{sect4}

From this point onwards, \mcal{M} will be a class of matroids
closed under isomorphism and minors, and
$N\in \mcal{M}$ will be a fixed $3$\dash connected matroid
such that $|E(N)|\geq 4$, and $N$ is neither a wheel nor a whirl.
We let $\mcal{F}_{N}$ be a family of disjoint fans in $N$.
Assume that whenever $M'\in \mcal{M}$ has $N$ as a minor,
$M'$ is $3$\dash connected up to series and parallel sets.
Note that we can replace \mcal{M} with the dual class
$\{M^{*} \mid M\in \mcal{M}\}$ and replace $N$ with $N^{*}$
and the same hypotheses will hold.

In the subsequent results we state some consequences
of these definitions.
First of all, the only $3$\dash connected minors of
wheels and whirls are wheels and whirls.
This has the following implication.

\begin{proposition}
\label{ether}
No matroid having $N$ as a minor is a wheel or a whirl.
\end{proposition}

We are going to make continuous use of the next two results.

\begin{proposition}
\label{lasso}
Let $M'$ be a matroid in \mcal{M} that has $N$ as a minor.
Every triangle in $M'$ is coindependent, and every
triad is independent.
\end{proposition}

\begin{proof}
By duality, it suffices to assume that
$T$ is a codependent triangle in $M'$.
Then $r_{M'}(T)=2$,
and $r^{*}_{M'}(T)<3$, so $\lambda_{M'}(T)\leq 1$.
Clearly $T$ is not contained in a parallel class.
If $T$ is contained in a series class, then $\lambda_{M'}(T)=0$,
so $E(M')=T$ since $M'$ is connected.
This contradicts the fact that $N$ is a minor of $M'$.
As $M'$ is $3$\dash connected up to series and parallel sets,
the complement of $T$ is contained in a series or
parallel class.
This implies that a $3$\dash connected minor of $M'$ with
at least four elements has precisely four elements.
Hence $|E(N)|=4$, and $N$ is isomorphic to $U_{2,4}$, which contradicts
the assumption that $N$ is not a whirl.
\end{proof}

The next result follows from \Cref{lasso} by orthogonality.

\begin{corollary}
\label{putty}
Let $M'$ be a matroid in \mcal{M} that has $N$
as a minor.
If $X$ is a $U_{2,4}$\dash restriction of $M'$, and $T^{*}$ is a triad
of $M'$, then $X\cap T^{*}=\emptyset$.
\end{corollary}

\begin{proposition}
\label{shrub}
Let $M'$ be a $3$\dash connected matroid in \mcal{M}.
Assume that $(e_{1},\ldots, e_{n})$ is a fan of $M'$ such that
$n\geq 4$, and $N$ is a minor of
$M'/e_{i}\ba e_{i+1}$ for some $i\in\{1,\ldots, n-1\}$.
Then $M'/e_{i}\ba e_{i+1}$ is $3$\dash connected.
\end{proposition}

\begin{proof}
If $M'/ e_{i}\ba e_{i+1}$ is not $3$\dash connected,
then, as it is $3$\dash connected up to series and parallel sets,
there is a triangle that contains $e_{i}$ but not $e_{i+1}$,
or a triad that contains $e_{i+1}$ but not $e_{i}$.
We assume $T^{*}$ is a triad satisfying
$T^{*}\cap \{e_{i},e_{i+1}\}=\{e_{i+1}\}$.

First assume that $e_{i+1}$ is a rim element.
If $i>2$, then $\{e_{i-2},e_{i-1},e_{i}\}$ is a
codependent triangle in $M'\ba e_{i+1}$,
contradicting \Cref{lasso}.
Similarly, if $i<n-2$, then
$\{e_{i+1},e_{i+2},e_{i+3}\}$ is a dependent triad in
$M'/e_{i}$.
Therefore $n-2\leq i \leq 2$, implying that $n=4$ and $i=2$.
Since $e_{i}\notin T^{*}$, orthogonality with
$\{e_{i},e_{i+1},e_{i+2}\}$ implies $T^{*}$
contains the parallel pair $\{e_{i+1},e_{i+2}\}$
in $M'/e_{i}$, and we have another contradiction to
\Cref{lasso}.

Thus $e_{i+1}$ is a spoke element.
If $\{e_{i},e_{i+1}\}$ is contained in a triangle, then
$T^{*}$ contains the third element of this triangle, and
$T^{*}$ is a triad of $M'/ e_{i}$ that contains a parallel pair,
contradicting \Cref{lasso}.
Therefore there is no triangle containing
$\{e_{i},e_{i+1}\}$, so $i=1$.
In this case $T^{*}$ contains two elements of the
triangle $\{e_{i+1},e_{i+2},e_{i+3}\}$.
In a matroid without series pairs, two triads
that intersect in two elements form a
$U_{2,4}$\dash corestriction.
Therefore, if $T^{*}$ contains $e_{i+2}$, then $T^{*}\cup e_{i}$ 
is a $U_{2,4}$\dash corestriction of $M'$ that intersects
the triangle $\{e_{i+1},e_{i+2},e_{i+3}\}$, contradicting
\Cref{putty}.
Therefore $T^{*}\cap\{e_{i+1},e_{i+2},e_{i+3}\}=\{e_{i+1},e_{i+3}\}$.
Let $M''$ be produced from $M'$ by swapping labels on $e_{i}$ and
$e_{i+2}$.
As $\{e_{i},e_{i+2}\}$ is a series pair in $M'\ba e_{i+1}$,
it follows that $M''\ba e_{i+1}=M'\ba e_{i+1}$, so
$M''\ba e_{i+1}/e_{i}$ has $N$ as a minor.
However, $M''/e_{i}$ contains the parallel pair
$\{e_{i+1},e_{i+3}\}$, which is contained in the triad
$T^{*}$.
This contradicts \Cref{lasso}.

To complete the proof, we must assume that there is a triangle
containing $e_{i}$ but not $e_{i+1}$.
In this case we replace $M'$, $N$, and $\mcal{M}$ by their
duals, and we reverse the fan $(e_{1},\ldots, e_{n})$,
relabeling $e_{n-i}$ as $e_{i}$.
After the relabeling, there is a triad that contains $e_{i+1}$
but not $e_{i}$, so we make the same argument as before.
\end{proof}

\begin{proposition}
\label{cabal}
Let $M'$ be a matroid in \mcal{M} with $N$ as a minor.
Assume $M'$ has a covering family, and that
$F=(e_{1},\ldots, e_{n})$ is a fan in that family, where
$\{e_{i},e_{i+1}\}\subseteq E(M')-E(N)$
for some rim element $e_{i}$.
Then $N$ is a minor of $M'/e_{i}\ba e_{i+1}$.
\end{proposition}

\begin{proof}
Note that $n\geq 5$, since $F$ contains a fan in
$\mcal{F}_{N}$ as well as $\{e_{i},e_{i+1}\}$.
If $N$ is a minor of $M'\ba e_{i}$, then $n=5$ and
$i=3$, for otherwise $\{e_{i-3},e_{i-2},e_{i-1}\}$ or
$\{e_{i+1},e_{i+2},e_{i+3}\}$ is a codependent triangle in
$M'\ba e_{i}$.
But now $F-\{e_{i},e_{i+1}\}\subseteq E(N)$, so $N$
contains the series pair $\{e_{i-2},e_{i-1}\}$.
Therefore $N$ is a minor of $M'/e_{i}$.
If $i>1$, then $N$ is a minor of $M'/e_{i}\ba e_{i+1}$, as
$\{e_{i-1},e_{i+1}\}$ is a parallel pair in $M'/e_{i}$.
If $i=1$, then $\{e_{3},e_{4},e_{5}\}$ is a dependent triad in
$M'/e_{i}/e_{i+1}$, so $N'$ is a minor of $M'/e_{i}\ba e_{i+1}$ in
either case.
\end{proof}

\begin{proposition}
\label{llama}
Let $M'$ be a $3$\dash connected matroid in \mcal{M}
with $N$ as a minor.
Assume that $M'/x\ba y$ is $3$\dash connected and
has $N$ as a minor.
Let $F$ be a fan of $M'$ such that a fan in
$\mcal{F}_{N}$ is consistent with $F$.
If either $x$ or $y$ is an internal element of $F$,
then they are consecutive in $F$, and $x$ is a rim
element.
\end{proposition}

\begin{proof}
Let $F=(e_{1},\ldots, e_{m})$, and
assume $x$ or $y$ is an internal element in $F$.
In the latter case, we swap $M'$, $N$, and \mcal{M} for
their duals, and swap labels on $x$ and $y$.
Therefore we lose no generality in assuming $x$ is an
internal element, $e_{i}$, of $F$.
Assume $x$ is a spoke element.
If $n>5$, then $\{e_{i-3},e_{i-2},e_{i-1}\}$ or
$\{e_{i+1},e_{i+2},e_{i+3}\}$ is a dependent
triad in $M'/x$, which is a contradiction.
Therefore $n\leq 5$.
If $n=4$, then $F-x\subseteq E(N)$, as
$F$ contains a fan in $\mcal{F}_{N}$, as well as $x$.
But $F-x$ also contains a parallel pair in $M'/x$.
This leads to the contradiction that $N$ contains a
parallel pair.
Therefore $n=5$, and $x=e_{3}$, for otherwise
$M'/x$ contains a dependent triad.
Now either $\{e_{1},e_{2}\}$ or $\{e_{4},e_{5}\}$ is
contained in $E(N)$, and as both these sets are
parallel pairs in $M'/x$, this is a contradiction.

Hence $x$ is a rim element in $F$.
Because $\{e_{i-1},e_{i+1}\}$ is a parallel pair in
$M'/x$, and $M'/x\ba y$ is $3$\dash connected, it
follows that $y\in \{e_{i-1},e_{i+1}\}$, and we are done.
\end{proof}

\begin{proposition}
\label{metro}
Let $M'$ be a $3$\dash connected member of \mcal{M}, and assume that
$(e_{1},\ldots, e_{n})$ is a fan of $M'$ containing
$x$, $y$, and $z$, where $\{x,y,z\}$ is a triangle
of $M'$, and $N$ is a minor of $M'/x\ba y$.
If $n\geq 5$, then $x$ and $y$ are consecutive elements
in $(e_{1},\ldots, e_{n})$.
\end{proposition}

\begin{proof}
By \Cref{steal}, there is some $i\in\{1,\ldots, n-2\}$ such that
$\{x,y,z\}=\{e_{i},e_{i+1},e_{i+2}\}$, where
$e_{i}$ is a spoke element.
By reversing $(e_{1},\ldots, e_{n})$ as
necessary, we can assume that
$(x,y)=(e_{i},e_{i+2})$, for otherwise there
is nothing left to prove.
If $i+2< n$, then $\{e_{i+1},e_{i+2},e_{i+3}\}$ is a
dependent triad of $M'/e_{i}=M'/x$, since $\{e_{i+1},e_{i+2}\}$
is a parallel pair.
This is impossible, as $M'/x$ has $N$ as a minor.
Therefore $n=i+2$.

Let $M''$ be the matroid obtained from
$M'$ by swapping the labels on $e_{i+1}=z$ and $e_{i+2}=y$.
Since $\{y,z\}$ is a parallel pair in $M'/x$,
it follows that $M''/x\ba y=M'/x\ba y$.
Therefore $N$ is a minor of $M''\ba y$.
But $\{e_{i-2},e_{i-1},e_{i}\}$ is a triangle
of $M''\ba y$ that contains a series pair.
This contradiction completes the proof.
\end{proof}

\begin{proposition}
\label{tooth}
Let $M'\in\mcal{M}$ be a $3$\dash connected matroid with $N$ as
a minor.
If $e\in E(M')-E(N)$, then for some $(M_{1},N_{1})$
in $\{(M',N),((M')^{*},N^{*})\}$ one of the following
statements holds:
\begin{enumerate}[label=\textup{(\roman*)}]
\item $M_{1}/ e$ is $3$\dash connected and has $N_{1}$ as a minor,
\item there is a triangle $\{x,e,f\}$ of $M_{1}$ such that
$N_{1}$ is a minor of $M_{1}/e\ba f$ and $M_{1}\ba f$
is $3$\dash connected,
\item there is a fan $(x,e,f,y)$ of $M_{1}$, where
$\{x,e,f\}$ is a triangle, and $M_{1}/e\ba f$ is $3$\dash connected
with $N_{1}$ as a minor.
\end{enumerate}
\end{proposition}

\begin{proof}
By duality, we can assume $N_{1}$ is a minor of $M_{1}/e$.
We assume $M_{1}/e$ is not $3$\dash connected.
Since $M_{1}$ is $3$\dash connected, and $M_{1}/e$ is
$3$\dash connected up to series and parallel sets,
$M_{1}$ has a triangle containing $e$.
Let this triangle be $\{x,e,f\}$, where $N_{1}$ is a minor of
$M_{1}/e\ba f$.
We assume $M_{1}\ba f$ is not $3$\dash connected, so $f$ is in a
triad of $M_{1}$.
This triad contains $x$ or $e$ by orthogonality, and does
not contain both by \Cref{lasso}.
Let $y$ be the third element in the triad.
If $\{x,f,y\}$ is a triad, then it is a dependent
triad in $M_{1}/e$, which is impossible, so $\{e,f,y\}$
is a triad, and statement~(iii) holds by \Cref{shrub}.
\end{proof}

\begin{proposition}
\label{freak}
Let $N'$ be a matroid in \mcal{M} such that
$(e_{1},e_{2},e_{3},e_{4})$ is a fan of $N'$ with $e_{2}$
as a rim element, and $N'/e_{2}\ba e_{3}=N$.
Then $N'$ is $3$\dash connected.
\end{proposition}

\begin{proof}
Because $N'$ has $N$ as a minor, it is $3$\dash connected
up to series and parallel sets, so it is connected.
In particular, it has no loops or coloops.
If $e_{2}$ is in a parallel pair, then $N'/e_{2}$
has $N$ as a minor and a loop, which is impossible.
If $e_{2}$ is in a series pair, then orthogonality requires
that the pair is contained in $\{e_{1},e_{2},e_{3}\}$, so
this set is a codependent triangle in $N'$, a contradiction to
\Cref{lasso}.
By a dual argument, $e_{3}$ is not contained in any series
or parallel pair in $N$.
As $N=N'/e_{2}\ba e_{3}$ is simple and cosimple,
it follows that $N'$ is simple and cosimple.

Assume $N'$ is not $3$\dash connected.
As it is connected, we can let $(X,Y)$ be a $2$\dash separation.
Since $N'$ is simple and cosimple, it follows that $|X|,|Y|\geq 3$.
If $X\cap \{e_{2},e_{3}\}$ and $Y\cap\{e_{2},e_{3}\}$ are
both non-empty, then $(X-\{e_{2},e_{3}\},Y-\{e_{2},e_{3}\})$
is a $2$\dash separation in $N$, a contradiction.
Therefore we can assume $e_{2},e_{3}\in X$.
If $|X|>3$, then $(X-\{e_{2},e_{3}\},Y)$ is a $2$\dash separation
in $N$.
Therefore $|X|\leq 3$.
If $|E(N)|=4$, then $N$ is isomorphic to $U_{2,4}$, contradicting
the hypothesis that it is not a whirl.
Hence $|E(N')|=|E(N)|+2\geq 7$, so $|Y|\geq 4$.
Now $(X\cup\{e_{1},e_{4}\},Y-\{e_{1},e_{4}\})$ is a
$2$\dash separation in $N'$ that induces a $2$\dash separation
in $N$, so we have a contradiction that completes the proof.
\end{proof}

Now we can begin the proof of our main theorem.

\begin{theorem}
\label{arrow}
Let \mcal{M} be a set of matroids that is closed under
minors and isomorphism.
Let $N\in \mcal{M}$ be a $3$\dash connected matroid such that
$|E(N)|\geq 4$ and $N$ is neither a wheel nor a whirl.
Assume that any member of \mcal{M}
with $N$ as a minor is $3$\dash connected
up to series and parallel sets.
Let $\mcal{F}_{N}$ be a family of pairwise disjoint
fans of $N$.
If there is a $3$\dash connected matroid in \mcal{M}
with $N$ as a minor that is not a fan-extension of $N$
relative to $\mcal{F}_{N}$, then there
exists such a matroid, $M$, satisfying $|E(M)|-|E(N)|\leq 2$.
\end{theorem}

For the remainder of the paper, we let \mcal{M}, $N$, and
$\mcal{F}_{N}$ be as in the statement of \Cref{arrow}.
We let $M\in \mcal{M}$ be a $3$\dash connected matroid
with $N$ as a minor such that $M$ is not a fan-extension
of $N$, and, subject to these constraints, $|E(M)|$ is as
small as possible.
We assume that $|E(M)|-|E(N)|> 2$, and ultimately derive a
contradiction from this, thereby proving \Cref{arrow}.

\begin{lemma}
\label{fever}
Let $M_{0}$ be isomorphic to $M$, and assume that $M_{0}$
has $N$ as a minor, but is not a fan-extension of $N$.
Then $M_{0}$ does not have a covering family (relative
to $N$ and $\mcal{F}_{N}$).
\end{lemma}

\begin{proof}
We assume for a contradiction that $M_{0}$ has a covering family.
Note $M_{0}\ne N$, as $M_{0}$ is not a fan-extension of $N$.
The fans in any covering family of $M_{0}$ contain the elements of
$E(M_{0})-E(N)$.

\begin{sublemma}
\label{grass}
In any fan belonging to a covering family of $M_{0}$,
there are no two consecutive elements in $E(M_{0})-E(N)$.
\end{sublemma}

\begin{proof}
Assume there is a covering family, \mcal{F}, of $M_{0}$, and a
fan, $F=(e_1,\ldots,e_n)$, in \mcal{F}, such that
$e_{i}$ and $e_{i+1}$ are not in $E(N)$ for some
$i\in\{1,\ldots, n-1\}$.
If $e_{i}$ is a spoke element, then we can replace
$M$, $M_{0}$, $N$, and \mcal{M} by their duals, and then apply
the forthcoming arguments.
Thus we assume that $e_{i}$ is a rim element.
Then $N$ is a minor of $M_{0}/e_{i}\ba e_{i+1}$ by \Cref{cabal}.
\Cref{cider,shrub} show that $M_{0}/e_{i}\ba e_{i+1}$ is a
$3$\dash connected matroid containing the fan $F-\{e_{i},e_{i+1}\}$.
The minimality of $M$ implies $M_{0}/e_{i}\ba e_{i+1}$ is a
fan-extension of $N$.
It is easy to see that $(\mcal{F}-\{F\})\cup\{F-\{e_{i},e_{i+1}\}\}$ is a
covering family in $M_{0}/e_{i}\ba e_{i+1}$, and $M_{0}$ is obtained from
$M_{0}/e_{i}\ba e_{i+1}$ by a fan-lengthening move on
$F-\{e_{i},e_{i+1}\}$.
Therefore $M_{0}$ is a fan-extension of $N$, a contradiction.
\end{proof}

\begin{sublemma}
\label{giant}
If $F$ is a fan in a covering family of $M_{0}$,
then all internal elements of $F$ belong to $E(N)$.
\end{sublemma} 

\begin{proof}
Assume $F=(e_{1},\ldots, e_{n})$ is a fan in a covering family, and
$e_{i}$ is an internal element in $E(M_{0})-E(N)$, so $n\geq 4$.
By replacing $M$, $M_{0}$, $N$, and \mcal{M} with their duals if
necessary, we can assume $e_{i}$ is a rim element.
If $N$ is a minor of $M_{0}/e_{i}$, we contradict \ref{grass}, as
$\{e_{i-1},e_{i+1}\}$ is a parallel pair in $M_{0}/e_{i}$, so
$N$ is a minor of $M_{0}/e_{i}\ba e_{i-1}$ or $M_{0}/e_{i}\ba e_{i+1}$.
Therefore $N$ is a minor of $M_{0}\ba e_{i}$, so $n\leq 5$,
and if $n=5$, then $i=3$, for otherwise $\{e_{i-3},e_{i-2},e_{i-1}\}$
or $\{e_{i+1},e_{i+2},e_{i+3}\}$ is a codependent triangle in
$M_{0}\ba e_{i}$.
By reversing $F$ as required, we will assume that $i<n-1$.
Now $\{e_{i+1},e_{i+2}\}$ is a series pair in $M_{0}\ba e_{i}$, so
\ref{grass} means that $N$ is a minor of $M_{1}\ba e_{i}/ e_{i+2}$.
Thus $n=5$, since $F$ contains $\{e_{i},e_{i+2}\}$ and at least
three elements of $E(N)$.
In this case, $\{e_{i-2},e_{i-1},e_{i+1}\}\subseteq E(N)$,
which leads to a contradiction as $\{e_{i-2},e_{i-1}\}$ is
is a series pair in $M_{0}\ba e_{i}$.
\end{proof}

Fix the covering family \mcal{F}, and 
let $F=(e_{1},\ldots, e_{n})$ be a fan in \mcal{F} such that
$e_{1}\notin E(N)$.
This implies $n\geq 4$.
By duality, we assume $e_{1}$ is a rim element, and this
easily implies that $N$ is a minor of $M_{0}/ e_{1}$.
If $M_{0}/e_{1}$ is $3$\dash connected, then $M_{0}$ is obtained from
$M_{0}/ e_{1}$ by a fan-lengthening move on $F-e_{1}$,
and $(\mcal{F}-\{F\})\cup\{F-e_{1}\}$ is a covering family in
$M_{0}/ e_{1}$.
By minimality, $M_{0}/ e_{1}$ is a fan-extension of $N$, and therefore,
so is $M_{0}$.
Hence $M_{0}/ e_{1}$ is not $3$\dash connected.
As $M_{0}/ e_{2}$ contains a parallel pair, it is not $3$\dash connected.
The dual of Tutte's Triangle Lemma (see \cite[Lemma~8.7.7]{Oxl11}) implies
there is a triangle, $T$, of $M_{0}$ containing $e_{1}$ and either
$e_{2}$ or $e_{3}$.
It follows from \Cref{steal} that $T$ cannot be contained
in $F$, for otherwise $\{e_{1},e_{2},e_{3}\}$ is a triad and a triangle.
Let $f$ be the element in $T-\{e_{1},e_{2},e_{3}\}$.
Then $N$ is a minor of $M_{0}/e_{1}\ba f$, by \ref{giant}.
Since \mcal{F} is a covering family, there is a fan,
$F_{f}\in \mcal{F}$, such that $f$ is in $F_{f}$.
As $F_{f}$ contains at least three elements of $E(N)$, any
internal element of $F_{f}$ is contained in a triad.
Orthogonality with the triangle $T$ now implies that $f$ is a
terminal spoke element of $F_{f}$.
Clearly $F_{f}-f$ is a fan of $M_{0}$.

If $e_{2}$ is in $T$, then $F+f=(f,e_{1},\ldots, e_{n})$ is a fan
of $M_{0}$, and it is straightforward to see that
$(\mcal{F}-\{F,F_{f}\})\cup\{F+f,F_{f}-f\}$ is a covering family of
$M_{0}$.
But now $e_{1}$ is an internal element of $F+f$ that is not in $E(N)$,
contradicting \ref{giant}.
Hence $T=\{f,e_{1},e_{3}\}$.
This means that $n=4$, for otherwise $T$ violates orthogonality with
the triad $\{e_{3},e_{4},e_{5}\}$.
Thus $(e_{2},e_{3},e_{4})$ or its reversal is in $\mcal{F}_{N}$.

Let $F_{f}=(f_{1},\ldots, f_{m})$, where $f=f_{1}$.
By applying the duals of the previous arguments, we see that
$\{f_{1},f_{3}\}$ is contained in a triad of $M_{0}$, that $m=4$,
and that either $(f_{2},f_{3},f_{4})$ or its reversal is in $\mcal{F}_{N}$.
The triad containing $\{f_{1},f_{3}\}$ contains either $e_{1}$ or $e_{3}$,
by orthogonality with $T$, and the latter case cannot occur, by
orthogonality
with $\{e_{2},e_{3},e_{4}\}$.
Thus $\{f_{3},f_{1},e_{1}\}$ is a triad, and
$(f_{4},f_{2},f_{3},f_{1},e_{1},e_{3},e_{2},e_{4})$ is a fan of $M_{0}$.

Let $I$ and $I^{*}$ be, respectively, independent and coindependent
sets such that $I\cap I^{*}=\emptyset$ and $N=M_{0}/I\ba I^{*}$.
Let $N'=M_{0}/(I-\{e_{1},f_{1}\})\ba (I^{*}-\{e_{1},f_{1}\})$.
Note that $\{f_{4},f_{2},f_{3}\}$ and $\{e_{3},e_{2},e_{4}\}$ are,
respectively, a triad and a triangle in $N'$, for otherwise
$N$ contains a circuit or cocircuit with at most two elements.
Since $\{f_{2},f_{3},f_{1}\}$ is a union of circuits in $N'$,
it is easy to see it is a circuit of $N'$, for otherwise
we can use orthogonality with $\{f_{4},f_{2},f_{3}\}$ to show that
$\{f_{2},f_{3}\}$ contains a circuit in $N'$, and hence in $N$.
Hence $f_{1}$ is not a loop of $N'$.
Similarly, $\{e_{1},e_{3},e_{2}\}$ is a triad in $N'$ and
$e_{1}$ is not a coloop.
Since $\{f_{3},f_{1},e_{1}\}$ is a union of cocircuits,
orthogonality with $\{f_{2},f_{3},f_{1}\}$
now implies that $\{f_{3},f_{1},e_{1}\}$ is a triad in $N'$.
Orthogonality with $\{e_{1},e_{3},e_{2}\}$ implies
$\{f_{1},e_{1},e_{3}\}$ is a triangle in $N'$.
Therefore $(f_{4},f_{2},f_{3},f_{1},e_{1},e_{3},e_{2},e_{4})$ is a
fan of $N'$.

Because $N'\ba e_{1}$ contains the codependent triangle
$\{f_{1},f_{3},f_{2}\}$, we see that $N$ is not a minor of
$N'\ba e_{1}$.
Therefore it is a minor of $N'/e_{1}$, and this matroid
contains the parallel pair $\{f_{1},e_{3}\}$, so
$N=N'/e_{1}\ba f_{1}$.
By applying \Cref{freak} to the fan $(e_{3},e_{1},f_{1},f_{3})$,
we see that $N'$ is $3$\dash connected.
Since $|E(M)|-|E(N)|>2=|E(N')|-|E(N)|$, $N'$ is a proper minor
of $M_{0}$.
Therefore $N'$ is a fan-extension of $N$.

As $|E(N')|-|E(N)|=2$, it follows that $N'$ is obtained
from $N$ by either one or two fan-lengthening moves.
However, none of $N'/e_{1}$, $N'\ba e_{1}$, $N'/f_{1}$,
$N'\ba f_{1}$ is $3$\dash connected, so
$N'$ is obtained from $N$ by a single fan-lengthening move.
Hence there is a covering family, $\mcal{F}'$, of $N'$ containing a fan,
$F'$, such that $e_{1}$ and $f_{1}$ are consecutive elements
of $F'$.
Assume that $e_{3}$ is in $F'$.
Because $\mcal{F}'$ is a covering family and
$(e_{2},e_{3},e_{4})$ or its reversal is in $\mcal{F}_{N}$,
it follows that $(e_{2},e_{3},e_{4})$ is consistent with $F'$.
This means that $F'$ contains $\{f_{1},e_{1},e_{3},e_{2},e_{4}\}$.
By applying \Cref{steal} to the triangle $\{f_{1},e_{1},e_{3}\}$,
the triad $\{e_{1},e_{3},e_{2}\}$, and the triangle $\{e_{3},e_{2},e_{4}\}$,
we see that $(f_{1},e_{1},e_{3},e_{2},e_{4})$ is a
contiguous subsequence of $F'$.
Hence $(e_{2},e_{3},e_{4})$ is not consistent with $F'$ after all.
This contradiction shows that $e_{3}$ is not in $F'$.
As $\{e_{3},e_{1},f_{1}\}$ is a
triangle, \Cref{ninny} implies that, up to reversing,
$e_{1}$ and $f_{1}$ are the first two elements in $F'$, and
$F'$ starts with a rim element.
If $f_{3}$ is not in $F'$, then we can apply \Cref{ninny} to the triad
$\{f_{3},f_{1},e_{1}\}$ and get the contradiction that $F'$ starts with
a spoke element.
Therefore $f_{3}$ is in $F'$, and therefore $(f_{2},f_{3},f_{4})$ is
consistent with $F'$.
As $\{e_{1},f_{1},f_{3}\}$, $\{f_{1},f_{3},f_{2}\}$,
and $\{f_{3},f_{2},f_{4}\}$ are triangles or triads of $N'$
contained in $F'$, \Cref{steal} implies that
$(e_{1},f_{1},f_{3},f_{2},f_{4})$ is a contiguous subsequence of
$F'$, so $(f_{2},f_{3},f_{4})$ is not consistent with $F'$ after all.
This contradiction completes the proof.
\end{proof}

\begin{lemma}
\label{grape}
Let $M_{0}$ be a matroid isomorphic to $M$ such that
$N$ is a minor of $M_{0}$, but $M_{0}$ is not a fan-extension of $N$.
Assume that $(e_{1},\ldots, e_{n})$ is a fan of $M_{0}$,
where $n\geq 5$, and $\{e_{1},e_{2},e_{3}\}$ is a triangle.
If $N$ is a minor of $M_{0}/e_{2}\ba e_{3}$
and some ordering of $\{e_{1},e_{4},\ldots, e_{n}\}$ is in $\mcal{F}_{N}$,
then either $(e_{1},e_{4},\ldots, e_{n})$ or its reversal is in
$\mcal{F}_{N}$.
\end{lemma}

\begin{proof}
We assume that some ordering of
$\{e_{1},e_{4},\ldots, e_{n}\}$ is in
$\mcal{F}_{N}$.
Let $I$ and $I^{*}$ be independent and coindependent
sets such that $I\cap I^{*}=\emptyset$ and $N=M_{0}/I\ba I^{*}$
and let $N'=M_{0}/(I-\{e_{2},e_{3}\})\ba (I^{*}-\{e_{2},e_{3}\})$.
If $4\leq i\leq n-2$ and $i$ is odd (respectively, even), then
$\{e_{i},e_{i+1},e_{i+2}\}\subseteq E(N)$, and
$\{e_{i},e_{i+1},e_{i+2}\}$ is a circuit (cocircuit)
in $M_{0}$, and is therefore a union of circuits (cocircuits) in $N'$.
Hence it is a circuit (cocircuit) in $N'$, or else $N$
contains a circuit (cocircuit) of at most two elements, which
contradicts the $3$\dash connectivity of $N$.
Thus $(e_{4},\ldots, e_{n})$ is a fan of $N'$.
As $\{e_{3},e_{4},e_{5}\}$ is a union of circuits in $N'$,
it is a circuit, for otherwise $\{e_{4},e_{5}\}$ contains a circuit
in $N'$, and hence in $N$.
Since $\{e_{1},e_{2},e_{3}\}$ is a union of circuits in $N'$,
$e_{2}$ is not a coloop of $N'$.
Now, as $\{e_{2},e_{3},e_{4}\}$ is a union of cocircuits,
orthogonality with $\{e_{3},e_{4},e_{5}\}$ shows that
$\{e_{2},e_{3},e_{4}\}$ is a triad of $N'$.
Finally, $e_{1}$ is not a loop of $N$, and hence not
a loop in $N'$.
Neither $e_{2}$ nor $e_{3}$ is a loop in $N'$, as
$\{e_{2},e_{3},e_{4}\}$ is a triad.
If $\{e_{1},e_{2},e_{3}\}$ is not a circuit, then
it is contained in a parallel class, and this leads
to a contradiction to orthogonality with $\{e_{2},e_{3},e_{4}\}$.
Thus we have shown $(e_{1},\ldots, e_{n})$ is a fan in $N'$.

As $N'\ba e_{2}$ contains a codependent triangle, it
cannot have $N$ as a minor.
Therefore $N'/e_{2}$ has $N$ as a minor, as well as
the parallel pair $\{e_{1},e_{3}\}$.
Therefore $N=N'/e_{2}\ba e_{3}$.
\Cref{freak} implies that $N'$ is $3$\dash connected.
As $|E(M_{0})|-|E(N)|>2=|E(N')|-|E(N)|$, we see
that $N'$ is a fan-extension of $N$.

Let $F'$ be the fan in a covering family of
$N'$ such that some ordering of $\{e_{1},e_{4},\ldots, e_{n}\}$
is consistent with $F'$.
As $(e_{1},\ldots, e_{n})$ is a fan of $N'$
that intersects $F'$ in at least three elements,
including $e_{1}$ and $e_{n}$, we see from \Cref{unity}
that there is a contiguous subsequence of $F'$
comprising the elements $\{e_{1},\ldots,e_{n}\}$.
Now, repeatedly applying \Cref{steal}, it follows that
$(e_{1},\ldots, e_{n})$ is a contiguous subsequence of $F'$.
Thus $(e_{1},e_{4},\ldots, e_{n})$ is consistent with
$F'$, so this fan, or its reversal, is in $\mcal{F}_{N}$.
\end{proof}

If we apply an arbitrary permutation to the element labels of
$M$, we do not know if the resulting matroid will be a
fan-extension of $N$, or, indeed, if it will have $N$ as a minor.
In the next \namecref{flint}, and in \Cref{flame},
we look for circumstances under which the permuted matroid
is guaranteed to have $N$ as a minor, but not to be a
fan-extension of $N$.
This will allow us to relabel elements of $M$ without
losing generality.

\begin{lemma}
\label{flint}
Let $M_{0}$ be isomorphic to $M$, and assume that $M_{0}$
has $N$ as a minor, but is not a fan-extension of $N$.
Assume that $(e_{1},e_{2},e_{3},e_{4})$ is a fan
of $M_{0}$ with $e_{2}$ as a rim element, and
$N$ is a minor of $M_{0}/e_{2}\ba e_{3}$.
Let $M'$ be the matroid obtained from $M_{0}$
by swapping the labels on $e_{1}$ and $e_{3}$.
Then $M'$ contains $N$ as a minor, but is not a
fan-extension of $N$.
\end{lemma}

\begin{proof}
Because $\{e_{1},e_{3}\}$ is a parallel pair of
$M_{0}/e_{2}$, it follows that
$M'/e_{2}\ba e_{3}=M_{0}/e_{2}\ba e_{3}$, so $N$
is a minor of $M'/e_{2}\ba e_{3}$.
\Cref{shrub} implies $M_{0}/e_{2}\ba e_{3}$, and hence
$M'/e_{2}\ba e_{3}$, is $3$\dash connected.
Note that $(e_{3},e_{2},e_{1},e_{4})$ is a fan of $M'$,
where $e_{2}$ is a rim element.
Assume for a contradiction that $M'$ is a fan-extension of $N$.
As $e_{2},e_{3}\notin E(N)$, there
is a covering family of $M'$ that contains fans
that contain $e_{2}$ and $e_{3}$.

\begin{sublemma}
\label{sniff}
If $\mcal{F}'$ is a covering family of $M'$, then no fan
in $\mcal{F}'$ contains $\{e_{1},e_{2},e_{3}\}$.
\end{sublemma}

\begin{proof}
Assume $F'\in\mcal{F}'$ contains $\{e_{1},e_{2},e_{3}\}$.
Then \Cref{steal} implies that $\{e_{1},e_{2},e_{3}\}$
is a set of three consecutive elements in $F'$.
As $e_{2},e_{3}\notin E(N)$, it now is obvious that
we can swap the labels on $e_{1}$ and $e_{3}$ and obtain a
covering family of $M_{0}$, contradicting \Cref{fever}. 
\end{proof}

\begin{sublemma}
\label{joker}
Let $\mcal{F}'$ be a covering family of $M'$ and let
$F_{2}$ be a fan in $\mcal{F}'$ that contains $e_{2}$.
Then $e_{2}$ is a terminal rim element of $F_{2}$.
\end{sublemma}

\begin{proof}
Assume that $e_{2}$ is in a triangle, $T$, contained in $F_{2}$.
If $e_{3}$ is not in $T$, then $M'/e_{2}\ba e_{3}$ contains a parallel pair,
and this is impossible as $M'/e_{2}\ba e_{3}$ is
$3$\dash connected with at least four elements in
its ground set.
Therefore $e_{2},e_{3}\in T$.
Orthogonality with the triad $\{e_{2},e_{1},e_{4}\}$
shows that $e_{1}\in T$ or $e_{4}\in T$.
The former case is not true by \ref{sniff}, so $T=\{e_{2},e_{3},e_{4}\}$.
As $\{e_{1},e_{2},e_{3}\}$ is also a triangle,
$\{e_{1},e_{2},e_{3},e_{4}\}$ is a $U_{2,4}$\dash restriction
of $M'$ that intersects the triad $\{e_{2},e_{1},e_{4}\}$,
contradicting \Cref{putty}.
Therefore $e_{2}$ is contained in
no triangle in $F_{2}$.
Note that $F_{2}$ contains a fan in $\mcal{F}_{N}$,
as well as the element $e_{2}$, so it contains
at least four elements.
Therefore $e_{2}$ is a terminal rim element in $F_{2}$, as desired.
\end{proof}

\begin{sublemma}
\label{ghost}
There is a covering family, $\mcal{F}'$, of
$M'$, and a fan $F'\in\mcal{F}'$, such that
$e_{2}$ and $e_{3}$ are consecutive elements
in $F'$.
\end{sublemma}

\begin{proof}
Assume that this is not true.
Let $\mcal{F}'$ be an arbitrary covering family of $M'$.
Let $F_{2}=(f_{1},\ldots, f_{m})$ be a fan in $\mcal{F}'$
where $e_{2}=f_{1}$ and $\{f_{1},f_{2},f_{3}\}$ is a triad.
Orthogonality with the triangle $\{e_{1},e_{2},e_{3}\}$ shows that
$e_{3}$ or $e_{1}$ is in $\{f_{2},f_{3}\}$.
Assume that $e_{3}$ is in $\{f_{2},f_{3}\}$, so
$e_{1}\notin F_{2}$ by \ref{sniff}.
Because we have assumed that $e_{2}$ and $e_{3}$
are not consecutive in $F_{2}$, we deduce that $e_{3}=f_{3}$.
Now $m\leq 4$, or else we have a contradiction to
orthogonality between
$\{e_{1},e_{2},e_{3}\}=\{e_{1},f_{1},f_{3}\}$ and
the triad $\{f_{3},f_{4},f_{5}\}$.
However, $F_{2}$ contains a fan in $\mcal{F}_{N}$
as well as the elements $e_{2}$ and $e_{3}$, so $m\geq 5$.
This contradiction shows that $e_{1}$ is in $\{f_{2},f_{3}\}$,
and therefore $e_{3}\notin F_{2}$.

Now $\{f_{1},f_{2},f_{3}\}=\{e_{1},e_{2},z\}$ for some
element $z$, where $z\ne e_{3}$.
If $z\ne e_{4}$, then $\{e_{2},e_{1},e_{4},z\}$ is a
$U_{2,4}$\dash corestriction of $M'$ that intersects the
triangle $\{e_{1},e_{2},e_{3}\}$,
contradicting the dual of \Cref{putty}.
Therefore $z=e_{4}$, so
$\{f_{1},f_{2},f_{3}\}=\{e_{1},e_{2},e_{4}\}$.
If $(f_{1},f_{2},f_{3})=(e_{2},e_{1},e_{4})$,
then $F_{2}+e_{3}=(e_{3},f_{1},\ldots, f_{m})$ is a fan of $M'$.
Let $F_{3}$ be the fan in $\mcal{F}'$ that contains $e_{3}$.
Then $F_{3}$ contains a fan in $\mcal{F}_{N}$
as well as $e_{3}$, so $|F_{3}|\geq 4$.
If $e_{3}$ is an internal element in $F_{3}$,
then it is contained in a triad that is contained in
$F_{3}$, and such a triad violates orthogonality with
$\{e_{1},e_{2},e_{3}\}$.
Therefore $e_{3}$ is a terminal spoke element of $F_{3}$.
Now we easily see that
$(\mcal{F}'-\{F_{2},F_{3}\})\cup\{F_{2}+e_{3},F_{3}-e_{3}\}$
is a covering family of $M'$, and $e_{3}$ and $e_{2}$
are consecutive in $F_{2}+e_{3}$.
This contradicts our assumption, so
$(f_{1},f_{2},f_{3})=(e_{2},e_{4},e_{1})$.

Now $m\leq 4$, or else we have a contradiction to
orthogonality between the triangle $\{e_{1},e_{2},e_{3}\}$ and
the triad $\{f_{3},f_{4},f_{5}\}$
(recall $e_{3}$ is not in $F_{2}$).
In fact, $m=4$, since $F_{2}$ contains one of the fans
in $\mcal{F}_{N}$, as well as $e_{2}$.
Therefore
$F_{2}=(f_{1},f_{2},f_{3},f_{4})=(e_{2},e_{4},e_{1},f_{4})$,
and $(e_{4},e_{1},f_{4})$ or $(f_{4},e_{1},e_{4})$ is in
$\mcal{F}_{N}$.
However,
$(e_{3},f_{1},f_{3},f_{2},f_{4})= (e_{3},e_{2},e_{1},e_{4},f_{4})$
is also a fan of $M'$, so $(e_{1},e_{2},e_{3},e_{4},f_{4})$ is a
fan of $M_{0}$.
Now $\{e_{1},e_{2},e_{3}\}$ is a triangle, and
$N$ is a minor of $M_{0}/e_{2}\ba e_{3}$.
By \Cref{grape}, $(e_{1},e_{4},f_{4})$ or
$(f_{4},e_{4},e_{1})$ is in $\mcal{F}_{N}$.
Thus $\mcal{F}_{N}$ contains two distinct fans that
are non-disjoint.
This is a contradiction.
\end{proof}

By \Cref{joker} and \Cref{ghost} we can let $\mcal{F}'$ be a covering family of
$M'$ and let $F'=(f_{1},\ldots, f_{m})$ in $\mcal{F}'$ be such that
$(f_{1},f_{2})=(e_{2},e_{3})$ where $m\geq 5$ and $\{f_{1},f_{2},f_{3}\}$
is a triad.
Now \Cref{sniff} implies $e_{1}\notin F'$.
Observe that $F'+e_{1}=(e_{1},f_{1},f_{2},\ldots,f_{m})$
is a fan of $M'$.
By \ref{sniff}, $(\mcal{F}'-\{F'\})\cup\{F'+e_{1}\}$ cannot be
a covering family in $M'$.
Therefore there is a fan,
$G=(g_{1},\ldots, g_{t})$, in $\mcal{F}'-\{F'\}$, that
contains $e_{1}$.
Orthogonality with $\{e_{1},e_{2},e_{3}\}$ shows that
$e_{1}$ is not contained in a triad that is contained in $G$.
Let $G_{N}$ be the fan in $\mcal{F}_{N}$ that is consistent
with $G$.
If $e_{1}$ is not in $G_{N}$, then $G$ contains at
least four elements, so $e_{1}$ must be a terminal
element of $G$.
In this case, it is easy to see that
$(\mcal{F}'-\{F',G\})\cup\{F'+e_{1},G-e_{1}\}$
is a covering family of $M'$, which leads to a contradiction
with \ref{sniff}.
Therefore $e_{1}$ is contained in $G_{N}$.

Assume that $t=3$.
Then, by reversing, we can assume that $(g_{1},g_{2},g_{3})$
is in $\mcal{F}_{N}$.
Orthogonality with $\{e_{1},e_{2},e_{3}\}$
implies $\{g_{1},g_{2},g_{3}\}$ is a triangle.
But $\{e_{2},e_{1},e_{4}\}$ is a triad in $M'$, so orthogonality requires
that $e_{4}$ is in $\{g_{1},g_{2},g_{3}\}$.
Let $g$ be the element in $\{g_{1},g_{2},g_{3}\}-\{e_{1},e_{4}\}$.
Then $(e_{3},e_{2},e_{1},e_{4},g)$ is a fan in $M'$, so
$(e_{1},e_{2},e_{3},e_{4},g)$ is a fan in $M_{0}$.
As  $N$ is a minor of $M_{0}/e_{2}\ba e_{3}$,
\Cref{grape} implies that, up to reversing
$(e_{1},e_{4},g)$ is in $\mcal{F}_{N}$, so
$(e_{1},e_{4},g)=(g_{1},g_{2},g_{3})$.
This means that $(e_{3},e_{2},e_{1},e_{4},g)$
is a fan in $M'$ that is consistent with
$(g_{1},g_{2},g_{3})$, and
\[
(\mcal{F}'-\{F',(e_{1},e_{4},g)\})
\cup
\{F'-\{f_{1},f_{2}\},(e_{3},e_{2},e_{1},e_{4},g)\}
\]
is a covering family in $M'$, contradicting \ref{sniff}.
Therefore $t\geq 4$.

Because $e_{1}$ is not contained in any triad that is contained
in $G$, we see that $e_{1}$ is a
terminal spoke element in $(g_{1},\ldots, g_{t})$.
By reversing as necessary, we can assume that
$e_{1}=g_{1}$ and $\{g_{1},g_{2},g_{3}\}$ is a triangle.

Orthogonality between $\{g_{1},g_{2},g_{3}\}$ and
$\{e_{1},e_{2},e_{4}\}$ requires that
$e_{4}$ is in $\{g_{2},g_{3}\}$.
Assume that $e_{4}=g_{2}$.
Then
\[G'=(e_{3},e_{2},g_{1},g_{2},g_{3},\ldots,g_{t})
=(e_{3},e_{2},e_{1},e_{4},g_{3},\ldots, g_{t})\]
is a fan in $M'$.
It is easy to see that $(\mcal{F}'-\{F',G\})\cup\{F'-\{f_{1},f_{2}\},G'\}$
is a covering family in $M'$, contradicting \ref{sniff}.
Therefore $g_{3}=e_{4}$.

Next we observe that $t=4$, for otherwise
$\{g_{3},g_{4},g_{5}\}$ is a triangle and
$\{e_{1},e_{2},e_{4}\}$ is a triad that intersects
it in the element $e_{4}=g_{3}$.
Now $G_{N}$ is a subsequence of at least three elements from
$(g_{1},g_{2},g_{3},g_{4})$ or
$(g_{4},g_{3},g_{2},g_{1})$ that contains $e_{1}$.

Assume that $g_{4}$ is not in $G_{N}$.
Then $G_{N}$ is equal, up to reversing,
to $(g_{1},g_{2},g_{3})$.
Since
$(e_{3},e_{2},g_{1},g_{3},g_{2})
=(e_{3},e_{2},e_{1},e_{4},g_{2})$
is a fan in $M'$,
$(e_{1},e_{2},e_{3},e_{4},g_{2})$
is a fan in $M_{0}$.
As $N$ is a minor of $M_{0}/e_{2}\ba e_{3}$, we see from
\Cref{grape} that $(e_{1},e_{4},g_{2})=(g_{1},g_{3},g_{2})$,
or its reversal, is in $\mcal{F}_{N}$.
Therefore $\mcal{F}_{N}$ contains two distinct
fans of three elements that are non-disjoint.
This contradiction shows that $g_{4}$ is in $G_{N}$.

Assume that exactly one of $g_{2}$ and $g_{3}$ is in
$G_{N}$.
Then $G_{N}$ is consistent with the fan
$G'=(e_{3},e_{2},g_{1},g_{3},g_{2},g_{4})=
(e_{3},e_{2},e_{1},e_{4},g_{2},g_{4})$
in $M'$, and
\[
(\mcal{F}'-\{F',G\})\cup\{F'-\{f_{1},f_{2}\},G'\}
\]
is a covering family of $M'$, violating \ref{sniff}. 
We conclude that $G_{N}$ is equal, up to reversing,
to $(g_{1},g_{2},g_{3},g_{4})$.

Now
$(e_{3},e_{2},g_{1},g_{3},g_{2},g_{4})
=(e_{3},e_{2},e_{1},e_{4},g_{2},g_{4})$
is a fan in $M'$, so
$(e_{1},e_{2},e_{3},e_{4},g_{2},g_{4})$
is a fan in $M_{0}$, and some ordering of
$\{e_{1},e_{4},g_{2},g_{4}\}$ is in $\mcal{F}_{N}$.
By \Cref{grape},
$(e_{1},e_{4},g_{2},g_{4})=(g_{1},g_{3},g_{2},g_{4})$
or its reversal is in $\mcal{F}_{N}$.
This shows that $\mcal{F}_{N}$ contains two distinct
fans that are non-disjoint.
We have a contradiction that completes the proof of the lemma.
\end{proof}

Now we have some control over the ways in which we
may permute the element labels on $M$.
The next result shows a consequence of this:
we can assume that there is an element in $M$ whose
removal is $3$\dash connected with $N$ as a minor.

\begin{lemma}
\label{bliss}
There exists a matroid $M_{0}$, isomorphic to $M$, such that
$N$ is a minor of $M_{0}$, but $M_{0}$ is not a fan-extension of $N$.
Moreover, $M_{0}$ has a $3$\dash connected single-element deletion
or contraction that has $N$ as a minor.
\end{lemma}

\begin{proof}
We will assume that the lemma is false, so that no
such matroid $M_{0}$ exists.
Let $e$ be an element in $E(M)-E(N)$.
We use \Cref{tooth}.
By possibly replacing $M$, $N$, and \mcal{M}
with their duals, we can assume that
$(e_{1},\ldots, e_{n})$ is a fan of $M$ such that $N$ is a
minor of $M/e_{i}\ba e_{i+1}$ for some rim element
$e_{i}$, where $i\in \{2,\ldots, n-2\}$.
Moreover, $M/e_{i}\ba e_{i+1}$ is $3$\dash connected.
We assume that amongst all such fans, $(e_{1},\ldots, e_{n})$ has been
chosen so that $n$ is as large as possible.

\begin{sublemma}
\label{saint}
If $e_{1}$ is a spoke element of $(e_{1},\ldots, e_{n})$, then
$M\ba e_{1}$ is $3$\dash connected, and if $e_{1}$ is a
rim element, then $M/e_{1}$ is $3$\dash connected.
\end{sublemma}

\begin{proof}
Assume $e_{1}$ is a rim element.
Then $n>4$, for otherwise $i=2$, and both $e_{1}$
and $e_{2}$ are rim elements.
If $M/e_{1}$ is not $3$\dash connected, then
$e_{1}$ is in a triangle.
It follows from \Cref{steal} that this triangle is not contained
in $\{e_{1},\ldots, e_{n}\}$.
Let $z\notin \{e_{1},\ldots, e_{n}\}$ be an element in a
triangle with $e_{1}$.
Orthogonality with $\{e_{1},e_{2},e_{3}\}$ and
$\{e_{3},e_{4},e_{5}\}$ shows $\{e_{1},e_{2},z\}$ is a triangle,
and $(z,e_{1},\ldots, e_{n})$ is a fan.
This contradicts the maximality of $n$.
Similarly, if $e_{1}$ is a spoke element and
$M\ba e_{1}$ is not $3$\dash connected, then
$e_{1}$ is in a triad with an element $z\notin\{e_{1},\ldots, e_{n}\}$,
and either $(z,e_{1},\ldots, e_{n})$ is a fan,
or $n=4$, and $\{e_{1},e_{3},z\}$ is a triad.
In the first case we have a contradiction to the
maximality of $n$.
In the latter case, $n=4$ and $i=2$, so
$N$ is a minor of $M/e_{2}\ba e_{3}$, but
$\{e_{1},e_{3},z\}$ is a triad that contains
the parallel pair $\{e_{1},e_{3}\}$ in $M/e_{2}$,
contradicting \Cref{lasso}.
\end{proof}

Define \M{0} to be $M$.
We obtain \M{1} from \M{0} by swapping labels on $e_{i-1}$ and $e_{i+1}$.
Then $\M{1}/e_{i}\ba e_{i+1}=\M{0}/e_{i}\ba e_{i+1}$, and
by applying \Cref{flint} to $(e_{i-1},e_{i},e_{i+1},e_{i+2})$ we see that
\M{1} is not a fan-extension of $N$.
Let \M{2} be obtained from \M{1} by swapping the labels on $e_{i-2}$ and
$e_{i}$.
Then $\M{2}/e_{i}\ba e_{i+1}=\M{1}/e_{i}\ba e_{i+1}$, and
the dual of \Cref{flint} implies \M{2} is not a
fan-extension of $N$.
In general, when $j\in \{1,\ldots, i-1\}$, we obtain
\M{j}\ from \M{j-1} by swapping labels on
$e_{i-j}$ and $e_{i+1}$ if $j$ is odd, and on
$e_{i-j}$ and $e_{i}$ if $j$ is even.
An obvious inductive argument establishes the
following statement.

\begin{sublemma}
\label{ratio}
For every $j\in \{0,1,\ldots, i-1\}$,
$\M{j}/e_{i}\ba e_{i+1}=M/e_{i}\ba e_{i+1}$,
so $\M{j}/e_{i}\ba e_{i+1}$ has $N$ as a minor.
Moreover, \M{j} is not a fan-extension of $N$.
\end{sublemma}

Note that \M{i-1}\ is obtained from $M$ by relabeling
$(e_{1},\ldots, e_{n})$ as
\[
\begin{cases}
(e_{i},e_{i+1},e_{1},\ldots, e_{i-1},e_{i+2},\ldots, e_{n})
\qquad\text{if}\ i\ \text{is odd}\\
(e_{i+1},e_{i},e_{1},\ldots, e_{i-1},e_{i+2},\ldots, e_{n})
\qquad\text{if}\ i\ \text{is even}.
\end{cases}
\]

Assume $i$ is odd, so $e_{1}$ is a rim element
of $(e_{1},\ldots, e_{n})$ in $M$.
There is an isomorphism from $M$ to \M{i-1}\ that relabels
$e_{1}$ as $e_{i}$.
Therefore $\M{i-1}/e_{i}$ is $3$\dash connected, by
 \ref{saint}.
Similarly, if $i$ is even, then $e_{1}$ is a
spoke element and $\M{i-1}\ba e_{i+1}$ is $3$\dash connected.
In either case, \M{i-1} is isomorphic to $M$,
and has $N$ as a minor, but is not a fan-extension of
$N$.
Since \M{i-1} has a $3$\dash connected
single-element deletion or
contraction that has $N$ as a minor, the lemma
is proved.
\end{proof}

Now we can assume $M$ has a $3$\dash connected single-element
deletion or contraction with $N$ as a minor.
\Cref{flint} considered swapping labels on elements that
belonged to a fan of $M$.
In the next lemma we swap labels on elements that
belong to a fan in a $3$\dash connected single-element
deletion.

\begin{lemma}
\label{flame}
Let $M_{0}$ be isomorphic to $M$, and assume that $M_{0}$ has $N$
as a minor, but is not a fan-extension of $N$.
Assume $M_{0}\ba e$ is $3$\dash connected
and has $(e_{1},e_{2},e_{3},e_{4})$ as a fan.
Moreover, assume that either:
\begin{enumerate}[label=\textup{(\roman*)}]
 \item $e_{2}$ is a rim element and $N$ is a minor of
$M_{0}\ba e/e_{2}\ba e_{3}$, or
\item $e_{2}$ is a spoke element, and $N$ is a minor of
$M_{0}\ba e\ba e_{2}/e_{3}$.
\end{enumerate}
Let $M'$ be obtained from $M_{0}$ by swapping the labels on
$e_{1}$ and $e_{3}$.
Then $M'$ has $N$ as a minor, but is not a fan-extension of $N$.
\end{lemma}

\begin{proof}
Note that in case (i), $M'\ba e/e_{2}\ba e_{3}=M_{0}\ba e/e_{2}\ba e_{3}$,
and in case (ii), $M'\ba e\ba e_{2}/e_{3}=M_{0}\ba e\ba e_{2}/e_{3}$, so
$M'$ certainly has $N$ as a minor.
If $(e_{1},e_{2},e_{3},e_{4})$ is a fan of $M_{0}$, then we
could apply \Cref{flint} or its dual to
$(e_{1},e_{2},e_{3},e_{4})$,
and deduce that $M'$ is not a fan-extension of $N$.
In this case there is nothing left to prove, so we assume that
$(e_{1},e_{2},e_{3},e_{4})$ is not a fan of $M_{0}$.
Thus, if statement (i) holds, $\{e,e_{2},e_{3},e_{4}\}$ is a
cocircuit of $M_{0}$.
In this case we set $(x_{1},x_{2},x_{3},x_{4})$ to be
$(e_{1},e_{2},e_{3},e_{4})$.
If statement (ii) holds, $\{e,e_{1},e_{2},e_{3}\}$ is a
cocircuit, and in this case we set $(x_{1},x_{2},x_{3},x_{4})$
to be $(e_{4},e_{3},e_{2},e_{1})$.
In either case, $(x_{1},x_{2},x_{3},x_{4})$ is a fan of
$M_{0}\ba e$ with $x_{1}$ as a spoke element, $N$ is a minor
of $M_{0}\ba e/x_{2}\ba x_{3}$, and
$\{e,x_{2},x_{3},x_{4}\}$ is a cocircuit of $M_{0}$.

\begin{sublemma}
\label{viper}
$M_{0}\ba e/x_{2}\ba x_{3}$ is $3$\dash connected.
\end{sublemma}

\begin{proof}
This follows by applying \Cref{shrub} to $M_{0}\ba e$.
\end{proof}

\begin{sublemma}
\label{cable}
$M_{0}\ba x_{3}$ is $3$\dash connected.
\end{sublemma}

\begin{proof}
Assume otherwise.
Then $x_{3}$ is contained in a triad, $T^{*}$, of $M_{0}$.
As $M_{0}\ba e$ is $3$\dash connected, $T^{*}$ is also a
triad in $M_{0}\ba e$.
\Cref{lasso} and orthogonality with $\{x_{1},x_{2},x_{3}\}$
shows that $T^{*}$ contains exactly one of $x_{1}$ or $x_{2}$.
If $x_{1}$ is in $T^{*}$, then $T^{*}$ is a dependent triad
in $M_{0}/x_{2}$.
Since $N$ is a minor of $M_{0}/x_{2}$ this contradicts \Cref{lasso}.
Therefore $x_{2}, x_{3}\in T^{*}$.
Note $T^{*}\ne \{x_{2},x_{3},x_{4}\}$, as
$\{e,x_{2},x_{3},x_{4}\}$ is a cocircuit of $M_{0}$.
Therefore $T^{*} \cup x_{4}$ is a $U_{2,4}$\dash corestriction in
$M_{0}\ba e$ that intersects the triangle $\{x_{1},x_{2},x_{3}\}$.
This contradiction to the dual of \Cref{putty} completes
the proof.
\end{proof}

We will assume for a contradiction that $M'$ is a fan-extension of $N$.
Then there is a covering family of $M'$ containing a fan, $F_{e}$,
that contains $e$.
As $M'\ba e$ is $3$\dash connected (since it is isomorphic to
$M_{0}\ba e$) and $F_{e}$ contains at least four elements,
it follows that $e$ is a terminal spoke element of a fan in $M_{0}$
that has at least four elements.
Let $T_{e}$ be the triangle in this fan that contains $e$.
Orthogonality with $\{e,x_{2},x_{3},x_{4}\}$ shows that $T_{e}$
contains $x_{2}$, $x_{3}$, or $x_{4}$.
As $T_{e}-e$ is contained in a triad of $M_{0}$,
it follows from \ref{cable} that $x_{3}$ is not in $T_{e}$.
Assume that $x_{2}$ is in $T_{e}$.
Orthogonality between $\{x_{1},x_{2},x_{3}\}$ and the triad
containing $T_{e}-e$ shows that $x_{1}$ is contained in the triad.
However, $x_{1}$ is not in $T_{e}$, for that would imply
$T_{e}\cup x_{3}$ is a $U_{2,4}$\dash restriction of $M_{0}$ that
intersects a triad, a contradiction.
Now we can swap the labels on $x_{1}$ and $x_{3}$, and delete $x_{3}$.
The resulting matroid has $N$ as a minor, since $x_{1}$ and
$x_{3}$ are parallel in $M_{0}/x_{2}$, and contains the
codependent triangle $T_{e}$.
Because this is a contradiction, we conclude that $x_{2}$ is not
in $T_{e}$, so $x_{4}\in T_{e}$.
Let $z$ be the third element of $T_{e}$, so that $\{e,x_{4},z\}$
is a triangle in $M_{0}$, and $\{x_{4},z\}$ is contained
in a triad, $\{x_{4},z,z'\}$, of $M_{0}$.

\begin{sublemma}
\label{virus}
The elements in $\{e,x_{1},x_{2},x_{3},x_{4},z,z'\}$ are pairwise
distinct.
\end{sublemma}

\begin{proof}
By the hypotheses of the lemma, $e$, $x_{1}$, $x_{2}$, $x_{3}$,
and $x_{4}$ are pairwise distinct.
Also, $e$, $x_{4}$, $z$, and $z'$ are distinct members of a fan.
If $z=x_{1}$, then orthogonality between $\{x_{1},x_{2},x_{3}\}$
and $\{x_{4},z,z'\}$ shows that $z'\in\{x_{2},x_{3}\}$.
Then $\{x_{4},z,z'\}$ and $\{x_{2},x_{3},x_{4}\}$ are
distinct triads of $M_{0}\ba e$ that intersect in two elements.
Hence $M_{0}\ba e$ has a $U_{2,4}$\dash corestriction intersecting
a triangle, a contradiction.
If $z=x_{2}$, then $\{e,x_{2},x_{4}\}$ is a triangle and a triad
in $M_{0}\ba x_{3}$, which is impossible by \ref{cable}.
Finally, if $z=x_{3}$, then $M_{0}\ba x_{3}$ contains the
series pair $\{x_{4},z'\}$, contradicting \ref{cable}.
This shows that $e$, $x_{1}$, $x_{2}$, $x_{3}$, $x_{4}$, and $z$
are pairwise distinct.
If $z'\in \{x_{1},x_{2},x_{3}\}$, then we contradict
orthogonality with $\{x_{4},z,z'\}$.
This completes the proof.
\end{proof}

Now we subdivide into two cases.

\textbf{Case 1}.
Statement (i) holds, so $M'$ is obtained from $M_{0}$ by swapping the
labels on $x_{1}$ and $x_{3}$.
Since $M'$ is a fan-extension of $N$, there is a covering family,
$\mcal{F}'$, of $M'$ containing a fan, $F_{2}$, that contains $x_{2}$.
Let $F_{2}=(f_{1},\ldots, f_{m})$
Note that $m\geq 4$.

Assume that $x_{2}$ is contained in a triangle, $T_{2}$, that
is contained in $F_{2}$.
Suppose $x_{3}$ is not in $T_{2}$.
From \ref{viper} we see that $M'\ba e/x_{2}\ba x_{3}$ 
contains no parallel pair, so $e$ is in $T_{2}$.
This means $m\geq 5$, as $F_{2}$ contains $e$, $x_{2}$,
and at least three elements of $E(N)$.
Because $M'\ba e$ is 3\dash connected, $e$ is contained in no
triad of $M'$.
Therefore we can assume that $e=f_{1}$, and $f_{1}$ is a spoke
element of $F_{2}$.
Thus $T_{2}=\{f_{1},f_{2},f_{3}\}$.
As $N$ is a minor of $M'/x_{2}\ba e$, \Cref{metro}
implies $e$ and $x_{2}$ are consecutive in $F_{2}$,
so $x_{2}=f_{2}$.
\Cref{shrub} applied to $(f_{m},\ldots, f_{1})$ implies that
$M'/f_{2}\ba f_{1}=M'/x_{2}\ba e$ is $3$\dash connected.
This is impossible, as this matroid contains the parallel
pair $\{x_{1},x_{3}\}$.
Therefore we conclude that $x_{3}$ is in $T_{2}$.

If $x_{1}$ is not in $T_{2}$, then $T_{2}\cup x_{1}$ is a
$U_{2,4}$\dash restriction of $M'\ba e$, and $\{x_{2},x_{1},x_{4}\}$
is a triad.
This contradiction shows that $T_{2}=\{x_{1},x_{2},x_{3}\}$.
Now $\{x_{1},x_{2},x_{3}\}$ form a set of three consecutive
elements in $F_{2}$, by \Cref{steal}.
Let $F_{2}'$ be the fan obtained from $F_{2}$ by swapping
$x_{1}$ and $x_{3}$.
As $x_{2}$ and $x_{3}$ are not in $E(N)$, it is clear that
$(\mcal{F}'-\{F_{2}\})\cup\{F_{2}'\}$ is a covering
family of $M_{0}$.
This is a contradiction to \Cref{fever}, so we have to
conclude that $x_{2}$ is contained in no triangle in
$F_{2}$.
Therefore we can assume that $x_{2}=f_{1}$, and this is a
rim element of $F_{2}$.

By orthogonality with $\{x_{1},x_{2},x_{3}\}$, and the fact that
$M'\ba x_{1}$ is $3$\dash connected (by \Cref{cable}), we see
that $\{f_{1},f_{2},f_{3}\}$ contains $x_{3}$, so $m\geq 5$.
The dual of \Cref{metro} shows that $x_{2}$ and $x_{3}$
are consecutive elements in $F_{2}$, so $x_{3}=f_{2}$.
If $x_{1}$ is in $F_{2}$, then we can apply \Cref{steal}
to deduce that $x_{1}=f_{3}$.
This implies $\{f_{1},f_{2},f_{3}\}$ is simultaneously
a triad and a triangle of $M'$, which is impossible.
Hence $x_{1}$ is not in $F_{2}$.
Let $(F_{2}+x_{1})'$ be the fan of $M_{0}$ obtained by appending
$x_{1}$ to the end of $F_{2}$ and then swapping the locations of
$x_{1}$ and $x_{3}$.
Since $(\mcal{F}'-\{F_{2}\})\cup (F_{2}+x_{1})'$ cannot be a
covering family in $M_{0}$, it follows that $x_{1}$ is in another
fan, $F_{1}$, belonging to $\mcal{F}'$.
Orthogonality with $\{x_{1},x_{2},x_{3}\}$ shows that $x_{1}$ is a
terminal spoke element in $F_{1}$.
Because $(\mcal{F}'-\{F_{1},F_{2}\})\cup \{F_{1}-x_{1},(F_{2}+x_{1})'\}$
is not a covering family of $M_{0}$, we deduce that
$x_{1}$ is in the fan of $\mcal{F}_{N}$ that is consistent with $F_{1}$.

Let $T_{1}$ be the triangle in $F_{1}$ that contains $x_{1}$.
Orthogonality with $\{e,x_{2},x_{1},x_{4}\}$, implies that
$T_{1}$ contains $e$ or $x_{4}$.
If $e\in T_{1}$, then $|F_{1}|\geq 4$, as $e$ is not in $E(N)$.
In this case, $e$ is in a triad that is contained in $F_{1}$,
which is a contradiction, as $M'\ba e$ is $3$\dash connected.
Therefore $T_{1}$ contains $x_{4}$.
Recall that $\{e,x_{4},z\}$ is a triangle and $\{x_{4},z,z'\}$ is
a triad of $M_{0}$, and hence of $M'$, for some element $z'$.
Assume $T_{1}$ contains $z$.
This implies that $\{e,x_{1},x_{4},z\}$ is a $U_{2,4}$\dash restriction
of $M'$ that meets the triad $\{x_{4},z,z'\}$, which is impossible.
Therefore $T_{1}$ contains $z'$.
By swapping labels on $x_{1}$ and $x_{3}$ in $M'$, we see that
$\{x_{3},x_{4},z'\}$ is a triangle of $M_{0}$.
Let $M_{1}$ be the matroid obtained from $M_{0}$ by swapping
labels on $x_{2}$ and $x_{4}$.
As $\{x_{2},x_{4}\}$ is a series pair in $M_{0}\ba e\ba x_{3}$,
it follows that $M_{1}\ba e\ba x_{3}/x_{2}=M_{0}\ba e\ba x_{3}/x_{2}$,
so $M_{1}\ba e\ba x_{3}/x_{2}$ has $N$ as a minor.
Now $\{x_{3},x_{2},z'\}$ is a triangle of $M_{1}$, so $\{x_{3},z'\}$
is a parallel pair in $M_{1}/x_{2}$.
We let $M_{2}$ be the matroid obtained from $M_{1}$ by swapping
labels on $x_{3}$ and $z'$.
Then $M_{2}/x_{2}\ba x_{3}\ba e=M_{1}/x_{2}\ba x_{3}\ba e$, and hence
$M_{2}/x_{2}\ba x_{3}\ba e$ has $N$ as a minor.
However, $\{e,x_{2},z\}$ is a triangle in $M_{2}\ba x_{3}$,
and $\{x_{2},z\}$ is a series pair.
This contradiction to \Cref{lasso} competes the analysis of Case~1.

\textbf{Case 2}.
Statement (ii) holds, so $M'$ is obtained from $M_{0}$ by swapping
the labels on $x_{2}$ and $x_{4}$.
Let $\mcal{F}'$ be a covering family of $M'$, and let
$F_{3} \in \mcal{F}'$ be the fan that contains $x_{3}$.
Note that $x_{3}$ is a terminal spoke element of $F_{3}$ as
$M'\ba x_{3}$ is $3$\dash connected.
Let $T_{3}$ be the triangle in $F_{3}$ that contains $x_{3}$.
Orthogonality with the cocircuit $\{e,x_{2},x_{3},x_{4}\}$
shows that $e$, $x_{2}$, or $x_{4}$ is in $T_{3}$.
Since $M'\ba e$ is $3$\dash connected, and $T_{3}-x_{3}$ is
contained in a triad, $e$ is not in $T_{3}$.
Assume $x_{2}$ is in $T_{3}$.
Since $\{x_{2},z,z'\}$ is a triad of $M'$, either $z$ or $z'$
is in $T_{3}$.
In the former case, $\{e,x_{2},x_{3},z\}$ is a
$U_{2,4}$\dash restriction that intersects a triad, which
is a contradiction.
Hence $T_{3}=\{x_{2},x_{3},z'\}$.
As $M'/x_{2}$ contains the parallel pair $\{x_{3},z'\}$ and
$N$ as a minor, we can swap labels on $x_{3}$ and $z'$ in $M'$
and then delete $x_{3}$.
The resulting matroid contains $N$ as a minor,
a triangle $\{e,x_{2},z\}$, and a series pair $\{x_{2},z\}$.
This contradicts \Cref{lasso}, so we deduce that
$x_{4}$ is in $T_{3}$.
Now $T_{3}=\{x_{1},x_{4},x_{3}\}$, for otherwise $T_{3}\cup x_{1}$
is a $U_{2,4}$\dash restriction of $M'\ba e$ that intersects a
triad.
Let $F_{3}=(f_{1},\ldots,f_{s})$.
By reversing, we can assume that $f_{s}=x_{3}$.

\Cref{steal} implies that $\{x_{1},x_{4},x_{3}\}=\{f_{s-2},f_{s-1},f_{s}\}$.
Assume that $x_{4}\ne f_{s-1}$.
If $s\geq 5$, then we
let $F_{3}'$ be the fan of $M_{0}$ obtained from $F_{3}$
by relabeling $x_{4}$ with $x_{2}$.
As $N$ is a minor of $M_{0}/x_{2}\ba x_{3}$, and
$\{x_{1},x_{2},x_{3}\}\subseteq F_{3}'$, \Cref{metro}
implies $x_{2}$ and $x_{3}$ are consecutive in $F_{3}'$,
which implies $x_{4}$ and $x_{3}$ are consecutive in
$F_{3}$, contrary to assumption.
Therefore $s=4$, and $F_{3}=(f_{1},x_{4},x_{1},x_{3})$.
Thus $(f_{1},x_{4},x_{1})$ or its reversal is in $\mcal{F}_{N}$.
Note that $\{f_{1},x_{1},x_{4}\}$ is a triad of $M'$, so
$\{f_{1},x_{1},x_{2}\}$ is a triad of $M_{0}$, and
$(f_{1},x_{1},x_{2},x_{3},x_{4})$ is a fan in $M_{0}\ba e$.
As $M_{0}\ba e$ is a fan-extension of $N$, it contains a
fan that has $(f_{1},x_{4},x_{1})$ as a subsequence.
This fan intersects $(f_{1},x_{1},x_{2},x_{3},x_{4})$ in at least
three elements, so it contains $\{f_{1},x_{1},x_{2},x_{3},x_{4}\}$,
by \Cref{unity}, and therefore contains
$(f_{1},x_{1},x_{2},x_{3},x_{4})$ as a contiguous subsequence,
by three applications of \Cref{steal} and its dual.
Hence $(f_{1},x_{4},x_{1})$ is not a subsequence after all.
From this contradiction we see that $x_{4}=f_{s-1}$ and
$(f_{s-2},f_{s-1},f_{s})=(x_{1},x_{4},x_{3})$.

Let $F_{N}$ be the fan in $\mcal{F}_{N}$ that is consistent
with $F_{3}$.
Let \mcal{F} be the family of fans in $M_{0}$ that is induced
from $\mcal{F}'$ by swapping the labels on $x_{2}$ and $x_{4}$.
As the fans in $\mcal{F}'$ contains all elements in
$E(M')-E(N)$, it follows that the fans in \mcal{F} contain
all elements in $E(M_{0})-E(N)$, except possibly
for $x_{2}$.
However, $x_{4}$ is in $F_{3}$, so $x_{2}$ is in a fan of
\mcal{F}.
Therefore the fans in \mcal{F} contain every element in
$E(M_{0})-E(N)$.
As $x_{2}$ does not belong to a fan of $\mcal{F}_{N}$, the
only way \mcal{F} can fail to be a covering family of $M_{0}$
is if $x_{4}$ is in $F_{N}$.
We deduce this is the case.
Then $x_{4}$ is a terminal element of $F_{N}$, because
$(f_{s-2},f_{s-1},f_{s})=(x_{1},x_{4},x_{3})$ and
$x_{3}\notin E(N)$.

We claim that $\{e,x_{2},x_{4}\}$ is the unique triad
of $M_{0}\ba x_{3}$ that contains $e$.
Suppose $T^{*}$ is another such triad.
If $x_{2}$ is in $T^{*}$, then $T^{*}\cup x_{4}$ is a
$U_{2,4}$\dash corestriction of $M_{0}\ba x_{3}$, and
$\{e,x_{4},z\}$ is a triangle that intersects it.
This is a contradiction, so $x_{2}\notin T^{*}$.
This means that $T^{*}-e$ is a series pair of
$M_{0}\ba e/x_{2}\ba x_{3}$, contradicting \ref{viper},
and proving the claim.

As $M_{0}\ba x_{3}$ is $3$\dash connected, it
contains a covering family.

\begin{sublemma}
\label{house}
There exists a covering family of $M_{0}\ba x_{3}$
containing a fan that contains $e$ and $x_{2}$.
\end{sublemma}

\begin{proof}
Let \mcal{G} be an arbitrary covering family of $M_{0}\ba x_{3}$,
and let $G=(g_{1},\ldots, g_{t})$ be the fan in \mcal{G} that contains
$e$.
If $e$ is in a triad that is contained in $G$,
then we are done, as this triad must be $\{e,x_{2},x_{4}\}$,
by the earlier claim.
Therefore we assume that $e=g_{1}$ is a terminal spoke element
of $G$, and $x_{2}$ is not in $G$.
Orthogonality between $\{e,x_{2},x_{4}\}$ and $\{g_{1},g_{2},g_{3}\}$
shows that $x_{4}\in\{g_{2},g_{3}\}$.
If $x_{4}=g_{3}$, then $t=4$, for otherwise we contradict
orthogonality between $\{e,x_{2},x_{4}\}$ and
$\{g_{3},g_{4},g_{5}\}$.
In this case $(g_{2},g_{3},g_{4})=(g_{2},x_{4},g_{4})$ or
its reversal is in $\mcal{F}_{N}$.
Recall $F_{N}$ is the fan in $\mcal{F}_{N}$ consistent with
$F_{3}$, and $x_{4}$ is in $F_{N}$.
Hence $F_{N}$ is $(g_{2},x_{4},g_{4})$ or its reversal, and this
is a contradiction, as we concluded
$x_{4}$ is a terminal element of $F_{N}$.
Therefore $x_{4}=g_{2}$, and
$G+x_{2}=(x_{2},g_{1},\ldots, g_{t})$ is a fan of $M_{0}\ba x_{3}$.
Let $G'$ be the fan in \mcal{G} that contains $x_{2}$.
Then $x_{2}$ is a terminal rim element of $G'$, by
orthogonality with $\{e,x_{2},x_{4}\}$.
Now $(\mcal{G}-\{G,G'\})\cup\{G+x_{2},G'-x_{2}\}$ is the
desired covering family of $M_{0}\ba x_{3}$.
\end{proof}

Next we claim $M_{0}\ba x_{3}/x_{2}$ is $3$\dash connected.
Note that $(f_{s-3},f_{s-2},f_{s-1},f_{s})=(f_{s-3},x_{1},x_{4},x_{3})$
is a fan of $M'$, so $(f_{s-3},x_{1},x_{2},x_{3})$ is a fan
in $M_{0}$.
Now \Cref{shrub} tells us that $M_{0}\ba x_{3}/x_{2}$ is
$3$\dash connected, as desired.

Let \mcal{G} be a covering family of $M_{0}\ba x_{3}$, and
let $G=(g_{1},\ldots, g_{t})$ be a fan in \mcal{G} that
contains $e$ and $x_{2}$.
Then $t\geq 5$.
As $M_{0}\ba x_{3}/x_{2}$ is $3$\dash connected, $x_{2}$ is
a terminal rim element of $G$.
We can assume $x_{2}=g_{1}$.
The only triad of $M_{0}\ba x_{3}$ that contains $e$ is
$\{e,x_{2},x_{4}\}$, so if $x_{4}$ is not in $G$ then $e$ is
contained in no triad in $G$, implying $e=g_{t}$ is a
spoke element.
But then $\{e,x_{2},x_{4}\}$ violates orthogonality
with $\{g_{t-2},g_{t-1},g_{t}\}$.
Hence $G$ contains $e$, $x_{2}$, and $x_{4}$.
The dual of \Cref{steal} implies
$\{e,x_{2},x_{4}\}=\{g_{1},g_{2},g_{3}\}$.
As $t\geq 5$, and $e$ is not contained in the triad
$\{g_{3},g_{4},g_{5}\}$, we see that
$(g_{1},g_{2},g_{3})=(x_{2},e,x_{4})$.

Since $G$ contains $x_{4}$ it follows that $F_{N}$
is consistent with $G$.
Now consider the fans $G$ and
$(f_{1},\ldots, f_{s-2},x_{2})$ in $M_{0}\ba x_{3}$.
They both contain $x_{2}$, and the elements of
$F_{N}-x_{4}$.
Therefore they intersect in at least three elements, so we
can apply \Cref{unity}.
This shows that the elements common to both fans
form contiguous subsequences in both.
This implies $\{g_{1},g_{2},g_{3}\}=\{x_{2},e,x_{4}\}$
is a triad contained in $(f_{1},\ldots, f_{s-2},x_{2})$.
Therefore \Cref{steal} implies
$\{x_{2},e,x_{4}\}=\{x_{2},f_{s-2},f_{s-3}\}=\{x_{2},x_{1},f_{s-3}\}$.
This is impossible as $e$, $x_{1}$, and $x_{4}$ are distinct
elements.
We have a contradiction that completes the proof of
\Cref{flame}.
\end{proof}

If $M_{0}\ba e$ is $3$\dash connected with $N$ as a minor,
where $M_{0}$ is isomorphic to $M$ and is not a fan-extension,
then $M_{0}\ba e$ must be a fan-extension, by the
minimality of $M$.
This implies $M_{0}\ba e$ has a covering family.
In the next result, we examine this covering family,
and discover that the elements in $E(M_{0}\ba e)-E(N)$
are concentrated at the ends of fans.

\begin{lemma}
\label{train}
Let $M_{0}$ be isomorphic to $M$.
Assume that $N$ is a minor of $M_{0}$, but that $M_{0}$ is not
a fan-extension of $N$.
Assume also that $M_{0}\ba e$ is
$3$\dash connected and has $N$ as a minor.
Let $\mcal{F}_{e}$ be a covering family of $M_{0}\ba e$, and
let $(e_{1},\ldots, e_{m})$ be a fan in $\mcal{F}_{e}$.
If $2<i<m-1$, then $e_{i}$ is in $E(N)$.
\end{lemma}

\begin{proof}
Assume that the \namecref{train} fails, so that
$(e_{1}',\ldots,e_{m}')$ is a fan in a covering family
of $M_{0}\ba e$ and $e_{j}'$ is in $E(M_{0}\ba e)-E(N)$,
for some $j\in \{3,\ldots, m-2\}$.
Note that $m\geq 5$.

We claim that either $e_{j-1}'$ or $e_{j+1}'$ is in
$E(M_{0}\ba e)-E(N)$.
Assume for a contradiction that $\{e_{j-1}',e_{j+1}'\}\subseteq E(N)$.
If $e_{j}'$ is a rim element in $(e_{1}',\ldots, e_{m}')$
then let $M'$ be $M_{0}\ba e$ and let $N'$ be $N$.
Otherwise let $M'$ be
$(M_{0}\ba e)^{*}$ and let $N'$ be $N^{*}$.
Thus $e_{j}'$ is a rim element in the
fan $(e_{1}',\ldots, e_{m}')$ in $M'$.
Now $N'$ is a minor of $M'\ba e_{j}'$, for if $N'$ is a
minor of $M'/e_{j}'$, then $N'$ contains the parallel pair
$\{e_{j-1}',e_{j+1}'\}$.
Therefore $m=5$ and $j=3$, because otherwise
$\{e_{j-3}',e_{j-2}',e_{j-1}'\}$ or
$\{e_{j+1}',e_{j+2}',e_{j+3}'\}$ is a codependent
triangle in $M'\ba e_{j}'$.
From $m=5$ and $j=3$, it follows that
$\{e_{1}',e_{2}',e_{4}',e_{5}'\}$ is a union
of series pairs in $M'\ba e_{j}'$, and contains
at least three elements of $E(N')$.
This leads to a contradiction to the $3$\dash connectivity of
$N'$.
Therefore $e_{j-1}'$ or $e_{j+1}'$ is in $E(M_{0}\ba e)-E(N)$,
as claimed.

By reversing $(e_{1}',\ldots, e_{m}')$
as necessary, we can assume that $e_{i}'$ and $e_{i+1}'$
are in $E(M_{0}\ba e)-E(N)$, where $1<i<m$, and
$e_{i}'$ is a rim element.
\Cref{cabal} tells us that $N$ is a minor of
$M_{0}\ba e/e_{i}'\ba e_{i+1}'$.

By applying \Cref{flame} to the fan
$(e_{i-1}',e_{i}',e_{i+1}',e_{i+2}')$, we see that we
can swap labels on $e_{i-1}'$ and $e_{i+1}'$ without
any loss of generality.
If $i-1>1$, we would then be able to apply \Cref{flame}
to the fan $(e_{i-2}',e_{i+1}',e_{i}',e_{i-1}')$, and swap labels on
$e_{i-2}'$ and $e_{i}'$.
Deleting $e$ from the resulting matroid reveals the following
fan:
\[
(e_{1},\ldots, e_{i-3}',e_{i}',e_{i+1}',e_{i-2}',
e_{i-1}',e_{i+2}',\ldots, e_{m}).
\]
In fact, by starting from
$(e_{1}',\ldots, e_{i}',e_{i+1}',\ldots, e_{m}')$, and repeatedly
applying \Cref{flame}, it is possible for us to assume that any
sequence of the form
\begin{linenomath}
\begin{multline*}
(e_{1}',\ldots, e_{j}',x,y,e_{j+1}',\ldots,
e_{i-1}',e_{i+2}',\ldots,e_{m}')\quad\text{or}\\
(e_{1}',\ldots,e_{i-1}',e_{i+2}',\ldots,e_{j}',x,y,e_{j+1}',
\ldots,e_{m}')
\end{multline*}
\end{linenomath}
is a fan of $M_{0}\ba e$.
If we have applied \Cref{flame} an even number of times
in this process, then $(x,y)$ is $(e_{i}',e_{i+1}')$,
otherwise it is $(e_{i+1}',e_{i}')$.
Note that this collection of sequences includes
$(x,y,e_{1}',\ldots,e_{i-1}',e_{i+2}',\ldots,e_{m}')$ and
$(e_{1}',\ldots,e_{i-1}',e_{i+2}',\ldots,e_{m}',x,y)$.
We remark that if $i$ had been equal to $1$ or $m-1$, then
\Cref{flame} would not have been applicable.
This is why it was necessary to show that $i$ lies between
$2$ and $m-1$.

Since there is a fan of $\mcal{F}_{N}$ that is consistent with
$(e_{1}',\ldots, e_{m}')$, by repeatedly applying \Cref{flame},
we can assume that we have the following situation:
$(e_{1},\ldots,e_{n+2})$ is a contiguous subsequence of
$(e_{1}',\ldots, e_{m}')$, and is therefore a fan of $M_{0}\ba e$.
Some fan $F\in \mcal{F}_{N}$ is consistent with $(e_{1},\ldots,e_{n})$
and contains $e_{1}$ and $e_{n}$.
Also, $e_{n+1}$ and $e_{n+2}$ are not in $E(N)$.
Let $\{x,y\}=\{e_{n+1},e_{n+2}\}$, where $x$ is a rim
element of $(e_{1},\ldots,e_{n+2})$.
Then $N$ is a minor of $M_{0}\ba e/x\ba y$
by \Cref{cabal}.

\begin{sublemma}
\label{paper}
Let $e_{i}$ be a rim element of $(e_{1},\ldots,e_{n+2})$,
where $1\leq i\leq n$.
Then $\{e_{i},e_{i+1},e_{i+2}\}$ is not a triad in $M_{0}$.
\end{sublemma}

\begin{proof}
Assume this is false.
By starting from $(e_{1}',\ldots, e_{m}')$, and using \Cref{flame}
to shift $x$ and $y$, we can, without losing any generality, assume
that $(e_{1},\ldots,e_{i-1},x,y,e_{i},\ldots,e_{n})$ is a
fan of $M_{0}\ba e$.
Here $x$ is a rim element and $\{x,y,e_{i}\}$
is a triad of $M_{0}$, while $N$ is a minor of $M_{0}\ba e/x\ba y$.
By \Cref{shrub}, $M_{0}\ba e/x\ba y$ is $3$\dash connected.
If $M_{0}/x\ba y$ is not $3$\dash connected, there is a triangle of
$M_{0}$ that contains $x$ and $e$ but not $y$.
Orthogonality with $\{x,y,e_{i}\}$ shows that
$\{e,x,e_{i}\}$ is a triangle of $M_{0}$, so $\{e,x,e_{i}\}$ is
a codependent triangle in $M_{0}\ba y$.
This contradiction shows that $M_{0}/x\ba y$ is $3$\dash connected,
and is therefore a fan-extension of $N$.

Let $F'$ be the fan in a covering family of $M_{0}/x\ba y$ such
that $F$ is consistent with $F'$.
Now $(e_{1},\ldots,e_{n})$ is a fan of $M_{0}/x\ba y\ba e$,
by \Cref{cider}, and $F$ is consistent with $(e_{1},\ldots,e_{n})$
and contains $e_{1}$ and $e_{n}$.
Note that if $e$ is in $F'$, then it is a terminal spoke element,
as $M_{0}/x\ba y\ba e$ is $3$\dash connected.
Hence $F'-e$ is a fan of $M_{0}/x\ba y\ba e$.
Because $F'-e$ and $(e_{1},\ldots,e_{n})$ have at least three
elements in common, and these elements include $e_{1}$ and $e_{n}$,
it follows from \Cref{unity} that the elements in $\{e_{1},\ldots,e_{n}\}$
form a contiguous subsequence of $F'-e$.

Assume that the elements $\{e_{1},\ldots, e_{n}\}$ do not appear
in the order $(e_{1},\ldots,e_{n})$ in $F'-e$.
Then it follows easily from \Cref{steal} that $n<5$.
Note that $F$ is consistent with $(e_{1},\ldots, e_{n})$
and with $F'$, so the elements of $F$
appear in $F'$ in the same order as they do in
$(e_{1},\ldots, e_{n})$.
Therefore our assumption means that $F$ is not
equal to $(e_{1},\ldots, e_{n})$ or its reversal, although it
contains $e_{1}$ and $e_{n}$.
If $n=3$, then $F$ is equal to $(e_{1},e_{2},e_{3})$ or its
reversal, contrary to our conclusion.
Therefore $n=4$, and $F$ is either
$(e_{1},e_{2},e_{4})$ or $(e_{1},e_{3},e_{4})$, or one of their
reversals.
Now $\{e_{1},e_{2},e_{3},e_{4}\}$ is contained in $E(N)$,
or else $N$ contains a series or parallel pair.
Therefore $(e_{1},e_{2},e_{3},e_{4})$ is a fan of $N$.
As $(e_{1},e_{2},e_{4})$ or $(e_{1},e_{3},e_{4})$ is also a fan
of $N$, we have contradicted \Cref{steal}.
Thus $(e_{1},\ldots,e_{n})$ is a contiguous subsequence of
$F'-e$, and hence of $F'$.

Let $F'$ be $(f_{1},\ldots,f_{j-1},e_{1},\ldots,e_{n},f_{j},\ldots,f_{p})$,
so this sequence is a fan in $M_{0}/x\ba y$.
We will show that
\[(f_{1},\ldots,f_{j-1},e_{1},\ldots,e_{i-1},x,y,e_{i},\ldots,e_{n},
f_{j},\ldots,f_{p})\]
is a fan in $M_{0}$.
Let this sequence be labeled
$(g_{1},\ldots, g_{k-1},x,y,g_{k},\ldots, g_{r})$, so that
$(g_{1},\ldots, g_{r})$ is a fan in $M_{0}/x\ba y$.
Note that $i\leq n$, so there is at least one element in
$\{g_{k},\ldots, g_{r}\}$.
By assumption, $\{x,y,e_{i}\}=\{x,y,g_{k}\}$ is a triad of $M_{0}$.

Assume that $r\geq k+1$.
If $g_{k+1}=e_{i+1}$, then $\{y,e_{i},e_{i+1}\}=\{y,g_{k},g_{k+1}\}$
is a triangle of $M_{0}$, since
$(e_{1},\ldots,e_{i-1},x,y,e_{i},\ldots,e_{n})$ is a fan in
$M_{0}\ba e$.
On the other hand, if $g_{k+1}=f_{j}$, then $i=n$, and
$\{e_{n-1},e_{n},f_{j}\}=\{g_{k-1},g_{k},g_{k+1}\}$ is a
triangle in $M_{0}/x\ba y$.
It cannot be a triangle in $M_{0}$ by orthogonality
with the triad $\{x,y,g_{k}\}$.
Therefore $\{g_{k-1},x,g_{k},g_{k+1}\}$ is a circuit in $M_{0}$.
Now $\{g_{k-1},x,y\}$ is a triangle of $M_{0}\ba e$, and
therefore of $M_{0}$.
Strong circuit-exchange between $\{g_{k-1},x,g_{k},g_{k+1}\}$
and $\{g_{k-1},x,y\}$ shows that there is a circuit in
$\{g_{k-1},y,g_{k},g_{k+1}\}$ that contains $y$.
Orthogonality shows that this circuit cannot contain $g_{k-1}$,
because $n\geq 3$, and $\{e_{n-2},e_{n-1},x\}=\{g_{k-2},g_{k-1},x\}$
is a triad of $M_{0}\ba e$.
Therefore $\{y,g_{k},g_{k+1}\}$ is a triangle in $M_{0}$ in any case.

Next we assume that $r\geq k+2$.
Then $\{g_{k},g_{k+1},g_{k+2}\}$ is a triad in $M_{0}/x\ba y$.
Assume that it is not a triad in $M_{0}$, so that
$\{y,g_{k},g_{k+1},g_{k+2}\}$ is a cocircuit.
If $n\geq i+2$, then
$\{e_{i},e_{i+1},e_{i+2}\}=\{g_{k},g_{k+1},g_{k+2}\}$ is a triad
in $M_{0}\ba e$, and $\{y,g_{k},g_{k+1},g_{k+2}\}$
is a union of cocircuits.
Therefore $\{y,g_{k},g_{k+1},g_{k+2}\}$ is a
$U_{2,4}$\dash corestriction of $M_{0}\ba e$.
As $\{y,g_{k},g_{k+1}\}$ is a triangle this is a contradiction
to \Cref{putty}, so $n<i+2$, meaning that $i>1$.
Now $\{e_{i-1},x,y\}=\{g_{k-1},x,y\}$ is a triangle of
$M_{0}\ba e$ that violates orthogonality with
$\{y,g_{k},g_{k+1},g_{k+2}\}$.
Hence $\{g_{k},g_{k+1},g_{k+2}\}$ is a triad in $M_{0}$.

Now we consider all the sets comprising three consecutive elements
in $(g_{k+1},\ldots,g_{r})$.
If such a set is a triangle in $M_{0}/x\ba y$, then it is a
triangle in $M_{0}$, by orthogonality with the triad
$\{x,y,g_{k}\}$.
If such a set is a triad in $M_{0}/x\ba y$, then it is a
triad in $M_{0}$ by orthogonality with the triangle
$\{y,g_{k},g_{k+1}\}$.
We have shown that $(x,y,g_{k},\ldots,g_{r})$
is a fan in $M_{0}$.

Assume that $k> 1$.
If $g_{k-1}=e_{i-1}$, then $\{g_{k-1},x,y\}$ is a triangle
in $M_{0}$, since $(e_{1},\ldots,e_{i-1},x,y,e_{i},\ldots,e_{n})$
is a fan of $M_{0}\ba e$.
Therefore we assume that $g_{k-1}=f_{j-1}$.
This means that $i=1$, which implies that
$\{g_{k},\ldots, g_{r}\}$ contains at least three elements.
Note $\{f_{j-1},e_{1},e_{2}\}=\{g_{k-1},g_{k},g_{k+1}\}$ is a
triangle in $M_{0}/x\ba y$, but not in $M_{0}$, by
orthogonality with the triad $\{x,y,g_{k}\}$.
Therefore $\{g_{k-1},x,g_{k},g_{k+1}\}$ is a circuit in $M_{0}$.
We perform strong circuit-exchange on $\{g_{k-1},x,g_{k},g_{k+1}\}$
and the triangle $\{y,g_{k},g_{k+1}\}$ to obtain a circuit
of $M_{0}$
contained in $\{g_{k-1},x,y,g_{k+1}\}$
that contains $y$.
Orthogonality with the triad $\{g_{k},g_{k+1},g_{k+2}\}$
shows that this circuit cannot contain $g_{k+1}$.
Therefore $\{g_{k-1},x,y\}$ is a triangle of $M_{0}$ in
any case.

Next we assume that $k> 2$.
We wish to show that $\{g_{k-2},g_{k-1},x\}$ is a triad
in $M_{0}$.
Now $\{g_{k-2},g_{k-1},g_{k}\}$ is a triad in $M_{0}/x\ba y$.
It is not a triad in $M_{0}$, by orthogonality with the triangle
$\{g_{k-1},x,y\}$, so $\{g_{k-2},g_{k-1},y,g_{k}\}$ is a cocircuit
of $M_{0}$.
By strong cocircuit-exchange with $\{x,y,g_{k}\}$, we find a
cocircuit, $C^{*}$, of $M_{0}$ contained in $\{g_{k-2},g_{k-1},x,g_{k}\}$
that contains $x$.
If $C^{*}$ does not contain $g_{k}$, then
$\{g_{k-2},g_{k-1},x\}$ is a triad, as desired.
Therefore we assume that $g_{k}\in C^{*}$.
Orthogonality with the triangle $\{g_{k-1},x,y\}$
shows that $g_{k-1}\in C^{*}$.
Because $\{x,y,g_{k}\}$ is a triad of $M_{0}$, it follows
that if $C^{*}=\{g_{k-1},x,g_{k}\}$, then
$C^{*}\cup y$ is a $U_{2,4}$\dash corestriction of $M_{0}$.
As $\{g_{k-1},x,y\}$ is a triangle, this contradicts \Cref{putty}.
Hence $C^{*}=\{g_{k-2},g_{k-1},x,g_{k}\}$.
It follows that $r=k$, for otherwise $C^{*}$ and $\{y,g_{k},g_{k+1}\}$
violate orthogonality.
From $r=k$, we deduce that $i=n$, and as
$n\geq 3$, it follows that $(g_{k-2},g_{k-1})=(e_{n-2},e_{n-1})$,
so $\{g_{k-2},g_{k-1},x\}$ is a triad in $M_{0}\ba e$.
Now $C^{*}$ is a union of cocircuits in $M_{0}\ba e$ that
contains the triad $\{g_{k-2},g_{k-1},x\}$, so $C^{*}$
is a $U_{2,4}$\dash corestriction in $M_{0}\ba e$.
Since $\{g_{k-1},x,y\}$ is a triangle, this is
impossible.
Therefore $\{g_{k-2},g_{k-1},x\}$ is a triad
in $M_{0}$, as we claimed.

Finally, we observe that any set of three consecutive elements
from $(g_{1},\ldots,g_{k-1})$ that is a triangle
in $M_{0}/x\ba y$ is a triangle in $M_{0}$, since
$\{x,y,g_{k}\}$ is a triad.
Any such set that is a triad in $M_{0}/x\ba y$
is a triad in $M_{0}$, as $\{g_{k-1},x,y\}$ is a triangle.

We have shown that
\begin{linenomath}
\begin{multline*}
(g_{1},\ldots, g_{k-1},x,y,g_{k},\ldots, g_{r})=\\
(f_{1},\ldots,f_{j-1},e_{1},\ldots,e_{i-1},x,y,e_{i},\ldots,e_{n},
f_{j},\ldots,f_{p})
\end{multline*}
\end{linenomath}
is a fan of $M_{0}$.
Now $M_{0}/x\ba y$ has a covering family containing
$F'=(f_{1},\ldots,f_{j-1},e_{1},\ldots,e_{n},f_{j},\ldots,f_{p})$, and
$M_{0}$ is obtained from $M_{0}/x\ba y$
by a fan-lengthening move on this fan.
Hence $M_{0}$ is a fan-extension of $N$.
This contradiction completes the proof of \ref{paper}.
\end{proof}

\begin{sublemma}
\label{brass}
If $e_{i}$ is a rim element of $(e_{1},\ldots,e_{n+2})$, and
there is a triad of $M_{0}\ba e$ contained in $\{e_{1},\ldots,e_{n+2}\}$
that does not contain $e_{i}$, then there is a triangle of $M_{0}$
containing $\{e,e_{i}\}$.
\end{sublemma}

\begin{proof}
Assume that the claim fails.
We will start from $(e_{1}',\ldots, e_{m}')$, and then use
\Cref{flame} to shift $x$ and $y$ appropriately.
By then possibly reversing, we can assume that
$(e_{1},\ldots,e_{i-1},x,y,e_{i},\ldots,e_{n})$ is a fan of
$M_{0}\ba e$, where $x$ is a rim element and $M_{0}\ba e/x\ba y$ has
$N$ as a minor;
moreover, we assume there is no triangle of $M_{0}$ containing $\{e,x\}$,
and $\{e_{j},e_{j+1},e_{j+2}\}$ is a triad of $M_{0}\ba e$,
for some $j\in \{i,\ldots,n-2\}$.
\Cref{shrub} implies that $M_{0}\ba e/x\ba y$ is $3$\dash connected.
Because there is no triangle containing $e$ and $x$ in $M_{0}$, it
follows that $M_{0}/x\ba y$ is
$3$\dash connected, and hence a fan-extension of $N$.
Let $F'$ be the fan in a covering family of $M_{0}/x\ba y$ such that
$F$ is consistent with $F'$.
If $F'$ contains $e$, then $e$ is a terminal element because
$M_{0}/x\ba y\ba e$ is $3$\dash connected, so
$F'-e$ is a fan of $M_{0}/x\ba y\ba e$.
Now $F$ is also consistent with $(e_{1},\ldots,e_{n})$, which is
a fan of $M_{0}/x\ba y\ba e$.
Since $(e_{1},\ldots,e_{n})$ and $F'-e$ have at least three
elements in common, including $e_{1}$ and $e_{n}$, it follows
from \Cref{unity} that the elements in $\{e_{1},\ldots,e_{n}\}$
form a contiguous subsequence of $F'-e$.
As $\{e_{j},e_{j+1},e_{j+2}\}$ is a triad in $M_{0}/x\ba y\ba e$,
it is a set of three consecutive elements in $F'$ by the dual of
\Cref{steal}.
It cannot be a triangle in $M_{0}/x\ba y$, for then it would
be a triad and a triangle in $M_{0}/x\ba y\ba e$.
Therefore $\{e_{j},e_{j+1},e_{j+2}\}$ is a triad in
$M_{0}/x\ba y$, but not in $M_{0}$ by \ref{paper}.
This shows that $y$ is in the coclosure of
$\{e_{i},\ldots,e_{n}\}$ in $M_{0}$, and hence in $M_{0}\ba e$.
It is in the closure of the same set because $\{y,e_{i},e_{i+1}\}$
is a triangle.
Thus $\{y,e_{i},\ldots,e_{n}\}$ is $2$\dash separating in
$M_{0}\ba e$, so the complement of $\{y,e_{i},\ldots, e_{n}\}$
contains at most one element.
\Cref{hocus} implies $M_{0}\ba e$ is a wheel or a whirl.
This contradiction completes the proof of \ref{brass}.
\end{proof}

\begin{sublemma}
\label{dance}
$n\leq 5$, and if $n=5$, then $e_{1}$ is a spoke element of
$(e_{1},\ldots, e_{n+2})$.
\end{sublemma}

\begin{proof}
If this statement is false, then, by starting from
$(e_{1}',\ldots, e_{m}')$, and applying \Cref{flame}
to shift $x$ and $y$ as necessary, we can let $(x_{1},\ldots,x_{5}, y,x)$
be either $(e_{1},\ldots, e_{5},y,x)$ or
$(e_{2},\ldots, e_{6},y,x)$, and assume that
$(x_{1},\ldots, x_{5},y,x)$ is a fan of $M_{0}\ba e$,
where $N$ is a minor of $M_{0}\ba e/x\ba y$, and
$x_{1}$ is a rim element of $(x_{1},\ldots, x_{5},y,x)$.
Since $\{x_{1},x_{2},x_{3}\}$, $\{x_{3},x_{4},x_{5}\}$,
and $\{x_{5},y,x\}$ are triads in $M_{0}\ba e$,
we apply \ref{paper} and \ref{brass} and see that
$\{e,x_{1},x_{2},x_{3}\}$, $\{e,x_{3},x_{4},x_{5}\}$,
and $\{e,x_{5},y,x\}$ are cocircuits in $M_{0}$, and that
$\{e,x_{1}\}$, $\{e,x_{3}\}$, $\{e,x_{5}\}$, and
$\{e,x\}$ are all contained in triangles of $M_{0}$.
Now it is easy to see, using orthogonality, that
$\{e,x_{1},x_{5}\}$ and $\{e,x_{3},x\}$ are triangles of $M_{0}$.
This implies that $x\in\cl_{M_{0}\ba e}(\{x_{1},\ldots, x_{5},y\})$.
Since $x$ is also in $\cl_{M_{0}\ba e}^{*}(\{x_{1},\ldots, x_{5},y\})$,
we see that $\{x_{1},\ldots, x_{5},y,x\}$ is $2$\dash separating in
$M_{0}\ba e$, and it follows that $M_{0}\ba e$ is a wheel or a
whirl.
This contradiction completes the proof.
\end{proof}

\begin{sublemma}
\label{crime}
$n\leq 4$.
\end{sublemma}

\begin{proof}
If $n>4$, then by \ref{dance}, $n=5$, and $e_{1}$ is a spoke
element.
Therefore $(e_{1},\ldots, e_{5},x,y)$ is a fan of $M_{0}\ba e$,
where $N$ is a minor of $M_{0}\ba e/x\ba y$.
As $\{e_{2},e_{3},e_{4}\}$ and $\{e_{4},e_{5},x\}$ are triads of
$M_{0}\ba e$, \ref{paper} and \ref{brass} imply that
$\{e,e_{2},e_{3},e_{4}\}$ and $\{e,e_{4},e_{5},x\}$ are cocircuits
of $M_{0}$, and $\{e,e_{2}\}$ and $\{e,x\}$ are contained
in triangles of $M_{0}$.
Assume $\{e,e_{2},x\}$ is not a triangle.
Then orthogonality requires that $\{e,e_{2}\}$ is in a triangle
with $e_{4}$ or $e_{5}$, and that $\{e,x\}$ is in a triangle with
$e_{3}$ or $e_{4}$.
This implies that $x$ is in the closure and the coclosure of
$\{e_{1},\ldots, e_{5}\}$ in $M_{0}\ba e$, and we arrive at the
contradiction that $M_{0}\ba e$ is a wheel or a whirl.
Therefore $\{e,e_{2},x\}$ is a triangle.

Assume $M_{0}\ba y$ is not $3$\dash connected.
Then there is a triad of $M_{0}$ that contains $y$.
By orthogonality with $\{e_{5},x,y\}$, $\{e_{3},e_{4},e_{5}\}$,
and $\{e,e_{2},x\}$, we see that $y$ is in the coclosure and
closure of $\{e_{1},\ldots, e_{5},x\}$ in $M_{0}$, and also in
$M_{0}\ba e$.
This leads to a contradiction as before, so $M_{0}\ba y$
is $3$\dash connected, and therefore a fan-extension of $N$.
This implies there is a fan, $F'$, of $M_{0}\ba y$
such that $F$ is consistent with $F'$ and 
$e_{1}$ and $e_{5}$ are terminal elements of $F'$.

\Cref{shrub} implies that $M_{0}\ba y\ba e/x$ is $3$\dash connected.
If neither $e$ nor $x$ is an internal element of $F'$, then
$F'-\{e,x\}$ is a fan of $M_{0}\ba y\ba e/x$ by the fact that
$M_{0}\ba y\ba e/x$ is $3$\dash connected.
Assume either $e$ or $x$ is an internal element in $F'$.
By applying \Cref{llama} to $M'=M_{0}\ba y$, we see that
$e$ and $x$ are consecutive in $F'$, and $x$ is a rim
element.
Then \Cref{cider} implies $F'-\{e,x\}$ is a fan of
$M_{0}\ba y\ba e/x$, so this is true in any case.
We note that $(e_{1},\ldots, e_{5})$ is also a fan in
$M_{0}\ba y\ba e/x$.
As $F'-\{e,x\}$ and $(e_{1},\ldots, e_{5})$ have at least three
elements in common, including $e_{1}$ and $e_{5}$,
\Cref{unity} shows that $\{e_{1},\ldots, e_{5}\}$ is contained
in $F'-\{e,x\}$.
Applying \Cref{steal} three times shows that
$(e_{1},\ldots, e_{5})$ is a contiguous subsequence of
$F'-\{e,x\}$.
Assume $(e_{2},e_{3},e_{4})$ is a contiguous subsequence
of $F'$.
Since $\{e_{2},e_{3},e_{4}\}$ is a triad in
$M_{0}\ba y\ba e/x$, it cannot be a triangle in $M_{0}\ba y$.
Hence it is a triad in $M_{0}\ba y$, but not in $M_{0}$.
This means $y$ is in the coclosure and
the closure of $\{e_{1},\ldots, e_{5},x\}$ in $M_{0}\ba e$.
This leads to a contradiction, so $(e_{2},e_{3},e_{4})$
is not a contiguous subsequence of $F'$.
This means that the elements $\{e,x,e_{3}\}$ are consecutive
in $F'$, so $\{e,x,e_{3}\}$ is a triad or a triangle in
$M_{0}\ba y$.
The former case is impossible, by orthogonality with
$\{e_{3},e_{4},e_{5}\}$.
Therefore $\{e,x,e_{3}\}$ is a triangle of $M_{0}\ba y$.
Since $\{e,x,e_{2}\}$ is also a triangle,
we see that $\{e,x,e_{2},e_{3}\}$ is a $U_{2,4}$\dash restriction
of $M_{0}\ba y$ that is contained in the fan $F'$.
Therefore it intersects a triad, and 
this contradiction to \Cref{putty} shows that $n\leq 4$.
\end{proof}

\begin{sublemma}
\label{leach}
$n=3$.
\end{sublemma}

\begin{proof}
If $n>3$, then by \ref{crime}, $n=4$.
If $e_{1}$ is a rim element, we let
$(x_{1},x_{2},x_{3},x_{4},x,y)$ be
$(e_{1},\ldots, e_{6})$.
Otherwise, we shift $x$ and $y$ using \Cref{flame}
so that $(y,x,e_{1},e_{2},e_{3},e_{4})$ is a
fan of $M_{0}\ba e$, and now we let
$(x_{1},x_{2},x_{3},x_{4},x,y)$ be
$(e_{4},e_{3},e_{2},e_{1},x,y)$.
In either case $(x_{1},x_{2},x_{3},x_{4},x,y)$ is a fan of
$M_{0}\ba e$ with $x_{1}$ as a rim element, and
$N$ is a minor of $M_{0}/x\ba y$.
Since $\{x_{1},x_{2},x_{3}\}$ and $\{x_{3},x_{4},x\}$ are
triads of $M_{0}\ba e$, we deduce from \ref{paper} and
 \ref{brass} that $\{e,x_{1},x_{2},x_{3}\}$ and
$\{e,x_{3},x_{4},x\}$ are cocircuits of $M_{0}$, and that
$\{e,x_{1}\}$ and $\{e,x\}$ are contained in triangles.
If $\{e,x_{1},x\}$ is not a triangle, then $x$
is contained in the closure and the coclosure of
$\{x_{1},x_{2},x_{3},x_{4}\}$ in $M_{0}\ba e$, exactly
as in the proof of \ref{crime}.
Because this leads to a contradiction, $\{e,x_{1},x\}$ is a
triangle of $M_{0}$.

If $M_{0}\ba y$ is not $3$\dash connected, then there is a
triad of $M_{0}$ containing $y$, and $y$ is in the coclosure
and the closure of $\{x_{1},x_{2},x_{3},x_{4},x\}$ in $M_{0}\ba e$,
as in the proof of \ref{crime}.
This gives a contradiction, so $M_{0}\ba y$ is $3$\dash connected,
and hence a fan-extension of $N$.
There is a fan, $F'$, in $M_{0}\ba y$ such that $F$ is consistent
with $F'$, and $x_{1}$ and $x_{4}$ are terminal elements
of $F'$.
If neither $e$ nor $x$ is an internal element of $F'$, then
$F'-\{e,x\}$ is a fan of $M_{0}\ba y\ba e/x$.
If either $e$ or $x$ is an internal element, then
\Cref{cider,llama} show
$F'-\{e,x\}$ is a fan of $M_{0}\ba y\ba e/x$.
As $(x_{1},x_{2},x_{3},x_{4})$ is also a fan in
$M_{0}\ba y\ba e/x$, the elements $\{x_{1},x_{2},x_{3},x_{4}\}$
form a contiguous subsequence of $F'-\{e,x\}$.

Assume $(x_{1},x_{2},x_{3},x_{4})$ is not a contiguous
subsequence of $F'-\{e,x\}$.
As $F$ is consistent with $F'-\{e,x\}$, it follows that,
up to reversing, $F$ is either $(x_{1},x_{2},x_{4})$ or
$(x_{1},x_{3},x_{4})$.
In this case $\{x_{1},x_{2},x_{3},x_{4}\}\subseteq E(N)$,
for otherwise $N$ contains a series or parallel
pair.
But then $(x_{1},x_{2},x_{3},x_{4})$ is a fan of $N$,
and so is either $(x_{1},x_{2},x_{4})$ or $(x_{1},x_{3},x_{4})$,
contradicting \Cref{steal}.
Therefore $(x_{1},x_{2},x_{3},x_{4})$ is a
contiguous subsequence of $F'-\{e,x\}$.
Assume $(x_{1},x_{2},x_{3})$ is a contiguous subsequence
of $F'$.
Then $\{x_{1},x_{2},x_{3}\}$ is a triad of $M_{0}\ba y$
but not of $M_{0}$, so $y$ is in the
coclosure and closure of $\{x_{1},x_{2},x_{3},x_{4},x\}$ in
$M_{0}\ba e$.
This leads to a contradiction, so
$(x_{1},x_{2},x_{3})$ is not a contiguous subsequence,
and thus $e$, $x$, and $x_{2}$ are consecutive in $F'$.
By orthogonality with $\{x_{2},x_{3},x_{4}\}$,
$\{e,x,x_{2}\}$ is not a triad in $M_{0}\ba y$, so
$\{e,x,x_{2}\}$ is a triangle of $M_{0}\ba y$.
Thus $\{e,x_{1},x_{2},x\}$ is a $U_{2,4}$\dash restriction
of $M_{0}\ba y$ that intersects a triad contained in $F'$.
This contradiction completes the proof.
\end{proof}

\begin{sublemma}
\label{queen}
$e_{1}$ is a spoke element.
\end{sublemma}

\begin{proof}
Assume that $e_{1}$ is a rim element.
Then $(e_{1},e_{2},e_{3},y,x)$ is a fan of $M_{0}\ba e$, and
$F$ is equal to $(e_{1},e_{2},e_{3})$ or its reversal.
As before, we see that $\{e,e_{1},e_{2},e_{3}\}$ and 
$\{e,e_{3},y,x\}$ are cocircuits in $M_{0}$, and $\{e,e_{1},x\}$ is
a triangle.

Assume $M_{0}\ba y$ is not $3$\dash connected, so that
$y$ is contained in a triad, $T^{*}$, of $M_{0}$.
Note $T^{*}$ is a triad of $M_{0}\ba e$ but
$T^{*}\ne \{e_{3},y,x\}$.
If $T^{*}\cap\{e_{3},x\}\ne \emptyset$, then
$T^{*}\cup\{e_{3},x\}$ is a $U_{2,4}$\dash corestriction in
$M_{0}\ba e$, and $\{e_{2},e_{3},y\}$ is a triangle.
This contradiction to \Cref{putty} shows that $T^{*}\cap \{e_{3},x\}=\emptyset$.
Orthogonality with $\{e_{2},e_{3},y\}$ now shows that $e_{2}$ is
in $T^{*}$.
Assume that we swap labels on $e_{3}$ and $x$ in
$M_{0}$, and then contract $x$.
The resulting matroid has $N$ as a minor, as
$\{e_{3},x\}$ is a series pair in $M_{0}\ba e\ba y$.
It also has the triad $T^{*}$, which contains the parallel pair
$\{e_{2},y\}$.
This contradiction shows that $M_{0}\ba y$ is $3$\dash connected,
and therefore a fan-extension of $N$.

Let $F'$ be a fan of $M_{0}\ba y$ such that
$(e_{1},e_{2},e_{3})$ is consistent with $F'$, and $e_{1}$
and $e_{3}$ are terminal elements.
Now $M_{0}\ba y\ba e/x$ is $3$\dash connected, and contains
the triad $\{e_{1},e_{2},e_{3}\}$.
As before, $F'-\{e,x\}$ is a fan of $M_{0}\ba y\ba e/x$.
It follows that $(e_{1},e_{2},e_{3})$ is a contiguous
subsequence of $F'-\{e,x\}$.
If $(e_{1},e_{2},e_{3})$ is a contiguous subsequence of
$F'$, then $\{e_{1},e_{2},e_{3}\}$ is a triad of $M_{0}\ba y$,
since it cannot be a triangle and a triad in $M_{0}\ba y\ba e/x$.
Therefore $\{e_{1},e_{2},e_{3}\}$ is a triad in $M_{0}\ba y$
but not in $M_{0}$.
This means $\{e_{1},e_{2},e_{3},y\}$
is $2$\dash separating in $M_{0}\ba e$, leading to
the contradiction that $M_{0}\ba e$ is a wheel or a whirl.
Therefore $(e_{1},e_{2},e_{3})$ is not a contiguous subsequence
of $F'$, so $\{e,x,e_{2}\}$ is a set of consecutive elements in $F'$.
If $\{e,x,e_{2}\}$ is a triad in $M_{0}\ba y$, then
$\{e,x,e_{2},y\}$ is a cocircuit in $M_{0}$, by orthogonality
with $\{e_{2},e_{3},y\}$.
Therefore $\{e_{2},y,x\}$ and $\{e_{3},y,x\}$ are both triads
in $M_{0}\ba e$, so $M_{0}\ba e$ has a $U_{2,4}$\dash corestriction
that intersects a triangle, a contradiction.
This shows that $\{e,x,e_{2}\}$ is a triangle in $M_{0}\ba y$, so
$\{e,x,e_{1},e_{2}\}$ is a $U_{2,4}$\dash restriction of
$M_{0}\ba y$ that intersects a triad contained in $F'$, a
contradiction.
\end{proof}

Now we know that $e_{1}$ is a spoke element and $(e_{1},e_{2},e_{3})$
or its reversal is in $\mcal{F}_{N}$.
By shifting $x$ and $y$, we can assume that $(e_{1},e_{2},y,x,e_{3})$
is a fan in $M_{0}\ba e$ with $e_{1}$ as a spoke element,
and $N$ is a minor of $M_{0}\ba e/x\ba y$.
It follows from \ref{paper} that $\{e,e_{2},y,x\}$ is a
cocircuit in $M_{0}$.
Assume that $y$ is in a triad, $T^{*}$, of $M_{0}$.
Then $T^{*}\ne \{e_{2},y,x\}$, and $T^{*}$ is a triad of $M_{0}\ba e$.
If $T^{*}\cap\{e_{2},x\}\ne \emptyset$, then $T^{*}\cup\{e_{2},x\}$
is a $U_{2,4}$\dash corestriction of $M_{0}\ba e$ that intersects
a triangle, which is impossible.
Therefore $T^{*}\cap\{e_{2},x\}=\emptyset$.
Orthogonality shows that $T^{*}=\{e_{1},y,e_{3}\}$.
This means $M_{0}/x$ contains a dependent triad, so we
have a contradiction.
We conclude that $y$ is in no triad in $M_{0}$, so
$M_{0}\ba y$ is $3$\dash connected, and is therefore a
fan-extension of $N$.

Let $F'$ be a fan of $M_{0}\ba y$ such that
$(e_{1},e_{2},e_{3})$ is a subsequence of $F'$ and
$e_{1}$ and $e_{3}$ are terminal elements.
As $\{e_{1},e_{2},e_{3}\}$ is a triangle in
$M_{0}\ba y\ba e/x$, we conclude as before that
$F'-\{e,x\}=(e_{1},e_{2},e_{3})$.
If $F'=(e_{1},e_{2},e_{3})$, then $\{e_{1},e_{2},e_{3}\}$
is a triangle in $M_{0}\ba y$.
As $\{e_{1},e_{2},y\}$ is a triangle in $M_{0}\ba e$,
we see that $\{e_{1},e_{2},e_{3},y\}$ is a
$U_{2,4}$\dash restriction in $M_{0}\ba e$, and
$\{e_{2},y,x\}$ is a triad, so we have a contradiction.
Now it follows that $F'\ne (e_{1},e_{2},e_{3})$,
so $e$ and $x$ are internal elements of $F'$.
By \Cref{llama}, $x$ is a rim element of $F'$,
and $x$ and $e$ are consecutive.
As $\{e_{1},e_{2},e_{3}\}$ is a triangle in
$M_{0}\ba y\ba e/x$, it follows from \Cref{cider} that
$e_{1}$ is a spoke element of $F'$.
It follows that $F'$ is either
$(e_{1},e_{2},e,x,e_{3})$ or $(e_{1},x,e,e_{2},e_{3})$.
In the first case, $\{e_{1},e_{2},e,y\}$ and $\{e_{3},e,x,y\}$
are $U_{2,4}$\dash restrictions of $M_{0}$.
If $r_{M_{0}}(\{e,x,y,e_{1},e_{2},e_{3}\})>2$, then submodularity
implies that $r_{M_{0}}(\{e,y\})\leq 1$, which contradicts
$3$\dash connectivity.
Therefore $r_{M_{0}\ba e}(\{x,y,e_{1},e_{2},e_{3}\})=2$, which
quickly leads to a contradiction.
We conclude that $F'=(e_{1},x,e,e_{2},e_{3})$, so
$\{e_{1},x,e\}$ and $\{e,e_{2},e_{3}\}$ are triangles of $M_{0}$.

Because $\{e,e_{1}\}$ and $\{e_{3},y\}$ are parallel pairs in
$M_{0}/x$, we will swap the labels on $e$ and $e_{1}$, and
on $e_{3}$ and $y$.
Let $M'$ be the resulting matroid, so
$N$ is a minor of $M'\ba e/x\ba y$, and
$M'$ has triangles $\{e,e_{2},e_{3}\}$,
$\{e_{3},x,y\}$, $\{e,e_{1},x\}$, and $\{e_{1},e_{2},y\}$,
as well as the cocircuit $\{e_{1},e_{2},e_{3},x\}$.
Assume that $e$ is in a triad, $T^{*}$, of $M'$.
Orthogonality implies that $T^{*}$ contains an element
from each of $\{e_{2},e_{3}\}$ and $\{e_{1},x\}$.
By orthogonality with $\{e_{1},e_{2},y\}$ and
$\{e_{3},x,y\}$, and the fact that $\{e_{1},e_{2}\}$
is not a series pair in $N$ (and therefore not in $M'\ba e$),
we see that $T^{*}=\{e,e_{3},x\}$.
Therefore $\{e_{3},x,y\}$ is a codependent triangle in
$M'\ba e$.
This contradiction shows that $M'\ba e$ is
$3$\dash connected, and is therefore a fan-extension of $N$.

Let $F'$ be a fan in $M'\ba e$ that contains
$(e_{1},e_{2},e_{3})$ as a subsequence, and that has
$e_{1}$ and $e_{3}$ as terminal elements.
By the same arguments as before, we see that
$F'-\{x,y\}$ is a fan in $M'\ba e/x\ba y$.
However, $\{e_{1},e_{2},e_{3}\}$ is a triangle in
$M'\ba e/x\ba y$, so 
$F'-\{x,y\}=(e_{1},e_{2},e_{3})$.
Assume $F'=(e_{1},e_{2},e_{3})$.
Because $\{e_{3},x,y\}$ is a triangle in $M'\ba e$,
orthogonality shows that $\{e_{1},e_{2},e_{3}\}$ is also a triangle.
Because $\{e,e_{2},e_{3}\}$,  $\{e_{1},e_{2},y\}$, and $\{e,e_{1},x\}$
are also triangles in $M'$, it follows that
$r_{M'}(\{e_{1},e_{2},e_{3},e,x,y\})=2$.
This means that $\{e_{1},e_{2},e_{3}\}$ is a parallel class
in $M'/x$, so we have a contradiction.
Therefore $F'$ is not $(e_{1},e_{2},e_{3})$, so
$\{x,y,e_{2}\}$ is a consecutive set in $F'$.
If $\{x,y,e_{2}\}$ is a triangle in $M'\ba e$, then
exactly as before, we see that 
$r_{M'}(\{e_{1},e_{2},e_{3},e,x,y\})=2$.
Therefore $\{x,y,e_{2}\}$ is a triad of $M'\ba e$,
and $\{e,e_{2},x,y\}$ is a cocircuit of $M'$, by
orthogonality with $\{e,e_{2},e_{3}\}$.
Because $\{e_{1},e_{2},e_{3},x\}$ is also a cocircuit,
we see that $\{e,e_{2},e_{3},x\}$ cospans
$\{e_{1},e_{2},e_{3},e,x,y\}$ in $M'$.
As $\{e,e_{3},y\}$ spans the same set, it follows that
$\lambda_{M'}(\{e_{1},e_{2},e_{3},e,x,y\})\leq 1$.
This means that $|E(M'\ba e)|\leq 6$, and
as $|F'|=5$, it follows from \Cref{hocus}
that $M'\ba e$ is a wheel or a whirl.
This contradiction completes the proof of \Cref{train}.
\end{proof}

In fact, we can strengthen \Cref{train} further,
and show that elements in $E(M_{0}\ba e)-E(N)$ are
terminal elements of fans in a covering family.

\begin{lemma}
\label{jewel}
Let $M_{0}$ be isomorphic to $M$.
Assume that $N$ is a minor of $M_{0}$, but that $M_{0}$ is not
a fan-extension of $N$.
Assume also that $M_{0}\ba e$ is
$3$\dash connected and has $N$ as a minor.
Let $\mcal{F}_{e}$ be a covering family of $M_{0}\ba e$, and
let $(e_{1},\ldots, e_{m})$ be a fan in $\mcal{F}_{e}$.
If $1<i<m$, then $e_{i}$ is in $E(N)$.
\end{lemma}

\begin{proof}
Assume that the \namecref{jewel} fails.
By applying \Cref{train} and reversing as required, we 
can assume that $(e_{1},\ldots, e_{m})$ is in
$\mcal{F}_{e}$, and that $e_{2}$ is in $E(M_{0}\ba e)-E(N)$.

\begin{sublemma}
\label{creep}
$e_{1}$ is in $E(M_{0}\ba e)-E(N)$.
\end{sublemma}

\begin{proof}
If $e_{2}$ is a rim element in $(e_{1},\ldots, e_{m})$, then
we let $M'$ be $M_{0}\ba e$ and we let $N'$ be $N$.
Otherwise we let $M'$ be $(M_{0}\ba e)^{*}$ and we
let $N'$ be $N^{*}$.
Thus $M'$ is a fan-extension of $N'$ and $\mcal{F}_{e}$
is a covering family, while $e_{2}\in E(M')-E(N')$ is a rim
element in the fan $(e_{1},\ldots, e_{m})\in\mcal{F}_{e}$ of $M'$.

Assume for a contradiction that $e_{1}$ is in $E(N')$.
If $N'$ is a minor of $M'\ba e_{2}$, then
$m=4$, for otherwise $\{e_{3},e_{4},e_{5}\}$ is a codependent
triangle in $M'\ba e_{2}$, violating \Cref{lasso}.
But if $m=4$, then $\{e_{1},e_{3},e_{4}\}$ are all contained
in a fan in $\mcal{F}_{N}$, and
$\{e_{3},e_{4}\}$ is a series pair in $M'\ba e_{2}$,
which leads to a contradiction to the $3$\dash connectivity of $N'$.
Therefore $N'$ is a minor of $M'/e_{2}$.
If $m=4$, then $\{e_{1},e_{3},e_{4}\}\subseteq E(N')$, and
$\{e_{1},e_{3}\}$ is a parallel pair in $M'/x_{2}$.
This leads to a contradiction, so $m\geq 5$, meaning that
\Cref{train} implies $e_{3}\in E(N')$.
Now $\{e_{1},e_{3}\}$ is a parallel pair in $M'/e_{2}$,
and we again obtain a contradiction to the
$3$\dash connectivity of $N'$.
\end{proof}

It follows from \ref{creep} that $m\geq 5$.
If $e_{1}$ is a rim element of $(e_{1},\ldots, e_{m})$, then
it is easily checked that $N$ is a minor of $M_{0}\ba e/e_{1}\ba e_{2}$.
If $e_{1}$ is a spoke element, then $N$ is a minor of
$M_{0}\ba e\ba e_{1}/e_{2}$.
Let $M'$ be the matroid obtained from $M_{0}$ by swapping
the labels on $e_{1}$ and $e_{3}$.
If $e_{1}$ is a rim element, then $\{e_{1},e_{3}\}$ is a series
pair in $M_{0}\ba e\ba e_{2}$, so
$M'\ba e/e_{1}\ba e_{2}=M_{0}\ba e/e_{1}\ba e_{2}$.
If $e_{1}$ is a spoke element, then we can similarly show that
$M'\ba e\ba e_{1}/ e_{2}=M_{0}\ba e\ba e_{1}/e_{2}$.
Therefore $N$ is a minor of
$M'\ba e/e_{1}\ba e_{2}$ or $M'\ba e\ba e_{1}/e_{2}$.

Assume that $M'$ is not a fan-extension of $N$.
Since $e_{1}$ and $e_{2}$ are not in $E(N)$,
it is easy to see that
\[
(\mcal{F}_{e}-\{(e_{1},\ldots, e_{m})\})\cup
\{(e_{3},e_{2},e_{1},e_{4},\ldots, e_{m})\}
\]
is a covering family of $M'\ba e$.
But this contradicts \Cref{train}, as $e_{1}$ is not in
$E(N)$.
Therefore $M'$ is a fan-extension of $N$.
Note that this does not contradict \Cref{flame}, which would
apply if $N$ were a minor of $M_{0}\ba e/e_{2}\ba e_{3}$
or $M_{0}\ba e\ba e_{2}/e_{3}$.

Let $F_{N}$ be the fan in $\mcal{F}_{N}$ that is consistent
with $(e_{1},\ldots, e_{m})$.

\begin{sublemma}
\label{tiger}
$e_{3}$ is in $E(N)$, but not in $F_{N}$, and therefore
$e_{3}$ is contained in no fan in $\mcal{F}_{N}$.
\end{sublemma}

\begin{proof}
It follows immediately from \Cref{train} that $e_{3}$
is in $E(N)$.
Assume that \ref{tiger} fails, so that $e_{3}$ is in $F_{N}$.
Because $M'$ is a fan-extension of $N$, there is a
covering family, $\mcal{F}'$, of $M'$.
Let $F'$ be the fan in $\mcal{F}'$ such that $F_{N}$
is consistent with $F'$.
If $e$ is in $F'$, then it is a terminal spoke element,
as $M'\ba e$ is isomorphic to $M\ba e$, which is $3$\dash connected.
Therefore $F'-e$ is a fan in $M'\ba e$, and so is
$(e_{3},e_{2},e_{1},e_{4},\ldots, e_{m})$.
As these fans both contain the elements of $F_{N}$, it
follows from \Cref{unity} that the elements they
have in common form a contiguous subsequence in both
fans.
In particular, this means that $\{e_{3},e_{2},e_{1}\}$ is
contained in $F'-e$.
As this set is either a triad or a triangle in $M'\ba e$,
it forms a set of three consecutive elements in $F'-e$, and
hence in $F'$.
Let $F''$ be the sequence obtained from $F'$ by swapping the
positions of $e_{1}$ and $e_{3}$.
As $e_{1}$ and $e_{2}$ are not in $E(N)$, it follows that
$(\mcal{F}'-\{F'\})\cup\{F''\}$ is a covering family
of $M_{0}$.
This contradiction to \Cref{fever} completes the proof of
\ref{tiger}.
\end{proof}

\begin{sublemma}
\label{daisy}
Let $\mcal{F}'$ be a covering family of $M'$.
Then no fan in $\mcal{F}'$ contains $e_{3}$.
\end{sublemma}

\begin{proof}
Assume for a contradiction that $e_{3}$ is contained in a fan
in $\mcal{F}'$.
Let $\mcal{F}''$ be the family of fans in $M_{0}$
obtained from $\mcal{F}'$ by swapping the labels on $e_{1}$
and $e_{3}$.
Obviously there are $|\mcal{F}_{N}|$ pairwise
disjoint fans in $\mcal{F}''$.
Since $e_{3}$ is not contained in a fan in
$\mcal{F}_{N}$ (by \ref{tiger}), and the same statement
applies to $e_{1}$, it follows that every
fan in $\mcal{F}_{N}$ is consistent with a fan in
$\mcal{F}''$.
Moreover, every element in $E(M')-E(N)$ is contained
in a fan in $\mcal{F}'$, and since $e_{3}$ is contained in a
fan in $\mcal{F}'$, and $e_{3}\in E(N)$, we conclude that
every element in $E(M_{0})-E(N)$ is contained in a fan in
$\mcal{F}''$.
We have just shown that $\mcal{F}''$ is a covering family
of $M_{0}$, contradicting \Cref{fever}.
\end{proof}

Let $\{x,y\}=\{e_{1},e_{2}\}$, where $N$ is a minor of
$M\ba e/x\ba y$.

\begin{sublemma}
\label{camel}
Let $\mcal{F}'$ be a covering family of $M'$.
Then there is a fan in $\mcal{F}'$ that
contains $e_{1}$ and $e_{2}$.
\end{sublemma}

\begin{proof}
Let $\mcal{F}'$ be a covering family of $M'$.
Assume that \ref{camel} fails, so that there are distinct
fans, $F_{1}$ and $F_{2}$, in $\mcal{F}'$ that
(respectively) contain $e_{1}$ and $e_{2}$.
Note that $F_{1}$ and $F_{2}$ contain at least four elements
each.
If $e$ is contained in either $F_{1}$ or $F_{2}$, then it
is a terminal spoke element, as $M'\ba e$ is $3$\dash connected.
Therefore $F_{1}-e$ and $F_{2}-e$ are fans in $M'\ba e$,
and these fans must each contain at least four elements, as
they contain a fan in $\mcal{F}_{N}$, and an element
in $\{e_{1},e_{2}\}$.
Now $\{e_{1},e_{2},e_{3}\}$ is a triangle or a triad
in $M'\ba e$, and $F_{1}-e$ and $F_{2}-e$ are fans.
Since \ref{daisy} asserts that $e_{3}$ is in neither $F_{1}-e$ nor
$F_{2}-e$, orthogonality requires that either
$e_{1}$ and $e_{2}$ are both terminal spoke
elements of $F_{1}-e$ and $F_{2}-e$, or they are
both terminal rim elements.
In this former case, $M'\ba e/x$ contains a
triad that contains a parallel pair, and in the latter
case, $M'\ba e\ba y$ contains a triangle that
contains a series pair.
In either case we have a contradiction, so
\ref{camel} must hold.
\end{proof}

Let $\mcal{F}'$ be a covering family of $M'$, and let
$F'$ be a fan in $\mcal{F}'$ that contains
$e_{1}$ and $e_{2}$.
Then $F'$ contains at least five elements.
If $e$ is in $F'$, then it is a terminal spoke
element, as $M'\ba e$ is $3$\dash connected, so
$F'-e$ is a fan in $M'\ba e$.
Let $F'-e$ be $(f_{1},\ldots, f_{n})$.
Note that $n\geq 5$.
Now $\{e_{1},e_{2},e_{3}\}$ is a triangle or a triad
in $M'\ba e$ that intersects in $F'-e$ in exactly
the elements $e_{1}$ and $e_{2}$ (by \ref{daisy}).
We apply \Cref{ninny} or its dual.
Assume that statement (iii), (iv), or (v) holds.
Then either $\{e_{1},e_{2}\}=\{f_{1},f_{n}\}$,
or $n=5$ and $\{e_{1},e_{2}\}=\{f_{2},f_{4}\}$, and
in any case $e_{1}$ and $e_{2}$ are both
spoke elements in $F'-e$, or both are rim elements.
If both are spoke elements, $M'\ba e/x$ contains a triad that
contains a parallel pair, and if both are rim elements, then
$M'\ba e\ba y$ contains a triangle
that contains a series pair.
In either case we have a contradiction, so
statements (iii), (iv), and (v) in \Cref{ninny} cannot
hold.
Now, by reversing as necessary, we can assume that
$\{e_{1},e_{2}\}=\{f_{1},f_{2}\}$.

\begin{sublemma}
\label{stone}
$F'$ is either $(f_{1},\ldots, f_{n})$ or
$(f_{1},\ldots, f_{n},e)$.
\end{sublemma}

\begin{proof}
If $e$ is not contained in $F'$, then
$F'=F'-e=(f_{1},\ldots, f_{n})$, and we are done.
Therefore assume that $e$ is a terminal spoke element
of $F'$.
The only way \ref{stone} can fail to be true is if
$F'=(e,f_{1},\ldots, f_{n})$, so let us assume this is the
case.
Then $\{e,e_{1},e_{2}\}=\{e,f_{1},f_{2}\}$ is a triangle of
$M'$.
If $\{e_{1},e_{2},e_{3}\}$ is a triangle of $M'$, then
$\{e,e_{1},e_{2},e_{3}\}$
is a $U_{2,4}$\dash restriction, which is a contradiction,
as $\{f_{1},f_{2},f_{3}\}=\{e_{1},e_{2},f_{3}\}$ is a triad.
Similarly, if $\{e_{1},e_{2},e_{3}\}$ is a triad in $M'$, then
$\{e_{2},e_{1},e_{4}\}$ is a triangle, so
$\{e,e_{2},e_{1},e_{4}\}$ is a $U_{2,4}$\dash restriction.
This again leads to a contradiction.
Therefore \ref{stone} holds.
\end{proof}

\begin{sublemma}
\label{horse}
$\{e,e_{1},e_{2},e_{3}\}$ is a cocircuit in $M'$.
\end{sublemma}

\begin{proof}
Assume that $\{e_{1},e_{2},e_{3}\}$ is a triangle in $M'\ba e$,
and hence in $M'$.
If $\{f_{1},f_{2},f_{3}\}$ is a triangle in $M'$, then
$\{e_{3},f_{1},f_{2},f_{3}\}$ is a $U_{2,4}$\dash restriction
in $M'$, and $\{f_{2},f_{3},f_{4}\}$ is a triad.
As this is a contradiction, it follows that
$\{f_{1},f_{2},f_{3}\}$ is a triad in $M'$.
Let $F'+e_{3}$ be the sequence obtained from $F'$ by appending
$e_{3}$ to the beginning.
Then $F'+e_{3}$ is a fan in $M'$, and it follows from
\ref{tiger} and \ref{daisy} that
$(\mcal{F}'-\{F'\})\cup\{F'+e_{3}\}$ is a covering family of $M'$.
But this contradicts \ref{daisy}, so we conclude that
$\{e_{1},e_{2},e_{3}\}$ is a triad in $M'\ba e$.

To conclude the proof of \ref{horse}, assume that
$\{e_{1},e_{2},e_{3}\}$ is a triad in $M'$.
If $\{f_{1},f_{2},f_{3}\}$ is a triad in $M'$, then
$\{e_{3},f_{1},f_{2},f_{3}\}$ is a $U_{2,4}$\dash corestriction
in $M'$, which is a contradiction as $\{f_{2},f_{3},f_{4}\}$
is a triangle.
Therefore $\{f_{1},f_{2},f_{3}\}$ is a triangle in $M'$.
We again let $F'+e_{3}$ be obtained from $F'$ by appending
$e_{3}$ to the beginning.
This leads to a contradictory covering family, just
as in the previous paragraph.
Therefore $\{e,e_{1},e_{2},e_{3}\}$ is a cocircuit in $M$.
\end{proof}

As $e$ is in $E(M')-E(N)$, there is a fan, $F_{e}$,
in $\mcal{F}'$, that contains $e$.
Then $F_{e}$ contains at least four elements.
As $M'\ba e$ is $3$\dash connected, it follows that
$e$ is a terminal spoke element in $F_{e}$.
Let $\{e,u,v\}$ be the triangle contained in $F_{e}$
that contains $e$.
Orthogonality between the triangle $\{e,u,v\}$ and
the cocircuit $\{e,e_{1},e_{2},e_{3}\}$ shows that
either $e_{1}$, $e_{2}$, or $e_{3}$ is in $\{u,v\}$.
The last case is impossible by \ref{daisy},
so $e_{1}$ or $e_{2}$ is in $\{u,v\}$.
This means that $F_{e}$ and $F'$ cannot
be disjoint fans, so
$F'=(f_{1},\ldots, f_{n},e)$, by \ref{stone}.
However it now follows that
$\{e_{1},e_{2}\}=\{f_{1},f_{2}\}$,
and $\{u,v\}=\{f_{n-1},f_{n}\}$.
Thus $F'$ can contain at most four elements, which is
impossible, as it contains $e$, $e_{1}$, $e_{2}$, and
a fan in $\mcal{F}_{N}$.
This completes the proof of \Cref{jewel}.
\end{proof}

\Cref{bliss} supplies us with a matroid $M_{0}$, which we
now relabel as $M$, such that $M$ has $N$ as a minor,
but is not a fan-extension of $N$.
By replacing $M$, $N$, and \mcal{M} with their duals as
necessary, we can assume that $M\ba e$ is $3$\dash connected
and has $N$ as a minor.
The minimality of $M$ means that $M\ba e$ is a fan-extension
of $N$.

\begin{lemma}
\label{balsa}
Let $\mcal{F}_{e}$ be a covering family of $M\ba e$, and
assume that $(e_{1},\ldots,e_{m})$ is a fan in $\mcal{F}_{e}$.
Assume also that $e_{1}$ is not in $E(N)$.
If $e_{1}$ is a spoke element of $(e_{1},\ldots,e_{m})$,
then $M\ba e\ba e_{1}$ is $3$\dash connected and has $N$
as a minor.
If $e_{1}$ is a rim element, then $M\ba e/e_{1}$ is
$3$\dash connected and has $N$ as a minor.
\end{lemma}

\begin{proof}
If $e_{1}$ is a rim element, then let $M'=M\ba e$ and let $N'=N$.
Otherwise, let $M'=(M\ba e)^{*}$ and let $N'=N^{*}$.
Thus, in either case, we assume that $\mcal{F}_{e}$ is a
covering family of $M'$ relative to $N'$ and $\mcal{F}_{N}$,
that $(e_{1},\ldots,e_{m})$ is a fan in
$\mcal{F}_{e}$, and $e_{1}$ is a rim element that is not in $E(N')$.
Certainly $M'/e_{1}$ has $N'$ as a minor, or else $M'\ba e_{1}$
has $N'$ as a minor and $\{e_{2},e_{3},e_{4}\}$ as a
codependent triangle, contradicting \Cref{lasso}.
Assume $M'/e_{1}$ is not $3$\dash connected.
Then $e_{1}$ is contained in a triangle, $T$, of $M'$.
Orthogonality with the triad $\{e_{1},e_{2},e_{3}\}$ shows that $T$
contains $e_{2}$ or $e_{3}$.
Note that $T\nsubseteq \{e_{1},\ldots, e_{m}\}$, for otherwise
\Cref{steal} implies $T=\{e_{1},e_{2},e_{3}\}$, and
thus $T$ is a triad and a triangle in $M'$.
Let $x$ be the element in $T-\{e_{1},e_{2},e_{3}\}$.

Assume that $e_{2}$ is in $T$.
Since $e_{2}$ is in $E(N')$ by \Cref{jewel}, and $\{e_{2},x\}$ is
a parallel pair in $M'/e_{1}$, it follows that $N'$ is a minor of
$M'/e_{1}\ba x$.
The definition of a covering family means there is a fan
in $\mcal{F}_{e}$ that contains $x$.
Let $F_{x}$ be this fan.
Then $F_{x}$ contains at least four elements.
Orthogonality with $T$ shows that $x$ is a terminal
spoke element of $F_{x}$.
Since $x$ is not in $E(N')$, it now follows that
$(\mcal{F}_{e}-\{(e_{1},\ldots,e_{m}), F_{x}\})\cup
\{(x,e_{1},\ldots,e_{m}), F_{x}-x\}$
is a covering family in $M'$, and now we have a
contradiction to \Cref{jewel} as $x,e_{1}\notin E(N')$.
Therefore $T=\{e_{1},e_{3},x\}$.

Because $x$ is not in $(e_{1},\ldots, e_{m})$,
we see that $m=4$, as otherwise we violate orthogonality
between $T$ and $\{e_{3},e_{4},e_{5}\}$.
Hence $(e_{2},e_{3},e_{4})$ or its reversal is
in $\mcal{F}_{N}$.
Let $M''$ be obtained from $M'$ by swapping labels on $e_{3}$ and $x$.
Thus $(e_{3},e_{1},x,e_{2},e_{4})$ is a fan of $M''$.
As $\{e_{3},x\}$ is a parallel pair in $M'/e_{1}$,
we see that $M''$ has $N'$ as a minor.
Since $|E(M'')|=|E(M')|<|E(M)|$, it follows that $M''$ is a
fan-extension of $N'$.
Therefore there is a fan of $M''$ that contains $(e_{2},e_{3},e_{4})$
as a subsequence.
Let this fan be $F$.
By comparing $F$ with $(e_{3},e_{1},x,e_{2},e_{4})$ in $M''$
using \Cref{unity}, we see that there is a contiguous subsequence
of $F$ using the elements $\{e_{3},e_{1},x,e_{2},e_{4}\}$.
We apply \Cref{steal} and its dual to $\{e_{3},e_{1},x\}$,
$\{e_{1},x,e_{2}\}$, and $\{x,e_{2},e_{4}\}$ and see that $F$
contains $(e_{3},e_{1},x,e_{2},e_{4})$ as a contiguous subsequence,
which contradicts the definition of $F$.
Therefore $M'/e_{1}$ is $3$\dash connected, and this
completes the proof.
\end{proof}

\begin{lemma}
\label{hovel}
Let $\mcal{F}_{e}$ be a covering family of $M\ba e$, and
assume that $(e_{1},\ldots,e_{m})$ is a fan in $\mcal{F}_{e}$.
Assume also that $e_{1}$ is not in $E(N)$.
If $e_{1}$ is a rim element of $(e_{1},\ldots,e_{m})$,
then assume that there is no triangle of $M$ that contains
$\{e_{1},e\}$.
Let $(e_{s},\ldots,e_{t})$ be a minimal contiguous subsequence
of $(e_{1},\ldots,e_{m})$ such that a fan, $F_{N}\in\mcal{F}_{N}$
is consistent with $(e_{s},\ldots,e_{t})$.
Then $2 \leq s <s+2\leq  t \leq m$ and
$F_{N}$ contains $e_{s}$ and $e_{t}$.
Moreover, $F_{N}$ is equal to $(e_{s},\ldots,e_{t})$ or its
reversal, and is a fan of $M$.
However, $(e_{1},\ldots,e_{t})$ is not a fan in $M$.
\end{lemma}

\begin{proof}
Since $(e_{1},\ldots,e_{m})$ is a fan in a covering family, there is a
fan in $\mcal{F}_{N}$ that is consistent with $(e_{1},\ldots,e_{m})$.
Let $e_{s}$ be the first element in this fan, and let $e_{t}$ be the last.
Then $2\leq s<s+2\leq t\leq m$ because each fan contains at least three
elements, and $e_{1}\notin E(N)$.
By \Cref{jewel} it follows that every element in
$\{e_{s},\ldots,e_{t}\}$ is in $E(N)$.
Since $N$ is $3$\dash connected, $(e_{s},\ldots,e_{t})$ is a fan
of $N$.
As $F_{N}$ is also a fan of $N$, we apply \Cref{unity} and
deduce that $F_{N}$ and $(e_{s},\ldots,e_{t})$ use the same
set of elements.
Since $F_{N}$ is consistent with $(e_{s},\ldots,e_{t})$, this
means that $F_{N}$ is equal to $(e_{s},\ldots,e_{t})$ or its
reversal.

If $e_{1}$ is a spoke element of $(e_{1},\ldots,e_{m})$,
let $M'$ be $M\ba e_{1}$.
In this case, $M'$ is $3$\dash connected, since
$M'\ba e$ is $3$\dash connected by \Cref{balsa}.
If $e_{1}$ is a rim element, then let $M'$ be $M/e_{1}$.
In this case also, $M'\ba e$ is $3$\dash connected, and the
hypotheses imply that $e$ is not in a parallel pair in $M'$.
Therefore $M'$ is $3$\dash connected in either case, and
$M'$ has $N$ as a minor.
As $|E(M')|<|E(M)|$, it follows that $M'$ is a fan-extension
of $N$.

Let $F$ be a fan of $M'$ such that $(e_{s},\ldots,e_{t})$
is consistent with $F$.
If $e$ is in $F$, then it is a terminal spoke element,
as $M'\ba e$ is $3$\dash connected, so in this case
$F-e$ is a fan in $M'\ba e$.
In fact, $F-e$ is a fan of $M'\ba e$ whether or not $e$ is in $F$.
Since $(e_{s},\ldots,e_{t})$ is also a fan in $M'\ba e$,
by \Cref{unity}, there is a contiguous subsequence of $F-e$
using the elements $\{e_{s},\ldots,e_{t}\}$.
Since $(e_{s},\ldots,e_{t})$ is consistent with $F-e$, this
means $(e_{s},\ldots,e_{t})$ is a contiguous subsequence of $F-e$.
Hence $(e_{s},\ldots,e_{t})$ is a fan in $M'$.

Assume $(e_{s},\ldots,e_{t})$ is not a fan in $M$.
If $e_{1}$ is a spoke element in $(e_{1},\ldots, e_{m})$, then
$M'=M\ba e_{1}$, so there is a rim element, $e_{i}$, of
$(e_{s},\ldots,e_{t})$, such that $s\leq i\leq t-2$, and
$\{e_{1},e_{i},e_{i+1},e_{i+2}\}$ is a cocircuit of $M$.
But in this case $e_{1}$ is in the closure and
coclosure of $(e_{2},\ldots, e_{m})$ in $M\ba e$.
Hence $(e_{1},\ldots, e_{m})$ is $2$\dash separating in
$M\ba e$, so \Cref{hocus} implies $M\ba e$ is a wheel or a
whirl, a contradiction.
The argument when $e_{1}$ is a rim element in
$(e_{1},\ldots, e_{m})$ is similar:
In this case $M'=M/e_{1}$, so there
is a spoke element, $e_{i}$, in
$(e_{s},\ldots,e_{t})$, such that 
$\{e_{1},e_{i},e_{i+1},e_{i+2}\}$ is a circuit of $M$.
Thus $e_{1}$ is in the closure and coclosure of
$(e_{2},\ldots, e_{m})$ in $M\ba e$, so we can deduce that
$M\ba e$ is a wheel or a whirl, contradicting \Cref{ether}.
From this contradiction we see that
$(e_{s},\ldots,e_{t})$ is a fan of $M$.

To complete the proof, we will assume that $(e_{1},\ldots,e_{t})$
is a fan of $M$, and deduce a contradiction.
This assumption means that $(e_{2},\ldots,e_{t})$ is a fan in $M'$.

\begin{sublemma}
\label{blast}
There is a covering family of $M'$ containing a fan
that has $(e_{2},\ldots,e_{t})$ as a contiguous subsequence.
\end{sublemma}

\begin{proof}
Let \mcal{F} be a covering family of $M'$.
Then \mcal{F} contains a fan, $F$, such that $(e_{s},\ldots,e_{t})$
is consistent with $F$.
As $(e_{s},\ldots, e_{t})$ is a fan in $M'$,
we can use \Cref{unity} to show that $(e_{s},\ldots, e_{t})$
is a contiguous subsequence of $F$.
Assume that $F$ contains $(e_{i},\ldots,e_{t})$ as a
contiguous subsequence, where $2\leq i\leq s$, and
\mcal{F} and $F$ have been chosen so that $i$ is as small as
possible.
If $i=2$, then \ref{blast} is already proved, so we
assume $2 < i$.
Let $F=(f_{1},\ldots,f_{n})$, where
$(e_{i},\ldots,e_{t})=(f_{j},\ldots,f_{j+t-i})$.

Assume $j>1$.
Then $f_{j-1}\ne e_{i-1}$, for otherwise
$(e_{i-1},\ldots,e_{t})$ is a contiguous subsequence of $F$,
and the minimality of $i$ is contradicted.
Since $\{f_{j},f_{j+1},f_{j+2}\}=\{e_{i},e_{i+1},e_{i+2}\}$
is not a triad and a triangle, $f_{j}=e_{i}$ is a
spoke element in both $F$ and $(e_{2},\ldots,e_{m})$, or a
rim element in both.
Now $\{f_{j-1},f_{j},f_{j+1}\}$ and $\{e_{i-1},e_{i},e_{i+1}\}$
are distinct triads or distinct triangles that intersect in two
elements, and their union is a $U_{2,4}$\dash corestriction or
restriction of $M'$ that intersects a triangle or triad.
This contradiction shows that $j=1$.

Now $(e_{i-1},f_{1},\ldots,f_{n})$ is a fan of $M'$.
Let $F'$ be this fan.
Note that because $i-1<s$, the element $e_{i-1}$
is not contained in the fan $(e_{s},\ldots, e_{t})$, and
since it is contained in $(e_{1},\ldots, e_{m})\in \mcal{F}_{e}$,
the definition of a covering family tells us that
$e_{i-1}$ is contained in no fan in $\mcal{F}_{N}$.
If $e_{i-1}$ is in no fan in \mcal{F}, then
$(\mcal{F}-\{F\})\cup\{F'\}$ is a covering family in $M'$.
If $e_{i-1}$ is in a fan, $F''\in \mcal{F}$, then it is a
terminal element of $F''$ by orthogonality with
$\{e_{i-1},e_{i},e_{i+1}\}$, so 
$(\mcal{F}-\{F,F''\})\cup\{F',F''-e_{i-1}\}$
is a covering family of $M'$.
In either case we have constructed a covering family of
$M'$ that contains the fan $F'$, so the minimality of $i$ is
contradicted.
This completes the proof of \ref{blast}.
\end{proof}

Now we assume \mcal{F} is a covering family of $M'$,
and \mcal{F} contains a fan, $F$, that has
$(e_{2},\ldots,e_{t})$ as a contiguous subsequence.
Let $F=(f_{1},\ldots,f_{n})$, where
$(e_{2},\ldots,e_{t})=(f_{j},\ldots,f_{j+t-2})$.
Assume $j>1$.
If $e_{1}$ is a rim element of $(e_{1},\ldots, e_{m})$,
then $\{e_{2},e_{3},e_{4}\}$ is a triangle in $M$, and hence
in $M'=M/e_{1}$.
Therefore $e_{2}$ is a spoke element in
$(f_{j},\ldots, f_{j+t-2})$, so
$\{f_{j-1},f_{j},f_{j+1}\}$ is a triad of $M'$ and therefore
in $M$.
Now $\{e_{1},f_{j-1},f_{j},f_{j+1}\}$ is a $U_{2,4}$\dash corestriction
of $M$ that intersects a triangle.
If $e_{1}$ is a spoke element, we reach a similar contradiction.
Therefore $j=1$.
But this means $(e_{1},f_{1},\ldots,f_{n})$ is a fan of $M$, and
$M$ is obtained from $M'$ by a fan-lengthening move on $F$.
This implies $M$ is a fan-extension of $N$, so we have a contradiction
that completes the proof of \Cref{hovel}.
\end{proof}

Now we fix $\mcal{F}_{e}$ to be a covering family of $M\ba e$.
Since we have assumed $|E(M)|-|E(N)|>2$, it follows that
$M\ba e\ne N$.
Any element in $E(M\ba e)-E(N)$ is a terminal element of
a fan in $\mcal{F}_{e}$, by \Cref{jewel}.

\begin{lemma}
\label{chaff}
If $(e_{1},\ldots,e_{m})$ is a fan in $\mcal{F}_{e}$ with the
property that $e_{1}\notin E(N)$, then
$e_{1}$ is a rim element, and $\{e_{1},e\}$ is contained in a triangle
of $M$.
\end{lemma}

\begin{proof}
If the \namecref{chaff} fails, then there is a fan
$(e_{1},\ldots,e_{m})$ in $\mcal{F}_{e}$ such that $e_{1}$
is not in $E(N)$, and either $e_{1}$ is a spoke element, or
$e_{1}$ is a rim element of $(e_{1},\ldots,e_{m})$ and
there is no triangle of $M$ containing $\{e_{1},e\}$.

There are indices $p$ and $q$ such that
$2\leq p<p+2\leq q\leq m$, and there is a fan in $\mcal{F}_{N}$
consistent with $(e_{p},\ldots, e_{q})$ that contains
$e_{p}$ and $e_{q}$.
Now \Cref{hovel} applies, so $(e_{p},\ldots, e_{q})$
or its reversal is in $\mcal{F}_{N}$, and is a fan of
$M$.
However, $(e_{1},\ldots, e_{q})$ is not a fan of $M$.
Since it is a fan in $M\ba e$, there is some
$j\in \{1,\ldots, p-1\}$ such that
$\{e,e_{j},e_{j+1},e_{j+2}\}$
is a cocircuit of $M$.

Let $M'$ be $M\ba e_{1}$ if $e_{1}$ is a spoke element, and
let it be $M/e_{1}$ otherwise.
\Cref{balsa} implies that $M'\ba e$ is
$3$\dash connected with $N$ as a minor.
Because the hypotheses imply $M'$ has no parallel pair
containing $e$, we now see that $M'$ is $3$\dash connected.
Therefore $M'$ is a fan-extension of $N$.

As $|E(M)|-|E(N)|>2$, we can let $u$ be an element in
$E(M'\ba e)-E(N)$.
Either $u$ belongs to a fan in $\mcal{F}_{e}$ that is distinct,
and hence disjoint, from $(e_{1},\ldots,e_{m})$, or, by
\Cref{jewel}, $u=e_{m}$.
The analyses in the two cases are similar, so we combine them.
If necessary we can reverse the fan that contains $u$,
and assume that we are in one of the following situations.
\begin{enumerate}[label=\textup{(\Roman*)}]
\item $(x_{1},\ldots,x_{n})$ is a fan in $\mcal{F}_{e}$ that
is disjoint from $(e_{1},\ldots,e_{m})$, and $x_{1}$ is not
in $E(N)$.
\item $e_{m}$ is not in $E(N)$.
In this case we let $(x_{1},\ldots,x_{n})$ be $(e_{m},\ldots,e_{2})$.
\end{enumerate}
Note that in either of these cases, $(x_{1},\ldots,x_{n})$ is a fan in
$M\ba e$ and in $M'\ba e$, and $x_{1}$ is not in $E(N)$.
By taking a minimal contiguous subsequence of $(x_{1},\ldots, x_{n})$
such that a fan in $\mcal{F}_{N}$ is consistent with the subsequence, we
see there are integers $s$ and $t$ such that $2\leq s < s+2\leq t\leq n$
and such that some fan in $\mcal{F}_{N}$ contains
$x_{s}$ and $x_{t}$ and is consistent with $(x_{s},\ldots,x_{t})$.

\begin{sublemma}
\label{rupee}
Either $(x_{s},\ldots, x_{t})$ or its reversal is in
$\mcal{F}_{N}$.
Moreover, $(x_{s},\ldots, x_{t})$ is a fan in $M$, but
$(x_{1},\ldots, x_{t})$ is not.
\end{sublemma}

\begin{proof}
We first show that $(x_{s},\ldots, x_{t})$ is a fan in $M$.
In Case~(II), $(x_{s},\ldots, x_{t})$ is equal to
$(e_{q},\ldots, e_{p})$, and we have already noted this is a
fan in $M$.
Therefore we assume Case~(I) holds.
We can apply \Cref{hovel} to $(x_{1},\ldots, x_{n})$,
unless $x_{1}$ is a rim element and $\{x_{1},e\}$ is contained
in a triangle of $M$.
If \Cref{hovel} does apply, then it tells us that
$(x_{s},\ldots, x_{t})$ is a fan of $M$.
Therefore we assume that $x_{1}$ is a rim element, and
$\{x_{1},e\}$ is contained
in a triangle.
Orthogonality with the cocircuit $\{e,e_{j},e_{j+1},e_{j+2}\}$
shows that $\{x_{1},e,e_{k}\}$ is a triangle of $M$
for some $k\in\{j,j+1,j+2\}$.
Because $(x_{s},\ldots, x_{t})$ is a fan in $M\ba e$, it follows that
if it is not a fan of $M$, then there is a
cocircuit of $M$ that contains $e$ and three consecutive
elements from $(x_{s},\ldots, x_{t})$.
This violates orthogonality with $\{x_{1},e,e_{k}\}$ as $s\geq 2$.
Hence $(x_{s},\ldots, x_{t})$ is a fan of $M$ in any case.

Next we show that $(x_{s},\ldots, x_{t})$ or its reversal is in
$\mcal{F}_{N}$.
Since $x_{s}$ and $x_{t}$ are in $E(N)$, \Cref{jewel}
implies $\{x_{s},\ldots, x_{t}\}\subseteq E(N)$.
It follows that $(x_{s},\ldots, x_{t})$ is a fan of $N$.
Some fan in $\mcal{F}_{N}$ is consistent with
$(x_{s},\ldots, x_{t})$ and contains $x_{s}$ and $x_{t}$.
It follows easily from \Cref{unity} that this fan must be
either $(x_{s},\ldots, x_{t})$ or its reversal.

Finally we show that $(x_{1},\ldots, x_{t})$ is not a fan of $M$.
We can apply \Cref{hovel}, unless $x_{1}$ is a
rim element of $(x_{1},\ldots, x_{n})$
and $\{x_{1},e\}$ is contained in a triangle of $M$.
If we can apply \Cref{hovel}, then it tells us that
$(x_{1},\ldots, x_{t})$ is not a fan in $M$, as desired.
Therefore we assume $x_{1}$ is a rim element, and
$\{x_{1},e\}$ is contained in a triangle, $T$, of $M$.
If $\{x_{1},x_{2},x_{3}\}$ is not a triad of $M$, then
$(x_{1},\ldots, x_{t})$ is not a fan of $M$, so we are done.
Therefore we assume $\{x_{1},x_{2},x_{3}\}$ is a triad of $M$.
Orthogonality shows that $T$ is $\{e,x_{1},x_{2}\}$ or
$\{e,x_{1},x_{3}\}$.

If Case~(I) holds, then we have a violation of orthogonality
between $T$ and the cocircuit $\{e,e_{j},e_{j+1},e_{j+2}\}$.
Therefore Case~(II) holds, so
$(x_{1},x_{2},x_{3})=(e_{m},e_{m-1},e_{m-2})$.
Orthogonality between $\{e,e_{j},e_{j+1},e_{j+2}\}$
and $T$ requires that $j+2\geq m-2$.
Recall that $j\leq p-1$.
Certainly $p+2\leq q$, as $(e_{p},\ldots, e_{q})$
or its reversal is in $\mcal{F}_{N}$.
Because $e_{q}$ is in $E(N)$, but $e_{m}$ is not, it follows
that $q\leq m-1$.
Putting these together, we see that
$j+2\leq p+1\leq q-1\leq m-2$, so $j+2=m-2$, and equality
holds throughout this expression.
Again applying orthogonality between
$\{e,e_{j},e_{j+1},e_{j+2}\}$ and $T$, we see that
$T=\{e,x_{1},x_{3}\}=\{e,e_{m},e_{m-2}\}$.
As $p=m-3$ and $q=m-1$, this shows that
$(x_{s},\ldots, x_{t})=(e_{m-1},e_{m-2},e_{m-3})$
and $(x_{2},x_{3},x_{4})=(e_{m-1},e_{m-2},e_{m-3})$ or
its reversal is in $\mcal{F}_{N}$.
If $(x_{1},\ldots, x_{t})=(e_{m},e_{m-1},e_{m-2},e_{m-3})$
is not a fan in $M$, then we are done, so we
assume it is a fan in $M$.
Therefore $(x_{1},x_{2},x_{3},x_{4})$ is a fan in $M'$
also.

Because $x_{1}$ is not in $E(N)$, it follows easily that
$N$ is a minor of $M'\ba e/x_{1}$.
Let $M''$ be the matroid obtained from $M'$ by
swapping labels on $e$ and $x_{3}$.
Since $\{e,x_{3}\}$ is a parallel pair in $M'/x_{1}$,
$M''$ has $N$ as a minor.
Moreover $M''$ is $3$\dash connected since $M'$ is,
and $|E(M'')|=|E(M')|<|E(M)|$, so $M''$ is a fan-extension
of $N$.
Let $F$ be a fan of $M''$ such that $(x_{2},x_{3},x_{4})$
is consistent with $F$.
Because $M''\ba x_{3}$ is isomorphic to
$M'\ba e$, and is therefore $3$\dash connected,
$x_{3}$ is contained in no triad in $F$.
If $F$ contains at least four elements, then
this means that $x_{3}$ is a terminal element of $F$,
which is impossible as $(x_{2},x_{3},x_{4})$
is consistent with $F$.
Therefore $F$ contains exactly three elements, so
$\{x_{2},x_{3},x_{4}\}$ is a triangle of $M''$.
As $\{x_{1},x_{2},e\}$ is a triad of $M''$, we have a
contradiction to orthogonality.
\end{proof}

Applying \ref{rupee}, we let $C_{x}^{*}$ be a cocircuit of $M$
of the form $\{e,x_{k},x_{k+1},x_{k+2}\}$, where $x_{k}$ is a rim
element of $(x_{1},\ldots,x_{t})$ in $M\ba e$, and $1\leq k\leq s-1$.

\begin{sublemma}
\label{anise}
$C_{x}^{*}$ is a cocircuit of $M'$.
\end{sublemma}

\begin{proof}
Assume Case~(I) holds.
If $M'=M/e_{1}$, then $C_{x}^{*}$ is a cocircuit in $M'$,
as $e_{1}$ is not in $C_{x}^{*}$.
On the other hand, if $M'=M\ba e_{1}$, then $\{e_{1},e_{2},e_{3}\}$
is a triangle disjoint from $\{x_{1},\ldots,x_{n},e\}$, so
$C_{x}^{*}$ is a cocircuit of $M'$.
Now assume Case~(II) holds.
Note
$C_{x}^{*}-e\subseteq \{x_{1},\ldots,x_{s+1}\}\subseteq
\{e_{m},\ldots,e_{3}\}$, so
$e_{1}\notin C_{x}^{*}$.
Therefore, if $M'=M/e_{1}$, then $C_{x}^{*}$ is a cocircuit
of $M'$.
If $M'=M\ba e_{1}$, then $C_{x}^{*}$ is a cocircuit of $M'$,
for otherwise $e_{1}$ is in the coclosure and closure of
$\{e_{2},\ldots,e_{m}\}$ in $M\ba e$.
This implies $\{e_{1},\ldots,e_{m}\}$ is $2$\dash separating in
$M\ba e$, so $M\ba e$ is a wheel or whirl, a contradiction.
\end{proof}

Recall that $M'$ is a fan-extension of $N$.
Let $F_{e}$ be the fan containing $e$ in a covering family of $M'$.
As $M'\ba e$ is $3$\dash connected, $e$ is a terminal spoke
element of $F_{e}$.
Let $(e,x,y,z)$ be the initial four elements of $F_{e}$, so
$\{e,x,y\}$ is a triangle of $M'$ and $\{x,y,z\}$ is a triad.

\begin{sublemma}
\label{scope}
$\{x,y,z\}\subseteq E(N)$.
\end{sublemma}

\begin{proof}
This follows from \Cref{jewel} if $|F_{e}|>4$, and from
the fact that $F_{e}$ contains $e$ and a fan in
$\mcal{F}_{N}$ otherwise.
\end{proof}

Now $\{e,x,y\}$ is a triangle in $M'$, and $C_{x}^{*}$ is a cocircuit,
so orthogonality requires that $(C_{x}^{*}-e)\cap \{x,y\}\ne\emptyset$.
Thus $\{x,y,z\}$ is a triad of $M'\ba e$ that intersects the
fan $(x_{1},\ldots,x_{n})$.

As $\{x,y,z\}$ is contained in $E(N)$ and $x_{1}$ is not,
we see that $x_{1}\notin \{x,y,z\}$.
Furthermore, any element in $C_{x}^{*}\cap\{x,y\}$ is not
equal to $x_{n}$, since $C_{x}^{*}-e$ is contained in
$\{x_{1},\ldots,x_{s+1}\}\subseteq\{x_{1},\ldots,x_{n-1}\}$.
This means that if $\{x,y,z\}$ intersects $\{x_{1},\ldots,x_{n}\}$
in exactly one element, then it is a triad of $M'\ba e$ that
intersects the fan $(x_{1},\ldots,x_{n})$ in a single,
internal, element, violating orthogonality.
Therefore $|\{x,y,z\}\cap \{x_{1},\ldots,x_{n}\}|\geq 2$.

\begin{sublemma}
\label{devil}
$|\{x,y,z\}\cap \{x_{1},\ldots,x_{n}\}|\ne 2$.
\end{sublemma}

\begin{proof}
Assume $|\{x,y,z\}\cap \{x_{1},\ldots,x_{n}\}|= 2$.
We apply the dual of \Cref{ninny} in $M'\ba e$
to $\{x,y,z\}$ and the fan $(x_{1},\ldots, x_{n})$.
Since $x_{1}$ is not in $\{x,y,z\}$, statement~(i)
or~(iii) in \Cref{ninny} cannot apply.
By the same reasoning, if statement~(v) holds, then $n=5$.
In summary, either $\{x,y,z\}$ intersects $\{x_{1},\ldots,x_{n}\}$ in
$\{x_{n-1},x_{n}\}$, or $n\leq 5$, $x_{1}$ is a rim element,
and $\{x,y,z\}$ intersects $\{x_{1},\ldots,x_{n}\}$ in $\{x_{2},x_{4}\}$.

Assume $\{x,y,z\}\cap \{x_{1},\ldots,x_{n}\}=\{x_{n-1},x_{n}\}$.
Then \Cref{ninny} implies $x_{n}$ is a spoke element in $M'\ba e$.
Since $C_{x}^{*}-e=\{x_{k},x_{k+1},x_{k+2}\}$ contains $x$ or
$y$, and $x_{k+2}$ is a rim element, it follows that
$C_{x}^{*}-e=\{x_{n-3},x_{n-2},x_{n-1}\}$.
This implies $s=n-2$, so $(x_{n-2},x_{n-1},x_{n})$ or its
reversal is in $\mcal{F}_{N}$.
Since $F_{e}$ contains $(e,x,y,z)$ as a contiguous subsequence,
it therefore contains an element of $(x_{n-2},x_{n-1},x_{n})$.
As $F_{e}$ is in a covering family of $M'$, this means that
$(x_{n-2},x_{n-1},x_{n})$ is consistent with $F_{e}$.
However, $x_{n-2}\notin \{x,y,z\}$, and
$x_{n-1}\in\{x,y\}$ as $(C_{x}^{*}-e)\cap\{x,y\}\ne \emptyset$.
Since $(x_{n-2},x_{n-1},x_{n})$ is consistent with $F_{e}$,
we are forced to the conclusion that $x=x_{n}$ and $y=x_{n-1}$.
Now $\{e,x_{n},x_{n-1}\}$ and $\{x_{n},x_{n-1},x_{n-2}\}$
are triangles in $M'$, and $\{x,y,z\}=\{x_{n},x_{n-1},z\}$
is a triad, which leads to a contradiction to \Cref{putty}.

Next we consider the case that $n\leq 5$, and
$\{x,y,z\}\cap\{x_{1},\ldots, x_{n}\}=\{x_{2},x_{4}\}$.
If $s=2$, then $(x_{2},x_{3},x_{4})$ must be consistent with $F_{e}$,
but $(e,x,y,z)$ contains $x_{2}$ and $x_{4}$ and not $x_{3}$.
If $s\ne 2$, then $s=3$, $n=5$, and $(x_{3},x_{4},x_{5})$ or
its reversal is in $\mcal{F}_{N}$.
However, $(e,x,y,z)$ contains $x_{4}$ and neither $x_{3}$
nor $x_{5}$, so it is impossible for $(x_{3},x_{4},x_{5})$ to be
consistent with $F_{e}$.
\end{proof}

Now we know that
$\{x,y,z\}\subseteq \{x_{1},\ldots,x_{n}\}$.
Applying the dual of \Cref{steal} to $\{x,y,z\}$ and
$(x_{1},\ldots,x_{n})$ in $M'\ba e$, we see that
$\{x,y,z\}=\{x_{i},x_{i+1},x_{i+2}\}$, where $x_{i}$ is a
rim element of $(x_{1},\ldots, x_{n})$, and $i\leq n-2$.

\begin{sublemma}
\label{fairy}
$1<i<n-2$.
\end{sublemma}

\begin{proof}
Since $x_{1}\notin E(N)$ and $\{x,y,z\}\subseteq E(N)$
it follows that $x_{1}\notin \{x,y,z\}$, so
$i>1$.
Assume that $i+2=n$.
Then $\{x_{n-2},x_{n-1},x_{n}\}$ is a triad in $M'$.
As $C_{x}^{*}$ is a cocircuit in $M'$ by \Cref{anise}, it follows
that $C_{x}^{*}-e\ne \{x_{n-2},x_{n-1},x_{n}\}$.
As $C_{x}^{*}$ contains at least one of
$x$ or $y$, we see that $C_{x}^{*}-e=\{x_{n-4},x_{n-3},x_{n-2}\}$.
As $(x_{s},\ldots, x_{t})$ is a fan in $M$, it follows that
$s\geq n-3$, so some three- or four-element contiguous
subsequence of $(x_{n-3},x_{n-2},x_{n-1},x_{n})$
is in $\mcal{F}_{N}$ (up to reversing).
If $\{x,y\}=\{x_{n-1},x_{n}\}$, then the triangle
$\{e,x_{n-1},x_{n}\}$ violates orthogonality with
$C_{x}^{*}$ in $M'$.
Assume that $\{x,y\}=\{x_{n-2},x_{n-1}\}$.
Now $\{x_{n-3},x_{n-2},x_{n-1}\}$ is a triangle in $M\ba e$,
and hence in $M'$.
As $\{e,x,y\}$ is also a triangle in $M'$, it follows that
$\{e,x_{n-3},x_{n-2},x_{n-1}\}$ is a $U_{2,4}$\dash restriction
in $M'$ that intersects the triad $\{x,y,z\}$.
This is a contradiction, so $\{x,y\}=\{x_{n-2},x_{n}\}$.
Therefore $(e,x,y,z)$ is $(e,x_{n-2},x_{n},x_{n-1})$ or
$(e,x_{n},x_{n-2},x_{n-1})$.
This means any three- or four-element contiguous
subsequence of $(x_{n-3},x_{n-2},x_{n-1},x_{n})$
contains elements from $(e,x,y,z)$, but cannot be
consistent with $F_{e}$, contradicting the fact that
$F_{e}$ is in a covering family of $M'$.
\end{proof}

\begin{sublemma}
\label{gnome}
$\{x,y\}=\{x_{i},x_{i+2}\}$, and therefore $z=x_{i+1}$.
\end{sublemma}

\begin{proof}
Assume that $\{x,y\}=\{x_{i},x_{i+1}\}$.
Then $\{e,x_{i-1},x_{i},x_{i+1}\}$ is a $U_{2,4}$\dash restriction
in $M'$ that intersects the triad $\{x,y,z\}$, a contradiction.
Next assume that $\{x,y\}=\{x_{i+1},x_{i+2}\}$.
Then $\{e,x_{i+1},x_{i+2}\}$ and $\{x_{i+1},x_{i+2},x_{i+3}\}$
are triangles of $M'$.
We again find that $M'$ has a $U_{2,4}$\dash restriction that
intersects a triad.
\end{proof}

Note that $F_{e}$ contains $e$ and a fan in $\mcal{F}_{e}$, so it
contains at least four elements.

\begin{sublemma}
\label{pearl}
$|F_{e}|>4$.
\end{sublemma}

\begin{proof}
If $F_{e}$ contains exactly four elements, then $(x,y,z)$
or its reversal is in $\mcal{F}_{N}$.
However, $(x,y,z)$ is $(x_{i},x_{i+2},x_{i+1})$ or
$(x_{i+2},x_{i},x_{i+1})$, implying that $(x,y,z)$
is not consistent with $(x_{1},\ldots, x_{n})$, although
$\{x,y,z\}\subseteq \{x_{1},\ldots, x_{n}\}$.
This contradicts the fact that either
$(x_{1},\ldots, x_{n})$ or
$(x_{1},\ldots, x_{n},e_{1})=(e_{m},\ldots, e_{1})$
is in the covering family $\mcal{F}_{e}$
(up to reversing).
\end{proof}

Let $(e,x,y,z,v)$ be the first five elements of $F_{e}$.
If $(y,z)=(x_{i},x_{i+1})$, then $v=x_{i-1}$, for otherwise
$\{v,x_{i-1},x_{i},x_{i+1}\}$ is a $U_{2,4}$\dash restriction
in $M'$, and $\{x_{i},x_{i+1},x_{i+2}\}$ is a triad.
A similar argument shows that if $(y,z)=(x_{i+2},x_{i+1})$,
then $v=x_{i+3}$, since otherwise 
$\{v,x_{i+1},x_{i+2},x_{i+3}\}$ is a $U_{2,4}$\dash restriction
in $M'$.

Assume that $F_{e}$ contains more than five elements.
Let $(e,x,y,z,v,w)$ be the first six elements of $F_{e}$.
Then $\{z,v,w\}$ is a triad of $M'\ba e$.
The previous paragraph shows that $\{z,v,w\}$ cannot be a
set of three consecutive elements in $(x_{1},\ldots, x_{n})$,
so the dual of \Cref{unity} shows that $w$ is not in
$\{x_{1},\ldots, x_{n}\}$.
Because $z$ and $v$ are not consecutive in $(x_{1},\ldots, x_{n})$,
it follows from \Cref{ninny} that $n\leq 5$.
Now \ref{fairy} implies that $n=5$, and $i=2$, so $z=x_{3}$.
But now the triad $\{z,v,w\}$ and the triangle
$\{x_{1},x_{2},x_{3}\}$ violate orthogonality.
Hence $F_{e}=(e,x,y,z,v)$.

Either $(e,x,y,z,v)$ is
$(e,x_{i},x_{i+2},x_{i+1},x_{i+3})$, or it is
$(e,x_{i+2},x_{i},x_{i+1},x_{i-1})$.
Now some three- or four-element subsequence of
$(x_{i},x_{i+2},x_{i+1},x_{i+3})$ or
$(x_{i+2},x_{i},x_{i+1},x_{i-1})$ is in $\mcal{F}_{N}$
(up to reversing).
But this subsequence must be equal to
$(x_{s},\ldots, x_{t})$ for some values of $s$ and $t$.
It is clear that no such subsequence exists, so
this completes the proof of \Cref{chaff}.
\end{proof}

By reversing as necessary, we can assume that
$(e_{1},\ldots,e_{m})$ is a fan in $\mcal{F}_{e}$ and
$e_{1}\notin E(N)$.
\Cref{chaff} tells us that $e_{1}$ is a rim element, and
$\{e,e_{1}\}$ is contained in a triangle, $T_{e}$, of $M$.
By \Cref{balsa}, $M\ba e/e_{1}$ is $3$\dash connected and
has $N$ as a minor.
As $|E(M)|-|E(N)|>2$, we let
$x$ be an element in $E(M\ba e/e_{1})-E(N)$.
First we assume that the fan in $\mcal{F}_{e}$ that contains $x$
is disjoint from $(e_{1},\ldots,e_{m})$.
Let $(x_{1},\ldots,x_{n})$ be this fan, where we can reverse
as necessary and assume that $x=x_{1}$.
Then $x_{1}$ is a rim element of $(x_{1},\ldots, x_{n})$ and
$\{e,x_{1}\}$ is contained in a triangle, $T_{x}$.
Note that $N$ is a minor of $M\ba e/e_{1}/x_{1}$, for
$\{x_{2},x_{3},x_{4}\}$ is a codependent triangle in
$M\ba e/e_{1}\ba x_{1}$.
This means $T_{e}\ne T_{x}$, for otherwise $e$ is a loop
in $M/e_{1}/x_{1}$, and this matroid has $N$ as a minor,
so is connected.
Let $C$ be a circuit of $M\ba e$ contained in
$(T_{e}\cup T_{x})-e$ that contains $e_{1}$.
If $|C|=3$, then $M\ba e/e_{1}$ contains a parallel pair,
contradicting its $3$\dash connectivity,
so $C=(T_{e}\cup T_{x})-e$, and therefore $x_{1}\in C$.
By orthogonality with $\{e_{1},e_{2},e_{3}\}$
we deduce that $C$ contains an internal element from
$(e_{1},\ldots, e_{m})$ (note $m\geq 4$ as
$(e_{1},\ldots, e_{m})$ contains $e_{1}$ and a fan from
$\mcal{F}_{N}$), and this element is in $E(N)$ by \Cref{jewel}.
By orthogonality with $\{x_{1},x_{2},x_{3}\}$, we see that
$C$ also contains an element from $(x_{2},\ldots, x_{n})$
that is in $E(N)$.
Thus $C-\{e_{1},x_{1}\}$ is a parallel pair in $M\ba e/e_{1}/x_{1}$,
and in $N$, a contradiction.

Now the only elements in $E(M\ba e/e_{1})-E(N)$ belong to the same
fan of $\mcal{F}_{e}$ as $e_{1}$.
By \Cref{jewel,chaff}, we can assume $(e_{1},\ldots,e_{m})$ is a fan in
$\mcal{F}_{e}$ such that $e_{1}$ and $e_{m}$ are rim elements
not in $E(N)$, and that $N=M\ba e/e_{1}/e_{m}$.
Let $T_{1}$ and $T_{m}$ be triangles of $M$ such that
$\{e,e_{1}\}\subseteq T_{1}$ and $\{e,e_{m}\}\subseteq T_{m}$.
Note that $T_{1}\ne T_{m}$, as
$M/e_{1}/e_{m}$ has $N$ as a minor, and is therefore connected.
Let $C$ be a circuit of $M\ba e$ contained in
$(T_{1}\cup T_{m})-e$ that contains $e_{1}$.
Then $|C|=4$, or else $M\ba e/e_{1}$ contains a parallel pair.
Hence $e_{m}\in C$.
Therefore $C-\{e_{1},e_{m}\}$ is a parallel pair in
$M\ba e/e_{1}/e_{m}=N$, a contradiction.
Now we have completed the proof of \Cref{arrow}.

\section{Acknowledgements}

We thank Geoff Whittle for valuable discussions
and the referee for a careful reading.


\end{document}